\documentclass[11pt]{article}
\usepackage[margin=1in]{geometry}
\usepackage[T1]{fontenc}
\usepackage[utf8]{inputenc}
\usepackage{lmodern}
\usepackage{amsfonts,amsmath,amssymb,amsthm}
\usepackage{mathtools,bm,mathrsfs,mleftright}
\usepackage{xcolor,enumerate}
\usepackage{graphicx}
\usepackage[numbers,sort&compress]{natbib}
\usepackage{prettyref}
\usepackage{algorithm,algpseudocode}
\usepackage[colorlinks=true,linkcolor=blue,citecolor=blue,urlcolor=blue]{hyperref}
\allowdisplaybreaks[4]

\numberwithin{equation}{section}
\theoremstyle{plain}

\newtheorem{theorem}{Theorem}[section]
\newtheorem{lemma}[theorem]{Lemma}
\newtheorem{proposition}[theorem]{Proposition}
\newtheorem{corollary}[theorem]{Corollary}
\theoremstyle{definition}
\newtheorem{definition}[theorem]{Definition}
\newtheorem{remark}[theorem]{Remark}

\newrefformat{eq}{\textup{(\ref{#1})}}
\newrefformat{lem}{Lemma~\textup{\ref{#1}}}
\newrefformat{thm}{Theorem~\textup{\ref{#1}}}
\newrefformat{cor}{Corollary~\textup{\ref{#1}}}
\newrefformat{prop}{Proposition~\textup{\ref{#1}}}
\newrefformat{assu}{Assumption~\textup{\ref{#1}}}
\newrefformat{cond}{Condition~\textup{\ref{#1}}}
\newrefformat{remark}{Remark~\textup{\ref{#1}}}
\newrefformat{example}{Example~\textup{\ref{#1}}}
\newrefformat{claim}{Claim~\textup{\ref{#1}}}
\newrefformat{def}{Definition~\textup{\ref{#1}}}
\newrefformat{alg}{Algorithm~\textup{\ref{#1}}}
\newrefformat{sec}{Section~\textup{\ref{#1}}}
\newrefformat{subsec}{Subsection~\textup{\ref{#1}}}
\newrefformat{subsubsec}{Subsection~\textup{\ref{#1}}}

\newcommand{\ep}{\varepsilon}
\newcommand{\R}{\mathbb{R}}
\newcommand{\E}{\mathbb{E}}
\newcommand{\T}{\top}

\newcommand{\hf}{\frac{1}{2}}

\newcommand{\bbS}{\mathbb{S}}

\newcommand{\bbP}{\mathbb{P}}

\DeclareMathOperator*{\argmin}{arg\,min}

\DeclareMathOperator{\Tr}{Tr}

\newcommand{\mr}{\mathrm}
\newcommand{\mf}{\mathsf}
\newcommand{\mca}{\mathcal}

\newcommand{\xk}[1]{\left(#1\right)}
\newcommand{\zk}[1]{\left[#1\right]}
\newcommand{\dk}[1]{\left\{#1\right\}}
\newcommand{\xkx}[1]{(#1)}

\newcommand{\dkx}[1]{\{#1\}}
\providecommand{\ceil}[1]{\left\lceil{#1}\right\rceil}
\providecommand{\ang}[1]{\left\langle{#1}\right\rangle}

\providecommand{\abs}[1]{\left\lvert{#1}\right\rvert}
\providecommand{\norm}[1]{\left\lVert{#1}\right\rVert}

\providecommand{\normx}[1]{\lVert{#1}\rVert}
\providecommand{\dd}{\,\mathrm{d}}

\newcommand{\dkm}[1]{\mleft\{#1\mright\}}
\newcommand{\normsch}[2][q]{\norm{#2}_{\mca{S}_{#1}}}

\newcommand{\ind}[1]{\bm{1}\dkm{#1}}

\newcommand{\caD}{\mathcal{D}}

\providecommand{\cref}{\prettyref}

\newtheorem{assumption}[theorem]{Assumption}

\newrefformat{para}{Paragraph~\textup{\ref{#1}}}

\newrefformat{fig}{Figure~\textup{\ref{#1}}}
\newrefformat{tab}{Table~\textup{\ref{#1}}}

\title{Minimax and Adaptive Covariance Matrix Estimation under Differential Privacy}
\author{%
  T. Tony Cai\\
  Department of Statistics and Data Science\\
  The Wharton School, University of Pennsylvania\\
  \texttt{tcai@wharton.upenn.edu}
  \and
  Yicheng Li\\
  KLATASDS-MOE, School of Statistics\\
  East China Normal University\\
  \texttt{ycli@sfs.ecnu.edu.cn}
}
\date{}

\begin{document}
\maketitle

\begin{abstract}
Estimating covariance matrices is fundamental to a wide range of statistical applications.
This paper studies minimax and adaptive estimation of high-dimensional covariance matrices under \(\rho\)-zero-concentrated differential privacy (\(\rho\)-zCDP) over three nested classes: the pointwise-decay class \(\mca{H}_\alpha\), the row-tail class \(\mca{G}_\alpha\), and the separated-block class \(\mca{F}_\alpha\). We consider both squared operator norm loss and normalized squared Frobenius norm loss.

For \(\mca{H}_\alpha\) and \(\mca{G}_\alpha\), we develop multiscale block estimators tailored to the refined geometry of the two classes, while for \(\mca{F}_\alpha\), we develop a blockwise tridiagonal estimator. The resulting minimax-optimal rates reveal a nontrivial interplay among the smoothness \(\alpha\), the loss, the geometry of the covariance class, and the privacy constraint. In contrast to the non-private setting, privacy distinguishes covariance classes that share the same leading non-private rate and induces a polynomial dependence on the ambient dimension.

We further develop procedures that adapt to the unknown decay parameter over all three covariance classes under both losses, at the cost of at most polylogarithmic factors.
To establish minimax lower bounds, we develop a novel differentially private van Trees inequality that connects Fisher information with the \(\rho\)-zCDP constraint and may be useful for other private estimation problems.
We also construct carefully designed prior distributions to obtain matching minimax lower bounds.
\end{abstract}




\section{Introduction}

Modern statistical analysis increasingly relies on individual-level data that are both high dimensional and highly sensitive.
In applications such as biomedical research, genomics, electronic health records, finance, social science, and federated data analysis, each observation may contain detailed information about a person or institution.
Covariance estimation is often a first step in analyzing such data, because covariance structure underlies principal component analysis, discriminant analysis, clustering, regression, graphical modeling, and many other multivariate methods.
At the same time, releasing or using covariance estimates without privacy protection can reveal sensitive information about individuals. Improper disclosure can lead to harms such as identity theft, discrimination, loss of confidentiality, or loss of public trust. These risks are particularly acute in healthcare and biomedical studies, where genetic and medical records enable important scientific advances but are highly sensitive \citep{hawes2020_ImplementingDifferential,adnan2022_FederatedLearning}.
Thus, modern covariance estimation must balance statistical accuracy with rigorous privacy guarantees.

Differential privacy (DP) \citep{dwork2006_CalibratingNoise,dwork2014_AlgorithmicFoundations} has emerged as a leading mathematical framework for privacy-preserving data analysis.
It provides a guarantee that the inclusion or exclusion of any single individual's data has only a limited effect on the output of an analysis, thereby protecting against inference attacks even when an adversary has auxiliary information.
This robustness makes DP especially attractive for statistical procedures intended for sensitive data.
However, privacy is not free.
A private estimator must inject sufficient randomness, or otherwise limit the information released, in order to protect individuals. This added randomization can substantially degrade statistical accuracy. A central question in private statistics, therefore, is to characterize the optimal privacy--utility trade-off: how much accuracy must be sacrificed to guarantee privacy, and how should estimators be designed to achieve the best possible balance?

This paper studies that question for high-dimensional covariance matrix estimation under differential privacy.
Suppose we observe independent and identically distributed random vectors
\(x_1,\dots,x_n \in \R^d\) with population covariance matrix \(\Sigma\).
In classical low-dimensional settings, the sample covariance matrix performs well.
Modern datasets, however, often operate in the high-dimensional regime where \(d\) is comparable to, or much larger than, \(n\).
In this regime the sample covariance matrix is unstable or singular, and structural assumptions are needed for accurate estimation.
A widely studied structural assumption is that of \emph{bandable} covariance matrices, where off-diagonal entries decay as they move away from the diagonal.
Such matrices arise naturally in temporal, spatial, longitudinal, and other ordered data settings, reflecting the weakening of correlations with increasing separation.
Optimal rates and adaptive procedures for estimating bandable covariance matrices in the non-private setting have been established under both the operator and Frobenius norms.
In particular, \citet{cai2010_OptimalRates} derived the minimax rate of convergence under the operator norm and proposed an optimal-rate tapering estimator.
Later, \citet{cai2012_AdaptiveCovariance} introduced a block-thresholding estimator that adaptively attains the optimal rate over a collection of parameter spaces without requiring knowledge of the decay parameter.
See \citet{cai2016_EstimatingStructured} for a comprehensive survey of structured covariance estimation.

The interaction between bandable structure and differential privacy is subtle.
In non-private covariance estimation, bandability reduces the effective complexity of the parameter space and leads to estimators that outperform the unstructured sample covariance matrix.
Under privacy constraints, however, the estimator must also control the sensitivity of the released object.
The geometry of the covariance class can therefore affect not only the statistical bias--variance trade-off, but also the amount and placement of privacy noise.
Consequently, it is not enough to privatize a non-private estimator in a black-box fashion.
One must understand how the structural decay of the covariance matrix, the loss function, and the privacy constraint jointly determine the minimax risk.

We work primarily with the \(\rho\)-zero-concentrated differential privacy (\(\rho\)-zCDP) framework
\citep{dwork2016_ConcentratedDifferential,bun2016_ConcentratedDifferential}.
While \((\epsilon,\delta)\)-DP remains the most common privacy formulation, zCDP is based on R\'enyi divergence and provides a single-parameter privacy budget with tight composition properties.
This is particularly useful for the multiscale and blockwise procedures developed in this paper.
Recall first that a randomized algorithm \(M\) satisfies \((\epsilon,\delta)\)-DP if, for any adjacent datasets \(S,S'\in\mathcal{D}\) differing by one observation and any measurable output set \(A\),
\begin{equation*}
  \bbP\dk{M(S) \in A} \leq e^\epsilon \bbP\dk{M(S') \in A} + \delta.
\end{equation*}
For \(\alpha>1\), the \(\alpha\)-R\'enyi divergence between random variables \(X\) and \(Y\), associated with probability measures \(P\) and \(Q\), is defined as
\[
D_{\alpha}(X\|Y)
=
\frac{1}{\alpha-1}
\log \E \left( \frac{\dd P}{\dd Q}(X) \right)^{\alpha}.
\]
The definition of \(\rho\)-zCDP is as follows.

\begin{definition}[$\rho$-zCDP]
  \label{def:rho-zCDP}
  A randomized algorithm \(M : \caD \to \mca{A}\) satisfies \(\rho\)-zero-concentrated DP (\(\rho\)-zCDP) if for all adjacent \(S,S' \in \mathcal{D}\),
  \begin{equation}
    D_{\alpha}(M(S) \| M(S')) \leq \rho \alpha,\quad \forall \alpha \in (1, \infty),
  \end{equation}
  where \(D_{\alpha}(M(S)\|M(S'))\) is taken only over the randomness of \(M\) conditioned on the data.
\end{definition}

The connection between \(\rho\)-zCDP and \((\epsilon,\delta)\)-DP is well established
\citep{dwork2014_AlgorithmicFoundations,bun2016_ConcentratedDifferential}.

\begin{lemma}
  \label{lem:DPMechanism_Relation}
  If \( M \) is \((\epsilon,0)\)-DP, then \( M \) is \((\epsilon^2/2)\)-zCDP.
  Conversely, if \( M \) is \(\rho\)-zCDP, then \( M \) is
  \(\left( \rho + 2 \sqrt{\rho \log (1/\delta)}, \delta \right)\)-DP for every \(\delta \in (0,1)\).
  In particular, for \(\epsilon \in (0,1]\) and \(\delta \in (0,e^{-1}]\), \(\rho\)-zCDP with
  \(\rho \leq \epsilon^2 (8 \log (1/\delta))^{-1}\) implies \((\epsilon,\delta)\)-DP.
\end{lemma}

Although \((\epsilon,\delta)\)-DP is widely used, the divergence-based formulation of zCDP often simplifies the analysis of complex algorithms and yields sharper composition bounds.
For this reason, we adopt \(\rho\)-zCDP as the primary privacy notion in this paper; the corresponding \((\epsilon,\delta)\)-DP guarantees follow from the lemma above.

In this paper, we develop minimax and adaptation theory for estimating bandable covariance matrices under differential privacy. We consider three nested covariance classes: the pointwise-decay class \(\mca{H}_\alpha\), the row-tail class \(\mca{G}_\alpha\), and the separated-block class \(\mca{F}_\alpha\), where \(\alpha\) denotes the smoothness parameter. The precise definitions of these classes are provided in Section \ref{subsec:three-classes}. They satisfy the inclusion \(\mca{H}_\alpha \subset \mca{G}_\alpha \subset \mca{F}_\alpha\), up to a constant factor in the class radius.

\subsection{Main Contribution}
\label{subsec:Contribution}

We first characterize the minimax risks under squared operator loss and normalized squared Frobenius loss over all three classes. The principal finding is a geometric separation created by privacy. Under operator loss, the classes have the same leading non-private rate, but the privacy cost over the separated-block class \(\mca{F}_\alpha\) is algebraically larger than the costs over the pointwise-decay and row-tail classes. Under Frobenius loss, the non-private rates already separate \(\mca{F}_\alpha\) from \(\mca{G}_\alpha\) and \(\mca{H}_\alpha\), and privacy creates a further separation between the separated-block class and the two smaller classes. Thus privacy depends not only on the decay exponent \(\alpha\), but also on the local geometry through which that decay is imposed. This distinction is invisible in the leading non-private operator rate and is one of the central phenomena identified by the paper.

\begin{table}[H]
  \caption{
    Leading algebraic terms under squared operator and normalized squared Frobenius losses in the high-dimensional regime.
    Logarithmic factors suppressed.
    Denote \(x=d/(\rho n^2)\).
    The marker ``(log)'' indicates a fixed polylogarithmic gap between the current upper and lower bounds.
  }
  \label{tab:Intro_Rates}
  \centering
  \small
  \setlength{\tabcolsep}{12pt}
  \begin{tabular}{@{}lcc|cc@{}}
    \hline
    & \multicolumn{2}{c|}{Non-private rate}
    & \multicolumn{2}{c}{Cost of privacy}\\
    \cline{2-5}
    Class
    & Operator
    & Frobenius
    & Operator
    & Frobenius\\
    \hline
    \(\mca{F}_\alpha\)
    & \(n^{-\frac{2\alpha}{2\alpha+1}}\)
    & \(n^{-\frac{2\alpha}{2\alpha+1}}\)
    & \(x^{\frac{\alpha}{\alpha+1}}\)
    & \(x^{\frac{\alpha}{\alpha+1}}\)
    \\[1mm]
    \(\mca{G}_\alpha\)
    & \(n^{-\frac{2\alpha}{2\alpha+1}}\)
    & \(n^{-\frac{2\alpha+1}{2\alpha+2}}\)
    & \(\displaystyle
        x^{\frac{2\alpha}{2\alpha+1}}
        \vee
        x^{\frac{2\alpha+1}{2\alpha+3}}
      \)  (log)
    & \(x^{\frac{2\alpha+1}{2\alpha+3}}\)  (log)
    \\[1mm]
    \(\mca{H}_\alpha\)
    & \(n^{-\frac{2\alpha}{2\alpha+1}}\)
    & \(n^{-\frac{2\alpha+1}{2\alpha+2}}\)
    & \(\displaystyle
        x^{\frac{2\alpha}{2\alpha+1}}
        \vee
        x^{\frac{2\alpha+1}{2\alpha+3}}
      \)  (log)
    & \(x^{\frac{2\alpha+1}{2\alpha+3}}\)
    \\
    \hline
  \end{tabular}
\end{table}

Moreover, the role of covariance geometry depends on the choice of loss. In the non-private setting, the operator norm does not distinguish among the three classes in terms of rate of convergence, whereas the Frobenius norm separates \(\mca{F}_\alpha\) from the two smaller classes \(\mca{G}_\alpha\) and \(\mca{H}_\alpha\). Under privacy, both losses distinguish \(\mca{F}_\alpha\) from \(\mca{G}_\alpha\) and \(\mca{H}_\alpha\), but they do so in qualitatively different ways. For \(\mca{G}_\alpha\) and \(\mca{H}_\alpha\), the privacy cost under the operator norm exhibits a phase transition: the first operator term in Table~\ref{tab:Intro_Rates} dominates when \(\alpha < 1/2\), the second dominates when \(\alpha > 1/2\), and the two terms coincide at \(\alpha = 1/2\). No such transition occurs under the Frobenius loss, where the privacy cost follows a single algebraic regime.

We then construct private estimators tailored to the geometry of each class.
For \(\mca{F}_\alpha\), we introduce a blockwise tridiagonal estimator to balance the statistical cost, privacy cost and truncation bias.
For \(\mca{G}_\alpha\) and \(\mca{H}_\alpha\), the useful local constraints occur at every distance from the diagonal, and a single block scale does not directly expose this multiscale structure while keeping the privacy noise under control.
We therefore organize the off-diagonal entries into dyadic increments and refine them into square blocks that separate the distance scales and support procedures tailored to each class.
This construction yields dyadic profile and nonlinear peeling estimators for \(\mca{G}_\alpha\), and a dyadic greedy procedure for \(\mca{H}_\alpha\).
More specifically, pointwise decay makes each dyadic outer block small in Frobenius norm, which enables a greedy rank-one release under operator loss.
The row-tail condition instead provides a local row--column \(\ell^1\) constraint, leading to profile blocks under operator loss and peeling blocks under Frobenius loss.
For \(\mca{H}_\alpha\) under Frobenius loss, the blockwise tridiagonal construction already exploits the available entrywise decay.
These constructions attain the corresponding minimax rates up to logarithmic factors in certain cases.
The methodological novelty is therefore not a single privatized tapering estimator, but a collection of releases tailored to the local geometry of each class.

Another key contribution is the DP van Trees inequality in \cref{thm:VanTrees__zCDP}.
Under the stated regularity conditions, we establish a contraction of Fisher information implied by zCDP and combine it with the classical Bayesian Cramér--Rao inequality. Let \(I_x(\theta)\) denote the Fisher information in one observation, let \(J_\pi\) denote the Fisher information of the prior \(\pi\), and define \(C_\rho=(e^{2\rho}-1)/\rho\). The denominator in the resulting Bayes risk bound contains
\[
  \min\left\{
    n\int\Tr I_x(\theta)\,\dd\pi(\theta),\,
    C_\rho\rho n^2\int\norm{I_x(\theta)}\,\dd\pi(\theta)
  \right\}
  +\Tr J_\pi.
\]
The first term inside the minimum is the classical contribution from the sample information, governed by the trace of the Fisher information, whereas the second is the information limit imposed by privacy, governed by its operator norm. This distinction between trace and operator norm yields the correct polynomial dependence on the dimension in the private lower bounds without introducing an additional logarithmic loss.

Using the DP van Trees inequality, we derive minimax lower bounds for all three covariance classes.
In particular, we construct a blockwise diagonal prior with random blocks whose operator norms are controlled for \(\mca{F}_\alpha\);
a banded entrywise prior for the Frobenius lower bound over \(\mca{H}_\alpha\);
and block priors of intermediate rank that lie simultaneously in \(\mca{G}_\alpha\) and \(\mca{H}_\alpha\) for their common operator norm lower bound.
These constructions isolate the different geometric obstructions created by the separated-block, row-tail, and pointwise-decay constraints.
Combined with the upper bounds, these results show precisely which distinctions among the three classes are induced by privacy.

The inequality is particularly well suited to structured matrix problems because, unlike approaches that rely on priors with independent coordinates or specially constructed testing packings, it accommodates dependent matrix-valued priors. Because it requires only suitable control of the likelihood Fisher information and the prior Fisher information, the lower-bound method may also be of independent interest for other high-dimensional estimation problems under differential privacy.

Finally, for each specified covariance class and loss, we present adaptive procedures that operate without prior knowledge of the smoothness \(\alpha\). For the blockwise tridiagonal estimator, adaptation is achieved by thresholding dyadic increments; for the multiscale block estimators, candidate local radii are chosen via private estimates using Lepski's principle.  Adaptation comes at no extra algebraic cost, as the adaptive rates preserve the exact scaling in \(n\), \(d\), and \(\rho\) up to a fixed polylogarithmic overhead.

\subsection{Related Work}

Existing work on private covariance estimation has largely focused on unstructured matrices. Under appropriate conditions, calibrated Gaussian perturbation of the sample covariance attains the benchmark under squared operator loss
\[
\frac{d}{n}+\frac{d^3}{\rho n^2}
\]\citep{amin2019_DifferentiallyPrivate,dong2022_DifferentiallyPrivate,portella2024_LowerBounds,narayanan2024_BetterSimpler}. Related work studies Gaussian covariance estimation under Mahalanobis loss, including ill-conditioned models\citep{biswas2020_CoinpressPractical,kamath2022_PrivateComputationallyEfficient,ashtiani2022_PrivatePolynomial,alabi2023_PrivatelyEstimating}, as well as private estimation of spiked covariance matrices \citep{cai2024_OptimalDifferentially}. These benchmarks do not exploit ordered off-diagonal decay and can therefore be substantially suboptimal for the structured covariance classes considered here. This paper quantifies the resulting minimax improvement and shows that it depends on both the precise form of the decay constraint and the loss function.

In contrast, there is a rich literature on structured covariance estimation in the non-private setting.
For bandable covariance matrices, regularization methods including banding~\citep{bickel2008_RegularizedEstimation}, tapering~\citep{cai2010_OptimalRates}, thresholding~\citep{bickel2008_CovarianceRegularization},
and adaptive block-thresholding~\citep{cai2012_AdaptiveCovariance} have been proposed to exploit the structural decay of correlations.
Optimal rates of convergence under both operator and Frobenius norms have been established~\citep{cai2010_OptimalRates}.
Other commonly studied structures include sparse covariance matrices~\citep{cai2011_AdaptiveThresholding,cai2012_OptimalRates},
Toeplitz covariance matrices~\citep{cai2013_OptimalRates}, and sparse precision matrices~\citep{cai2011_ConstrainedL1}.
See \citet{cai2016_EstimatingStructured} for a comprehensive overview of this literature.

Privacy constraints typically introduce a ``cost of privacy'' in minimax rates.
Several techniques have been developed for proving minimax lower bounds under differential privacy. These include private analogues of classical information-theoretic tools such as Fano's inequality and Le Cam's method\citep{duchi2014_LocalPrivacy,duchi2018_MinimaxOptimal,acharya2020_DifferentiallyPrivate}, fingerprinting arguments\citep{bassily2014_DifferentiallyPrivate,kamath2022_NewLower,peter2024_SmoothLower}, score attacks\citep{cai2021_CostPrivacy,cai2023_ScoreAttack}, and van Trees inequalities\citep{cai2024_OptimalFederated,xue2024_OptimalEstimation}. More recently, \citet{narayanan2024_BetterSimpler,portella2024_LowerBounds} used Wishart constructions to establish minimax lower bounds for private estimation of general unstructured covariance matrices.

Within this literature, the differentially private van Trees inequality
developed here is especially useful because its trace--operator separation
yields the correct dimension dependence for the structured covariance lower
bounds; its methodological contribution is described in Section \ref{subsec:Contribution}  and
developed in \cref{sec:LowerBound}.

%

\subsection{Organization}

The rest of the paper is organized as follows. We conclude this section with basic notation.
\cref{sec:ModelResults} introduces the statistical model and the three covariance classes and states the main minimax results with discussions.
\cref{sec:UpperBounds} develops the private estimators: a blockwise tridiagonal estimator for the separated-block class, a greedy multiscale block procedure for the pointwise-decay class under operator loss, and profile and peeling procedures for the row-tail class under operator and Frobenius losses, respectively.
The pointwise-decay class under Frobenius loss is handled by the blockwise tridiagonal construction.
\cref{sec:LowerBound} presents the DP van Trees inequality and derives minimax lower bounds for bandable covariance matrix estimation under DP.
In \cref{sec:Adaptive}, we develop estimators adaptive to smoothness and analyze their performance under both losses.
\cref{sec:Precision} considers extensions of our methodology to precision matrix estimation and \cref{sec:Conclusion} concludes with a discussion of our findings and directions for future research.

\subsection{Notation}\label{subsec:Notations}

We use $C$ and $c$ to denote generic positive constants that may vary from line to line.
Write $a \lesssim b$ if there exists a constant $C>0$ such that $a \le Cb$, and write $a \asymp b$ if $a \lesssim b$ and $b \lesssim a$.
Write \(a\lesssim_{\log} b\) if \(a\leq C\log^A(edn/\rho)b\) for fixed \(C,A>0\), and write \(a\asymp_{\log} b\) if \(a\lesssim_{\log} b\) and \(b\lesssim_{\log} a\).
We use the notations $a \wedge b$ and $a \vee b$ to denote $\min(a,b)$ and $\max(a,b)$, respectively.
The indicator function is denoted by $\ind{\cdot}$, and $|S|$ denotes the cardinality of a set $S$.
We write $[d] = \{1,2,\dots,d\}$.
For \(R>0\), define the radial clipping map on \(\R^m\) by \(\chi_R(0)=0\) and \(\chi_R(z)=z\min\{1,R/\norm{z}_2\}\) for \(z\neq0\).

For a matrix $A$, $\norm{A}$ denotes the spectral (operator) norm and $\norm{A}_F$ the Frobenius norm.
We write $\ang{A,B} = \Tr(A^\T B)$ for the trace inner product.
The matrix $\ell^r$ norm is defined as
\begin{math}
  \norm{A}_{\ell^r} = \max_{\norm{x}_{\ell^r}=1} \norm{Ax}_{\ell^r}.
\end{math}
In particular,
$\norm{A}_{\ell^2} = \norm{A}$ is the spectral norm,
$\norm{A}_{\ell^\infty} = \max_i \sum_j |A_{ij}|$ is the maximum absolute row sum,
and $\norm{A}_{\ell^1} = \max_j \sum_i |A_{ij}|$ is the maximum absolute column sum.
The matrix Schatten-$q$ norm is defined as
\begin{math}
  \normsch{A} = \xk{\sum_i \sigma_i(A)^q}^{1/q},
\end{math}
where $\sigma_i(A)$ are the singular values of $A$.
In particular, $\normsch[2]{A} = \norm{A}_F$ is the Frobenius norm and $\normsch[\infty]{A} = \norm{A}$ is the spectral norm.
Write \( \Pi_{M_0}(A) \) for the spectral projection of a symmetric matrix to \( [0,M_0] \).
For a matrix $M \in \R^{d\times d}$ and an index set $B \subseteq [d]^2$,
we denote by $M_B$ or $M[B]$ the matrix obtained by setting all entries outside $B$ to zero.
A similar notation $v_I$ applies to vectors.
We call $B$ a block if $B = I \times J$ for some $I,J \subseteq [d]$,
and say that it has size $k$ if $|I| = |J| = k$.

The sub-Gaussian norm~\citep{vershynin2018_HighdimensionalProbability} for a random variable $\xi$ is defined as
\begin{math}
  \norm{\xi}_{\psi_2} = \inf \{ t > 0 : \E (e^{\xi^2/t^2}) \le 2 \},
\end{math}
and for a random vector $X$ as
\begin{math}
  \norm{X}_{\psi_2} \coloneqq \sup_{v:\norm{v}_2=1} \norm{\ang{X,v}}_{\psi_2}.
\end{math}


\section{Model, Covariance Classes and Results}
\label{sec:ModelResults}

\subsection{Statistical model}

Let \(X_1,\ldots,X_n\) be independent copies of a random vector \(X\in\R^d\), with mean \(\mu=\E X\) and covariance matrix
\[
   \Sigma=\E\zk{(X-\mu)(X-\mu)^\T}.
\]
For fixed constants \(K,M_0<\infty\), we assume
\begin{equation}
  \norm{X-\mu}_{\psi_2} \leq K,
  \qquad
  0\preceq\Sigma\preceq M_0 I_d.
  \label{eq:ModelAssumption}
\end{equation}

We consider the pairing symmetrization \( W_i=(X_{2i}-X_{2i-1})/\sqrt{2}, \) for \( 1\leq i\leq \lfloor n/2\rfloor \).
Then, the paired observations are independent and satisfy \(\E W_i=0\), \(\operatorname{Cov}(W_i)=\Sigma\) and \(\norm{W_i}_{\psi_2} \leq CK\).
Moreover, replacing one original record changes at most one \(W_i\).
Consequently, composing any \(\rho\)-zCDP mechanism for the paired data recovers a \(\rho\)-zCDP mechanism for the original data.
Therefore, since our focus is the optimal rates, we assume that \( \mu = 0 \) without loss of generality throughout the paper.

In this paper, for an estimator \( \hat{\Sigma} \), we consider the squared operator norm loss
\( \normx{\hat{\Sigma} - \Sigma}^2 \)
and the normalized squared Frobenius norm loss
\( d^{-1} \normx{\hat{\Sigma} - \Sigma}_F^2 \).
The normalization makes the two losses comparable due to the relation
\( d^{-1} \norm{A}_F^2 \leq\norm{A}^2 \).

Throughout,
we consider the multidimensional regime \( d \gtrsim \log n \).
We focus on the nondegenerate privacy regime \(d/(\rho n^2)=o(1)\);
otherwise, the minimax lower bounds show that the rates degenerate.
In terms of privacy, we consider the meaningfully private case \(0<\rho\lesssim1\), where the implicit upper bound is fixed and the constants below may depend on it.
Hence, \(d=o(n^2)\).
All asymptotic statements below are understood along sequences satisfying these conditions.

\subsection{Three covariance classes}
\label{subsec:three-classes}
We consider the following classes of bandable covariance matrices.
The first two are standard in the non-private covariance estimation literature
\citep{bickel2008_CovarianceRegularization,bickel2008_RegularizedEstimation,
cai2010_OptimalRates,cai2012_AdaptiveCovariance}.
The third class is introduced in this paper and turns out to possess especially favorable geometric properties for both non-private and private estimation.
The first class imposes a power-law decay condition on the covariance entries as their distance from the main diagonal increases.
\begin{definition}[Pointwise-decay class]
  \label{def:Class_H}
  For \(\alpha,M_0,M>0\), let \(\mca{H}_\alpha(M_0,M)\) be the set of covariance matrices satisfying \(0\preceq\Sigma\preceq M_0 I_d\) and
  \begin{equation}
    \abs{\Sigma_{ij}}
    \leq M\abs{i-j}^{-(\alpha+1)},
    \qquad i\neq j.
    \label{eq:Cov_Bandable_PointwiseDecay}
  \end{equation}
\end{definition}

Another closely related class controls the tail sum of each row~\citep{cai2010_OptimalRates,cai2012_AdaptiveCovariance}:

\begin{definition}[Row-tail class]
  \label{def:Class_G}
  For \(\alpha,M_0,M>0\), let \(\mca{G}_\alpha(M_0,M)\) be the set of covariance matrices satisfying \(0\preceq\Sigma\preceq M_0 I_d\) and
  \begin{equation}
    \max_{j\in[d]}
    \sum_{i:\,\abs{i-j}>k}\abs{\Sigma_{ij}}
    \leq Mk^{-\alpha},
    \qquad k\geq1.
    \label{eq:Cov_Bandable_RowTail}
  \end{equation}
\end{definition}

These bandable covariance classes encode the decay of covariance away from the main diagonal.
Another natural way to express this structure is to control blocks separated from the diagonal.
We call a block \(R\) \(k\)-off-diagonal if it is contained in \(\{(i,j):j\geq i+k\}\) or, symmetrically, in \(\{(i,j):i\geq j+k\}\).
For example, \([1,\dots,i_0]\times[i_0+k,\dots,d]\) is a \(k\)-off-diagonal block.
This motivates the following definition.

\begin{definition}[Separated-block class]
  \label{def:Class_F}
  For \(\alpha,M_0,M>0\), let \(\mca{F}_\alpha(M_0,M)\) be the set of covariance matrices satisfying \(0\preceq\Sigma\preceq M_0 I_d\) and
  \begin{equation}
    \norm{\Sigma_R}
    \leq Mk^{-\alpha},
    \qquad
    \text{for every }k\geq1\text{ and every }k\text{-off-diagonal block }R.
    \label{eq:Cov_Bandable_Norm}
  \end{equation}
\end{definition}

Alternatively, the class can be described in terms of the operator norm of covariances between two $k$-separated
index sets $I,J\subseteq[d]$, i.e.\  $\min_{i\in I,j\in J}|i-j|\ge k$:
\begin{equation*}
  \sup_{k\text{-separated } I,J}
  \norm{\mr{Cov}(x_I, x_J)} \leq C_1 k^{-\alpha}.
\end{equation*}
From this perspective, the class $\mca{F}_\alpha$ captures the decay of long-range dependencies in the random vector $x$.
Thus, it is a natural model for various applications where such decay occurs, e.g., time series analysis and spatial statistics.

Although formulated at different levels, the three classes are closely related.
Combining $\norm{A}^2 \leq {\norm{A}_{\ell^1} \norm{A}_{\ell^\infty}}$ with tail summation yields the following inclusion relation, up to a constant adjustment of the class radius.

\begin{proposition}[Class inclusions]
  \label{prop:ClassInclusions}
  For every fixed \(\alpha,M_0,M>0\), there is a constant \(C_\alpha>0\), depending only on \(\alpha\), such that
  \[
    \mca{H}_\alpha(M_0,M)
    \subseteq
    \mca{G}_\alpha(M_0,C_\alpha M)
    \subseteq
    \mca{F}_\alpha \bigl(M_0,C_\alpha(M+M_0)\bigr).
  \]
  Moreover, we have reverse inclusions \( \mca{F}_{\alpha+1/2}\subseteq\mca{G}_{\alpha} \) and \( \mca{F}_{\alpha+1}\subseteq\mca{H}_{\alpha} \).
\end{proposition}

\subsection{Results overview}

For \(\mca{J}\in\{\mca{F},\mca{G},\mca{H}\}\), we write the class of distributions
\[
  \mca{P}_{\alpha}^{\mca{J}}(K,M_0,M)
  =
  \left\{
  P:
  P\text{ satisfies \cref{eq:ModelAssumption} and }
  \Sigma(P)\in\mca{J}_\alpha(M_0,M)
  \right\}.
\]
When \(K,M_0,M\) are fixed, we suppress them and write \(\mca{P}_\alpha^{\mca{J}}\) for readability.
Let \(\mca{M}_\rho\) be the collection of all \(\rho\)-zCDP estimators based on \(n\) observations.
To uniformly characterize the optimal rates for the classes \( \mca{G}_\alpha \) and  \( \mca{H}_\alpha \), define
\begin{equation}
  \mca{Q}_{\alpha,d}(u)
  =
  \sup_{\substack{1\leq k\leq\lfloor d/4\rfloor\\1\leq r\leq k}}
  \min\xk{
    \frac{k^{-2\alpha}}{r},
    ukr
  },
  \label{eq:Rate_Q}
\end{equation}
where \(u>0\) and \(k,r\) are integers.
Under \(d/(\rho n^2)=o(1)\), direct optimization in \cref{eq:Rate_Q} gives the following explicit form.
If \(\alpha\geq1/2\), then
\begin{equation}
  \mca{Q}_{\alpha,d} \left(\frac{d}{\rho n^2}\right)
  \asymp_\alpha
  \begin{cases}
    \displaystyle\frac{d^3}{\rho n^2},
    &
    d\lesssim(\rho n^2)^{1/(2\alpha+4)},
    \\[4pt]
    \displaystyle
    \left(\frac{d}{\rho n^2}\right)^{\frac{2\alpha+1}{2\alpha+3}},
    &
    d\gtrsim(\rho n^2)^{1/(2\alpha+4)}.
  \end{cases}
  \label{eq:Rate_Q_alpha_ge_half}
\end{equation}
If \(0<\alpha<1/2\), then
\begin{equation}
  \mca{Q}_{\alpha,d} \left(\frac{d}{\rho n^2}\right)
  \asymp_\alpha
  \begin{cases}
    \displaystyle\frac{d^3}{\rho n^2},
    &
    d\lesssim(\rho n^2)^{1/(2\alpha+4)},
    \\[4pt]
    \displaystyle\frac{d^{1-\alpha}}{\sqrt{\rho}\,n},
    &
    (\rho n^2)^{1/(2\alpha+4)}
    \lesssim d
    \lesssim(\rho n^2)^{1/(2\alpha+2)},
    \\[6pt]
    \displaystyle
    \left(\frac{d}{\rho n^2}\right)^{\frac{2\alpha}{2\alpha+1}},
    &
    d\gtrsim(\rho n^2)^{1/(2\alpha+2)}.
  \end{cases}
  \label{eq:Rate_Q_alpha_lt_half}
\end{equation}

The first result gives matching minimax rates over the separated-block class under both losses.

\begin{theorem}[Optimal rates under \( \mca{F}_\alpha \)]
  \label{thm:Overview_F_Rates}
  Fix \(\alpha,K,M_0,M>0\).
  If \(\rho n^2/d\gtrsim(\log d)^{2(\alpha+1)}\), then
  \begin{equation}
    \begin{gathered}
    \inf_{\hat{\Sigma}\in\mca{M}_\rho}
    \sup_{P\in\mca{P}_\alpha^{\mca{F}}}
    \E_P \norm{\hat{\Sigma}-\Sigma(P)}^2
     \asymp
    \inf_{\hat{\Sigma}\in\mca{M}_\rho}
    \sup_{P\in\mca{P}_\alpha^{\mca{F}}}
    \frac{1}{d}\E_P \norm{\hat{\Sigma}-\Sigma(P)}_F^2 \\
    \asymp
    n^{-\frac{2\alpha}{2\alpha+1}}
    \wedge\frac{d}{n}
    +
    \left(\frac{d}{\rho n^2}\right)^{\frac{\alpha}{\alpha+1}}
    \wedge\frac{d^3}{\rho n^2}.
    \end{gathered}
    \label{eq:Overview_F_Rates}
  \end{equation}
\end{theorem}

The stronger local structure of the row-tail and pointwise-decay classes leads to a common and generally smaller operator-norm rate.

\begin{theorem}[Operator norm; \( \mca{G}_\alpha \), \( \mca{H}_\alpha \)]
  \label{thm:Overview_GH_Operator}
  Fix \(\alpha,K,M_0,M>0\).
  For each \(\mca{J}\in\{\mca{G},\mca{H}\}\),
  \begin{equation}
    \inf_{\hat{\Sigma}\in\mca{M}_\rho}
    \sup_{P\in\mca{P}_\alpha^{\mca{J}}}
    \E_P \norm{\hat{\Sigma}-\Sigma(P)}^2
    \asymp_{\log}
    n^{-\frac{2\alpha}{2\alpha+1}}
    \wedge \frac{d}{n}
    +
    \mca{Q}_{\alpha,d} \left(\frac{d}{\rho n^2}\right).
    \label{eq:Overview_GH_Operator}
  \end{equation}
\end{theorem}

Under squared Frobenius loss, the two smaller classes share the same algebraic rate, with a fixed polylogarithmic gap only for the row-tail class.

\begin{theorem}[Frobenius norm; \( \mca{G}_\alpha \), \( \mca{H}_\alpha \)]
  \label{thm:Overview_GH_Frobenius}
  Fix \(\alpha,K,M_0,M>0\).
  For each \(\mca{J}\in\{\mca{G},\mca{H}\}\),
  \begin{equation}
    \inf_{\hat{\Sigma}\in\mca{M}_\rho}
    \sup_{P\in\mca{P}_\alpha^{\mca{J}}}
    \frac{1}{d}\E_P \norm{\hat{\Sigma}-\Sigma(P)}_F^2
    \asymp_{\log}
    n^{-\frac{2\alpha+1}{2\alpha+2}}
    \wedge\frac{d}{n}
    +
    \left(\frac{d}{\rho n^2}\right)^{\frac{2\alpha+1}{2\alpha+3}}
    \wedge\frac{d^3}{\rho n^2}.
    \label{eq:Overview_GH_Frobenius}
  \end{equation}
  For \(\mca{J}=\mca{H}\), the relation can be strengthened by replacing \(\asymp_{\log}\) with \(\asymp\).
\end{theorem}

In the following sections,
\cref{sec:UpperBounds} constructs DP estimators attaining the upper bounds when the smoothness \(\alpha\) is known,
\cref{sec:LowerBound} gives the minimax lower bounds,
and \cref{sec:Adaptive} establishes procedures adapting to unknown smoothness.

\subsection{Discussions}
\label{subsec:upper-discussions}

The three preceding theorems show how the covariance class, loss function, privacy budget, and ambient dimension jointly determine the minimax risk.
We discuss how local geometry changes the optimal procedure and rate, how dimension determines the active regime, and where logarithmic gaps remain.
The structured rates below should be read together with their finite-dimensional minima against the unstructured terms.

\subsubsection{Geometry, loss, and estimator design}
\label{subsubsec:discussion-geometry}

Under squared operator loss, the three classes share the same non-private term \(n^{-2\alpha/(2\alpha+1)}\), but their privacy costs differ.
For the separated-block class \(\mca{F}_\alpha\), the privacy term is \(\xk{d/(\rho n^2)}^{\alpha/(\alpha+1)}\), whereas the row-tail and pointwise-decay classes have the smaller benchmark \(\mca{Q}_{\alpha,d}\xk{d/(\rho n^2)}\).
Privacy therefore distinguishes structures whose leading non-private operator rates coincide.
The gain for \(\mca{G}_\alpha\) and \(\mca{H}_\alpha\) comes from stronger local constraints that permit less costly distance-specific releases.

Under normalized squared Frobenius loss, both the non-private and private terms separate \(\mca{F}_\alpha\) from the two smaller classes.
The former retains the operator-loss exponents, while the latter have the minimax rate
\[
  n^{-\frac{2\alpha+1}{2\alpha+2}}
  +
  \left(\frac{d}{\rho n^2}\right)^{\frac{2\alpha+1}{2\alpha+3}},
\]
before the finite-dimensional minima.
Frobenius loss averages entrywise errors, whereas operator loss is governed by the worst spectral direction.
Consequently, for \(\mca{G}_\alpha\) and \(\mca{H}_\alpha\), the Frobenius privacy term is smaller in the structured high-dimensional regime when \(0<\alpha<1/2\) and has the same algebraic order when \(\alpha\geq1/2\).

These distinctions directly inform the design of our estimators. We start with a blockwise tridiagonal estimator whose block size carefully balances truncation bias, sampling error, and Gaussian privacy noise. This simple construction attains the optimal rates over \(\mca{F}_\alpha\) under both losses and over \(\mca{H}_\alpha\) under Frobenius loss.

To achieve the sharper rates over \(\mca{G}_\alpha\) and \(\mca{H}_\alpha\), a more targeted release scheme is required.
Because decay constraints limit the signal content in off-diagonal blocks, releasing full blocks introduces excess privacy noise.
Our multiscale block construction addresses this by treating blocks differently according to their distance from the diagonal, exploiting localized decay bounds to minimize privacy overhead. The corresponding estimators for each class and loss are presented in Section~\ref{subsec:MultiscaleBlock}.

The blockwise tridiagonal estimator is computationally efficient and easy to implement. While multiscale block procedures offer sharper statistical limits for smaller classes, they require searching over finite nets and, for the profile procedure, resolving feasible sets, adding implementation complexity. Thanks to the class inclusions, we recommend to always use the blockwise tridiagonal estimator in practice.

\subsubsection{Finite-dimensional regimes and the cost of privacy}
\label{subsubsec:discussion-phases}

Each minimax risk is the sum of a non-private statistical contribution and a privacy contribution.
For example, over \(\mca{F}_\alpha\) under operator loss, privacy dominates when \(\rho\lesssim d n^{-2\alpha/(2\alpha+1)}\); sampling error dominates on the other side.
The results are stated under \(\rho\)-zCDP.
Standard conversion to \((\epsilon,\delta)\)-DP introduces the corresponding \(\log(1/\delta)\) dependence in the upper bounds.

Dimension also determines whether the structured rate or the unstructured endpoint is active.
The unstructured statistical and privacy terms are \(d/n\) and \(d^3/(\rho n^2)\).
For \(\mca{F}_\alpha\) under either loss, the statistical and privacy transitions occur at \(d\asymp n^{1/(2\alpha+1)}\) and \(d\asymp(\rho n^2)^{1/(2\alpha+3)}\).
The operator statistical threshold for \(\mca{G}_\alpha\) and \(\mca{H}_\alpha\) is the same.
Under Frobenius loss, their thresholds are \(d\asymp n^{1/(2\alpha+2)}\) and \(d\asymp(\rho n^2)^{1/(2\alpha+4)}\).

The operator privacy term over \(\mca{G}_\alpha\) and \(\mca{H}_\alpha\) has a finer structure.
With \(u=d/(\rho n^2)\), the exact benchmark is defined in \cref{eq:Rate_Q}, and its explicit regimes are given in \cref{eq:Rate_Q_alpha_ge_half} and \cref{eq:Rate_Q_alpha_lt_half}.
The two terms are, respectively, the squared operator signal permitted at distance scale \(k\) and effective rank \(r\), and the corresponding private information capacity.
Balancing them over \(r\) gives \(u^{1/2}k^{(1-2\alpha)/2}\), explaining the change at \(\alpha=1/2\).
For \(\alpha\geq1/2\), the two regimes are those in \cref{eq:Rate_Q_alpha_ge_half}.
Here the optimizer is first dimension-limited and then satisfies \(r\asymp k\asymp(\rho n^2/d)^{1/(2\alpha+3)}\).
For \(0<\alpha<1/2\), the three regimes are those in \cref{eq:Rate_Q_alpha_lt_half}.
In the middle branch, \(k\asymp d\) while the effective rank decreases from order \(d\) to order one.
In the last branch, \(r\asymp1\) and the optimal scale becomes interior.
Thus the three branches correspond to both variables being dimension-limited, only the scale being dimension-limited, and the rank reaching its lower endpoint.

Unlike the residual non-private term \((\log d)/n\), every privacy term depends polynomially on \(d\).
In particular, when \(\rho\asymp1\), the lower bounds imply that consistency requires \(d=o(n^2)\).

\subsubsection{Sharpness, adaptation, and extensions}
\label{subsubsec:discussion-adaptation}

For known smoothness, the upper and lower bounds match up to constants over \(\mca{F}_\alpha\) under both losses and over \(\mca{H}_\alpha\) under Frobenius loss.
Over \(\mca{G}_\alpha\) under both losses and over \(\mca{H}_\alpha\) under operator loss, the bounds match up to fixed polylogarithmic factors.
The condition \(\rho n^2/d\gtrsim(\log d)^{2(\alpha+1)}\) for the constant-factor \(\mca{F}_\alpha\) comparison excludes a near-degenerate boundary at which the lower bound decreases only logarithmically.

The adaptive procedures incur no algebraic loss relative to their known-smoothness benchmarks.
For \(\mca{F}_\alpha\) under both losses and \(\mca{H}_\alpha\) under Frobenius loss, they preserve the non-private terms and add logarithmic factors only to the privacy terms.
For \(\mca{G}_\alpha\) under both losses and \(\mca{H}_\alpha\) under operator loss, they attain fixed-polylogarithmic oracle ratios over compact smoothness intervals.
Whether these adaptation costs are intrinsic remains open.

The blockwise analysis also extends to covariance classes with exponential off-diagonal decay.  If every block \(R_k\) separated by distance \(k\) satisfies \(\norm{\Sigma_{R_k}}\leq C_1e^{-\gamma k}\), then choosing \(k\asymp(\log n)\wedge\log(\rho n^2/d)\) gives squared operator risk:
\[
  \frac{\log n}{n}
  +
  \frac{d}{\rho n^2}
  \xk{
    \log\left(\frac{\rho n^2}{d}\right)+\log d
  }^2.
\]
The corresponding row-tail and pointwise-decay results follow from the class inclusions.

\section{DP Estimators and Upper Bounds}
\label{sec:UpperBounds}

In this section, we construct private estimators and establish their upper bounds with the smoothness parameter \(\alpha\) known. We begin with a blockwise tridiagonal estimator and show how its block size balances statistical error and privacy noise against truncation bias. It attains the optimal rates over \(\mathcal{F}_\alpha\) under both losses. We then exploit the stronger local decay constraints of the smaller classes \(\mathcal{G}_\alpha\) and \(\mathcal{H}_\alpha\) to improve upon naive covariance block release, which suffers from excessive privacy cost. A multiscale block estimator is introduced as a unified procedure for these refinements. We note that the idea of dividing the covariance matrix into dyadic blocks was first introduced in \citet{cai2012_AdaptiveCovariance} in the context of block thresholding for adaptive estimation in the pointwise-decay class \(\mathcal{H}_\alpha\) and the row-tail class \(\mathcal{G}_\alpha\).

\subsection{Basic Privacy Mechanisms}

We use the Gaussian mechanism and the composition and post-processing properties of zCDP summarized below \citep{dwork2014_AlgorithmicFoundations,bun2016_ConcentratedDifferential}.

\begin{lemma}[Gaussian Mechanism]
  \label{lem:DPMechanism_Gaussian_zCDP}
  Let \( f: \caD \to \R^p \) be an algorithm such that
  \begin{equation*}
    \sup_{D, D' \text{ adjacent}} \norm{f(D) - f(D')}_2 \leq \Delta.
  \end{equation*}
  Then, the algorithm \( M_f(D) = f(D) + \sigma w \) is \(\rho\)-zCDP,
  where \( w \sim N(0, I_p) \) and \( \sigma^2 = \frac{\Delta^2}{2 \rho} \).
\end{lemma}

\begin{lemma}[Composition]
  \label{lem:DPMechanism_Composition}
  If \( M: \caD \to \mca{Y} \) is \(\rho\)-zCDP and \( f: \mca{Y} \to \mca{Z} \) is an arbitrary algorithm, then \( f \circ M \) is also \(\rho\)-zCDP\@.
  If \( M_1, \ldots, M_T \) are \(\rho_1, \ldots, \rho_T\)-zCDP and \( M \) is a function of \( M_1, \ldots, M_T \),
  then \( M \) is \(\rho\)-zCDP for \(\rho = \sum_t \rho_t\).
\end{lemma}

\subsection{Dyadic block structure}
\label{subsec:DyadicBlockStructure}

We first introduce the block notation used by the estimators below and their adaptive counterparts.
Starting with blocks near the diagonal, we successively double the bandwidth and partition each newly added region into square blocks.
Without loss of generality, for each construction we append deterministic zero coordinates until \(d=2^J k_0\), where \(k_0\) is its initial block size (with \(k_0=k\) for a fixed-bandwidth estimator), since this preserves the model and privacy and changes the dimension by less than a factor of two.
For an integer block size \(k\geq1\), let \(N_k=d/k\) and define
\(   I_{k,\ell} = [1+(\ell-1)k,\ell k], \) \( \ell\in[N_k]. \)
For \(r\geq0\) and \(\ell+r\leq N_k\), put \( B_{k;\ell,r} = I_{k,\ell} \times I_{k,\ell+r}.\)
Only \(r=0,1,2,3\) will be used, and lower-triangular blocks are represented by transposes.
Blocks with invalid boundary indices are always omitted.
The blockwise tridiagonal region is
\begin{equation}
  \mca{T}_k
  =
  \bigsqcup_{\ell=1}^{N_k} B_{k;\ell,0}
  \sqcup
  \bigsqcup_{\ell=1}^{N_k-1}
  \left(B_{k;\ell,1} \sqcup B_{k;\ell,1}^{\T} \right).
  \label{eq:BlockwiseTridiagonalRegion}
\end{equation}

The regions \(\mca{T}_k\) are nested along a doubling sequence.
For \(k<d\) and \(1\leq\ell<N_{2k}\), define the upper L-shaped increment
\begin{math}
  \Gamma_{k,\ell}
  =
  B_{2k;\ell,1}\setminus B_{k;2\ell,1}.
\end{math}
The corresponding lower increment is its transpose, and hence
\begin{equation}
  \mca{T}_{2k}
  =
  \mca{T}_k
  \sqcup
  \bigsqcup_{\ell=1}^{N_{2k}-1}
  \left(\Gamma_{k,\ell} \sqcup\Gamma_{k,\ell}^{\T} \right).
  \label{eq:DyadicTridiagonalIncrement}
\end{equation}
For the dyadic construction initialized at \(k_0\), write \(k_m=2^m k_0\), \(N_m=N_{k_m}\), and abbreviate
\begin{math}
  \Gamma_{m,\ell}
  =
  \Gamma_{k_{m-1},\ell},
  m\geq1.
\end{math}
Once \(k_m \geq d/2\), \(\mca{T}_{k_m}=[d]^2\) and all subsequent increments are empty.
The upper-triangular part of the initial region \(\mca{T}_{k_0}\) consists of
\(B_{k_0;l,r}\) for \(l\in[N_{k_0}]\) and
\(r\in\{0,1\}\) satisfying \(l+r\leq N_{k_0}\).
Each increment then decomposes as
\begin{math}
  \Gamma_{k,\ell}
  =
  B_{k;2\ell-1,2}
  \sqcup
  B_{k;2\ell-1,3}
  \sqcup
  B_{k;2\ell,2}.
\end{math}
For each \(k\in\{k_0,2k_0,\ldots,d/4\}\), let \(\mca{B}_k\) be the collection of the three blocks on the right-hand side as \(1\leq\ell<N_{2k}\), which partitions the upper-triangular part of \(\mca{T}_{2k} \setminus\mca{T}_k\).
Iterating \cref{eq:DyadicTridiagonalIncrement} therefore makes the initial blocks \(B_{k_0;l,r}\) together with the families \(\mca{B}_k\) a disjoint cover of the upper-triangular entries.
For budget allocation, write \( J(k_0) = 1\vee \xk{\log_2(d/k_0)-1} \) for the number of active bands.

\begin{figure}[htp]
  \centering
  \includegraphics[height=4cm]{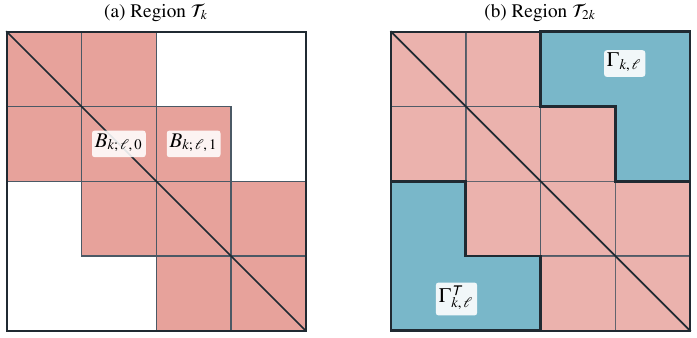}
  \caption{
    Dyadic block structure at bandwidths \(k\) and \(2k\).
  }
  \label{fig:DyadicBlockStructure}
\end{figure}

\subsection{Blockwise Tridiagonal Estimator}

We first introduce the simplest blockwise tridiagonal estimator.
\cref{alg:DP_Cov_Helper_Block} gives a private estimate of a single covariance block.
\cref{alg:DP_Cov_BlockDiagonal} applies this primitive to form a blockwise tridiagonal structure.
The block size parameter $k$ will be chosen later to balance the bias--variance--privacy trade-off.
For the results below, we use \(R=CK\sqrt{k}\), where \(C\) is a sufficiently large numerical constant.
By the Gaussian mechanism and composition, \cref{alg:DP_Cov_Helper_Block} is \(\rho_B\)-zCDP and \cref{alg:DP_Cov_BlockDiagonal} is \(\rho\)-zCDP.

\begin{algorithm}
  \caption{DP Covariance Block: \( \mf{DPCovBlock}(X,B;\rho_B,R) \)}
  \label{alg:DP_Cov_Helper_Block}
  \begin{algorithmic}
    \State \textbf{Input:} \(X\), square block $B = I\times J$, privacy budget $\rho_B$, clipping radius $R$.
    \State Let $\widetilde{X}_{i,I}=\chi_R(X_{i,I})$ and $\widetilde{X}_{i,J}=\chi_R(X_{i,J})$.
    \State Let
    \begin{math}
      \tilde{\Sigma}_{B} = \frac{1}{n} \sum_{i \in [n]} \widetilde{X}_{i,I} \widetilde{X}_{i,J}^\T,
    \end{math}
    \State \Return \( \hat{\Sigma}^{\mf{DP}}_{B} = \tilde{\Sigma}_{B} + \sigma_B M_{B} \), where \( M_{B} = \mr{GUE}(d)_{B}\) and \( \sigma_B^2 = 18R^4/(\rho_B n^2) \).
  \end{algorithmic}
\end{algorithm}

\begin{algorithm}
  \caption{DP Blockwise Tridiagonal Estimator:
    \(\mf{DPCovTri}(X;k,\rho,R)\)}
  \label{alg:DP_Cov_BlockDiagonal}
  \begin{algorithmic}
    \State \textbf{Input:} \(X\), block size $k$, privacy parameter $\rho$, truncation radius $R$.
    \State \textbf{Output:} $\rho$-zCDP estimator $\hat{\Sigma}^{\mf{DP}}$.
    \State Take $\rho_0 = \rho / (2N_k)$.
    \State Compute
    \begin{math}
      \hat{\Sigma}^{\mf{DP}}_{B_{k;l,r}}
      =
      \mf{DPCovBlock}\left(X,B_{k;l,r};\rho_0,R\right)
    \end{math}
    for all \(r\in\{0,1\}\) and \(l+r\leq N_k\).
    \State Fill the lower triangular part by symmetry, leaving other entries to zero.
    \State \textbf{return} \( \hat{\Sigma}^{\mf{DP}} \).
  \end{algorithmic}
\end{algorithm}

Compared with the tapering estimator of \citet{cai2010_OptimalRates}, the proposed blockwise tridiagonal construction is simpler and better suited for the private setting, as it streamlines both the calibration of privacy noise and the subsequent statistical analysis.
Moreover, because the estimator focuses only on a limited subset of covariance entries, the amount of noise required to guarantee privacy is substantially reduced, enabling optimal estimation accuracy under privacy constraints.
We next derive upper bounds for the blockwise tridiagonal estimator under the operator and Frobenius norms.

\begin{theorem}
  \label{thm:DPCov_Operator}
  Let $\Sigma \in \mca{F}_\alpha$ for some $\alpha > 0$.
  Consider the blockwise tridiagonal estimator $\hat{\Sigma}^{\mf{DP}}$ given in \cref{alg:DP_Cov_BlockDiagonal} with block size $k \leq c n$ for some small constant $c > 0$.
  Then, we have
  \begin{equation}
    \label{eq:DPCov_Operator}
    \E \norm{\hat{\Sigma}^{\mf{DP}}  - \Sigma}^2
    \lesssim \frac{k+ \log d}{n} + \frac{d k(k + \log d)}{\rho n^2} + k^{-2\alpha}.
  \end{equation}
\end{theorem}

\cref{thm:DPCov_Operator} characterizes the estimation error under operator norm with arbitrary block size $k$.
Optimizing over $k$ yields the following corollary.

\begin{corollary}
  \label{cor:DPCov_Operator}
  Under the same setting as in \cref{thm:DPCov_Operator},
  suppose that \(\rho n^2/d\gtrsim(\log d)^{2(\alpha+1)}\) and
  \(d\gtrsim n^{1/(2\alpha+1)} \wedge(\rho n^2)^{1/(2\alpha+3)}\).
  Take
  \[
    k
    \asymp
    \min\xk{
      n^{1/(2\alpha+1)},
      (\rho n^2/d)^{1/(2\alpha+2)}
    }
    \vee\log d
  \]
  and let \(\hat{\Sigma}_{\mca{F},\mathrm{op}}\) be the resulting blockwise tridiagonal estimator.
  Then
  \begin{equation}
    \label{eq:DPCov_Operator_OptimalRate}
    \E \norm{\hat{\Sigma}_{\mca{F},\mathrm{op}}-\Sigma}^2
    \lesssim
    n^{-\frac{2\alpha}{2\alpha+1}}
    +
    \left(\frac{d}{\rho n^2}\right)^{\frac{\alpha}{\alpha+1}}.
  \end{equation}
\end{corollary}

Under Frobenius loss, the pointwise-decay condition defining \(\mca{H}_\alpha\) gives the sharper entrywise truncation bias required by the same estimator.

\begin{theorem}
  \label{thm:DPCov_Frobenius}
  Let $\Sigma \in \mca{H}_\alpha$ for some $\alpha > 0$.
  Consider the blockwise tridiagonal estimator $\hat{\Sigma}^{\mf{DP}}$ given in \cref{alg:DP_Cov_BlockDiagonal} with block size $k \leq c n$ for some small constant $c > 0$.
  We have
  \begin{equation}
    \label{eq:DPCov_Frobenius}
    \frac{1}{d} \E \norm{\hat{\Sigma}^{\mf{DP}} - \Sigma}_F^2
    \lesssim \frac{k}{n} + \frac{d k^2}{\rho n^2} + k^{-(2\alpha+1)}.
  \end{equation}
\end{theorem}

\begin{corollary}
  \label{cor:DPCov_Frobenius}
  Under the same setting as in \cref{thm:DPCov_Frobenius}, suppose that
  \(d\gtrsim n^{1/(2\alpha+2)} \wedge(\rho n^2)^{1/(2\alpha+4)}\).
  Take
  \[
    k
    =
    \min\xk{
      n^{1/(2\alpha+2)},
      (\rho n^2/d)^{1/(2\alpha+3)}
    }
  \]
  and let \(\hat{\Sigma}_{\mca{H},\mathrm{F}}\) be the resulting blockwise tridiagonal estimator.
  Then
  \begin{equation}
    \frac{1}{d}\E\norm{\hat{\Sigma}_{\mca{H},\mathrm{F}}-\Sigma}_F^2
    \lesssim
    n^{-\frac{2\alpha+1}{2\alpha+2}}
    +
    \left(\frac{d}{\rho n^2}\right)^{\frac{2\alpha+1}{2\alpha+3}}.
  \end{equation}
\end{corollary}

The block size \(k\) balances truncation bias, sampling variance, and privacy noise.
Under operator loss, balancing the bias with the latter two terms gives \(k\asymp n^{1/(2\alpha+1)}\) and \(k\asymp(\rho n^2/d)^{1/(2\alpha+2)}\), respectively, yielding \cref{cor:DPCov_Operator}.
The Frobenius choice follows analogously from \cref{eq:DPCov_Frobenius} with loss-specific exponents.

\subsection{Multiscale Block Estimator}
\label{subsec:MultiscaleBlock}

We now turn to the pointwise-decay class \(\mca{H}_\alpha\) and the row-tail class \(\mca{G}_\alpha\).
If there were no privacy requirement, the blockwise tridiagonal estimator, with a properly tuned choice of \( k \),
can achieve the minimax optimal rate thanks to the class inclusions and the identical non-private minimax rate, see \cref{tab:Intro_Rates}.
However, under privacy, a direct application of the same construction cannot attain the sharper privacy rates for \(\mca{H}_\alpha\) and \(\mca{G}_\alpha\) under operator loss, because releasing each dense block incurs too much privacy noise to exploit their local geometry.
Likewise, naively privatizing classical estimators such as truncation or tapering does not resolve this issue.

The key difference lies in the geometry of the two smaller bandable covariance classes.
There are intermediate regions where we can exploit class geometry to pay less privacy cost than using the dense block estimator in \cref{alg:DP_Cov_Helper_Block}.
The insight is that, if the class geometry admits low-complexity approximation, then we can often construct private procedures whose cost of privacy scales with the intrinsic complexity instead of parameter dimension.

To this end, we split the matrix into three parts: a dense center \(\mca{T}_{k_0}\), dyadic outer shells between \(k_0\) and a terminal scale \(k_+\), and a tail beyond \(k_+\). We estimate the center with the blockwise tridiagonal procedure and set the tail to zero. Each outer shell is divided into the square blocks in \(\mca{B}_k\), with each block estimated by a local procedure \(\mf{BlockRelease}\).
The procedure is given in \cref{alg:DP_Cov_DyadicTemplate}.
In the following subsections, we will give the different choices of \(\mf{BlockRelease}\) for both classes \( \mca{H}, \mca{G} \) under different losses, and provide corresponding tuning parameters \( k_0, k_+ \).
The resulting architecture and the three local release modules are illustrated in \cref{fig:CenterOuterLocalReleases}.
The privacy budget allocation, zCDP composition, and post-processing make every estimator constructed from this template \(\rho\)-zCDP.
The resulting estimators are denoted by \(\widehat{\Sigma}_{\mca{H}}\), \(\widehat{\Sigma}_{\mca{G}}\) and \( \widehat{\Sigma}_{\mca{G},\mathrm{F}} \) respectively.

\begin{figure}[htp]
  \centering
  \begin{minipage}[c]{0.3\textwidth}
    \centering
    \includegraphics[width=\linewidth]{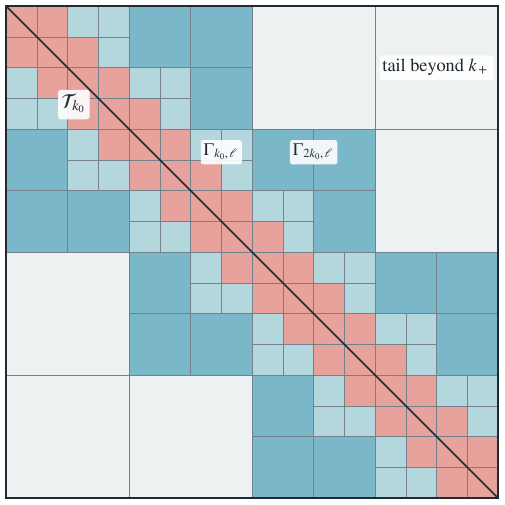}\\
    \vspace{0.5cm}
    \includegraphics[width=0.7\linewidth]{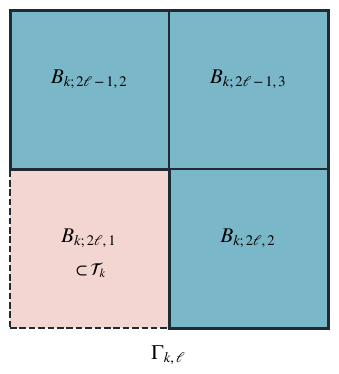}
  \end{minipage}
  \hfill
  \begin{minipage}[c]{0.6\textwidth}
    \scriptsize
    Greedy rank-one (\cref{alg:DP_Cov_H_GreedyRelease}) \\[0em]
    \includegraphics[height=3cm]{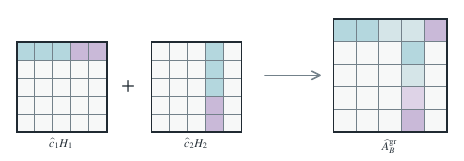}\\[-0em]
    Profile (\cref{alg:DP_Cov_G_ProfileRelease})\\[0em]
    \includegraphics[height=3cm]{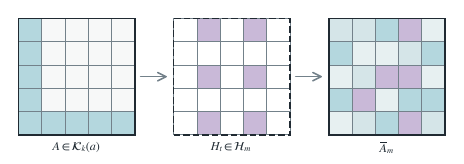}\\[-0em]
    Peeling (\cref{alg:DP_Cov_G_FrobeniusRelease})\\[0em]
    \includegraphics[height=3cm]{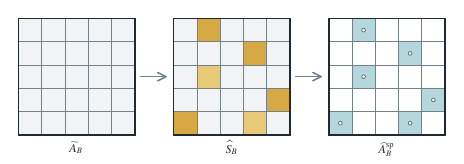}
  \end{minipage}
  \caption{Multiscale block construction and local release modules.}
  \label{fig:CenterOuterLocalReleases}
\end{figure}

\begin{algorithm}[htp]
  \caption{DP Multiscale Block Estimator}
  \label{alg:DP_Cov_DyadicTemplate}
  \begin{algorithmic}
    \State \textbf{Input:} \(X\), privacy budget \(\rho\), center size \(k_0\), terminal size \(k_+\in\{k_0,2k_0,\ldots,d/2\}\), and local procedure \(\mf{BlockRelease}\).
    \State Initialize the center by \(\widehat{\Sigma}=\mf{DPCovTri}\bigl(X;k_0,\rho/4,K\sqrt{C_1(k_0+\log n)}\bigr)\).

    \For{each \(k\in\{k_0,2k_0,\ldots\}\) satisfying \(k<k_+\)}
      \State Let \(\rho_k=\rho/ (4J(k_0)N_{2k}) \).
      \State Set \(\widehat{\Sigma}[B]=\mf{BlockRelease}(X,B;k,\rho_k)\) for each \(B\in\mca{B}_k\).
    \EndFor
    \State Fill lower-triangular of \( \widehat{\Sigma} \) by symmetry.
    \State \textbf{return} \(\Pi_{M_0}(\widehat{\Sigma})\).
  \end{algorithmic}
\end{algorithm}

\begin{theorem}[Operator-norm upper bounds]
  \label{thm:DPCov_GH_Operator}
  Fix $\alpha,K,M_0,M>0$.
  For each \(\mca{J}\in\{\mca{G},\mca{H}\}\) and all sufficiently large \(n\), the corresponding \(\widehat{\Sigma}_{\mca{J}}\) defined below is \(\rho\)-zCDP and
  \begin{equation}
    \sup_{P\in\mca{P}_{\alpha}^{\mca{J}}}
    \E_P \norm{\widehat{\Sigma}_{\mca{J}}-\Sigma(P)}^2
    \lesssim_{\log}
    n^{-\frac{2\alpha}{2\alpha+1}}
    \wedge\frac{d}{n}
    +
    \mca{Q}_{\alpha,d} \left(\frac{d}{\rho n^2}\right).
    \label{eq:DPCov_GH_Operator}
  \end{equation}
\end{theorem}

\begin{theorem}[Row-tail Frobenius upper bound]
  \label{thm:DPCov_G_Frobenius}
  Fix $\alpha,K,M_0,M>0$.
  For all sufficiently large \(n\), $\widehat{\Sigma}_{\mca{G},\mathrm{F}}$ defined below is $\rho$-zCDP and
  \begin{equation}
    \sup_{P\in\mca{P}_{\alpha}^{\mca{G}}}
    \frac{1}{d}
    \E_P \norm{\widehat{\Sigma}_{\mca{G},\mathrm{F}}-\Sigma(P)}_F^2
    \lesssim_{\log}
    n^{-\frac{2\alpha+1}{2\alpha+2}}
    \wedge\frac{d}{n}
    +
    \left(\frac{d}{\rho n^2}\right)^{\frac{2\alpha+1}{2\alpha+3}}
    \wedge\frac{d^3}{\rho n^2}.
    \label{eq:DPCov_G_Frobenius}
  \end{equation}
\end{theorem}


\subsubsection{Pointwise-decay class under operator loss: greedy blocks}
\label{subsubsec:pointwise-decay-upper-block}

Pointwise decay makes every outer block small in Frobenius norm.
The greedy procedure exploits this by building a rank-one approximation: at each round it selects a direction aligned with the current residual and releases only the corresponding noisy coefficient.
Using \(r\) rounds produces an approximation term of order \(a^2/r\), but the privacy noise only grows linearly with \(r\),
so if \( r \) is small, the privacy noise is smaller than that of the dense block estimator.

For a \(k\times k\) block \(A\), define the feasible set \(\mca{E}_k(a)=\{A:\norm{A}_F \leq a,\ \norm{A}\leq M_0\}\).
For \(\Sigma\in\mca{H}_\alpha\), every scale-\(k\) square of \(\Sigma\) belongs to \(\mca{E}_k(a_{\mca{H},k})\), where \(a_{\mca{H},k}=C_{\mca{H}}k^{-\alpha}\) and \(C_{\mca{H}}\) depends only on \((M,M_0)\).
Hence, \(\mca{E}_k(a)\) records the two facts about the local geometry of \( \mca{H}_\alpha \):
the Frobenius radius \(a\) controls low-rank approximation, and \(M_0\) supplies a uniform operator norm bound.
Let \(\mca{W}_k=\{\xi uv^\T:u,v\text{ lie in a fixed }1/4\text{-net of the unit sphere in }\R^k,\ \xi\in\{-1,1\}\}\) be the set of rank-one candidates.
The greedy rank-one block estimator is described in \cref{alg:DP_Cov_H_GreedyRelease}.

\begin{algorithm}[htpb]
  \caption{Greedy Rank-One Block Estimator}
  \label{alg:DP_Cov_H_GreedyRelease}
  \begin{algorithmic}
    \State \textbf{Input:} \(X\), square block \(B=I\times J\) of side length \(k\), radius \(a\), and budget \(\rho_B\).
    \State Choose \( r = \argmin_{1\leq r\leq k} (a^2/r + k^2 r /(\rho_B n^2)) \).
    \State Set \(\varepsilon=\sqrt{\rho_B/r}\) and \( L=CK^2 \log n \).
    \State Set \(B_0=0\).
    \For{\(t=1,\ldots,r\)}
      \State For \(H=\xi uv^\T \in\mca{W}_k\), set
      \[
        s_t(X,H)
        =
        \frac{1}{n}\sum_{i=1}^n
        \chi_L \xk{
          \xi(u^\T X_{i,I})(v^\T X_{i,J})
        }
        -
        \langle B_{t-1},H\rangle_F.
      \]
      \State Draw \(H_t\) by the exponential mechanism with privacy parameter
      \(\varepsilon\) and score \(s_t(X,H)\).
      \State Draw \(G_t \sim N(0,4L^2 r/(\rho_B n^2))\), set \(\widehat{c}_t=s_t(X,H_t)+G_t\), and update \(B_t=B_{t-1}+\widehat{c}_t H_t\).
    \EndFor
    \State Set \(\overline{B}_r=r^{-1}\sum_{t=1}^r B_{t-1}\), and project
    \(\overline{B}_r\) onto the operator-norm ball of radius \(2M_0\) to
    obtain \(\widehat{A}_B^{\mathrm{gr}}\).
    \State \textbf{return} \(\widehat{A}_B^{\mathrm{gr}}\).
  \end{algorithmic}
\end{algorithm}

For a fixed rank target \( r \),
the greedy rank-one estimator has squared operator norm error \(k/n+a_{\mca{H},k}^2/r+k^2 r/(\rho_k n^2)\),
up to logarithmic factors.
Hence, we minimize the last two terms for an optimal \( r_k \) as in \cref{alg:DP_Cov_H_GreedyRelease}.
Then, balancing the estimation errors and the tail truncation bias yield the choice of \( k_0 \) and dyadic scale \( k_+ \) satisfying
\begin{equation}
  \begin{aligned}
    k_0
    &\asymp
    \frac{d}{2}
    \wedge
    n^{1/(2\alpha+1)}
    \wedge
    \left(\frac{d}{\rho n^2}\right)^{-1/(2\alpha+3)},\\
    k_+
    &\asymp
    \frac{d}{2}
    \wedge
    n^{1/(2\alpha+1)}
    \wedge
    \left(\frac{d}{\rho n^2}\right)^{-1/(2\alpha+1)}.
  \end{aligned}
  \label{eq:DPCov_OperatorScales}
\end{equation}

\subsubsection{Row-tail class under operator loss: profile blocks}

The row--column \(\ell^1\) constraint controls the total absolute mass in each row and column, but does not give the same Frobenius radius bound used by the preceding greedy construction.
We instead take a dual perspective: even when a low-complexity approximation of the block is unavailable, its accuracy under operator loss can be controlled through simpler collections of linear measurements, allowing us to exploit their lower complexity to reduce the privacy cost.
Since the operator norm is determined by bilinear forms over pairs of unit vectors, we seek collections of probing directions that are simpler to handle privately while still controlling the full operator norm.
We construct these collections by grouping directions according to their support sizes, which allows us to exploit the row and column bounds within each group.
For a \(k\times k\) block, define the local feasible set
\[
  \mca{K}_k(a)
  =\left\{A:\max_u\sum_v\abs{A_{uv}}\vee\max_v\sum_u\abs{A_{uv}}\leq a\right\}.
\]
For \(\Sigma\in\mca{G}_\alpha\), every scale-\(k\) outer block belongs to \(\mca{K}_k(a_{\mca{G},k})\), where \(a_{\mca{G},k}=C_{\mca{G}}k^{-\alpha}\) and \(C_{\mca{G}}\) depends only on \((M,M_0)\).
Let \(\mca{D}_k=\{2^j:0\leq j\leq\ceil{\log_2k}\}\) and \(D=\abs{\mca{D}_k}\).
For \(s\in\mca{D}_k\), define
\[
  \mca{V}_s
  =\left\{u\in\R^k:\norm{u}_0\leq s,\quad\norm{u}_2\leq1,\quad\norm{u}_\infty\leq\sqrt{2/s}\right\},
\]
and group the probes into profiles
\[
  \mca{H}_m
  =\left\{\xi uv^\T:u\in\mca{V}_s,\ v\in\mca{V}_r,\ s,r\in\mca{D}_k,\ \max(s,r)=m,\ \xi\in\{-1,1\}\right\}.
\]
The discrepancy \(d_m(A,A')=\sup_{H\in\mca{H}_m}\abs{\ang{A-A',H}_F}\) measures how well two blocks agree on the probes in profile \(m\).
These profiles suffice to control operator norm error because
\begin{equation}
  \norm{E}\leq D^2\max_{m\in\mca{D}_k}d_m(E,0),
  \qquad E\in\R^{k\times k}.
  \label{eq:DPCov_ProfileGeometry}
\end{equation}
To see this, decompose each unit singular vector into dyadic groups ordered by coordinate magnitude, each belonging to the corresponding \(\mca{V}_s\).
The bilinear form then expands into at most \(D^2\) profile terms, giving only a logarithmic loss.

For each profile, we construct a candidate block by repeatedly identifying directions in which it disagrees with the data.
At each round, the exponential mechanism favors probes with a large discrepancy between the clipped empirical measurement and the value predicted by the current candidate.
The selected probes are used to update the candidate by entropy-regularized fitting, and the iterates are averaged to obtain \(\overline{A}_m\).
The row and column bounds control the fitting error within each profile.
See \cref{alg:DP_Cov_G_ProfileRelease} for an overview.
The precise clipping, update, and privacy allocation rules are given in the supplement.

Each candidate controls one profile, so we reconcile them by finding a matrix that satisfies all profile constraints.
Set
\[
  R^2=C\log^C n\left(\frac{k}{n}+\frac{ak}{n\sqrt{\rho_B}}\right),
  \qquad
  \mca{A}_B^{\mathrm{prof}}(a)
  =\left\{A\in\mca{K}_k(a):\max_{m\in\mca{D}_k}d_m(A,\overline{A}_m)\leq R\right\}.
\]
With high probability, this set contains the true block and has small diameter in operator norm.
We select its minimum Frobenius norm element to obtain a single estimate.

\begin{algorithm}[H]
  \caption{Profile Block Estimator (overview)}
  \label{alg:DP_Cov_G_ProfileRelease}
  \begin{algorithmic}
    \State \textbf{Input:} \(X\), square block \(B=I\times J\) of side length \(k\), radius \(a\), and budget \(\rho_B\).
    \For{each \(m\in\mca{D}_k\)}
      \State Construct \(\overline{A}_m\) by private probe selection and fitting, using budget \(\rho_B/D\), as specified in the supplement.
    \EndFor
    \State Form the common feasible set \(\mca{A}_B^{\mathrm{prof}}(a)\) defined above.
    \State If the set is nonempty, set \(\widehat{A}_B^{\mathrm{prof}}=\argmin_{A\in\mca{A}_B^{\mathrm{prof}}(a)}\norm{A}_F^2\); otherwise set \(\widehat{A}_B^{\mathrm{prof}}=0\).
    \State \textbf{return} \(\widehat{A}_B^{\mathrm{prof}}\).
  \end{algorithmic}
\end{algorithm}

Since fitting and reconciliation are post-processing, composition across profiles makes the block estimator \(\rho_B\)-zCDP.
The resulting block estimator has squared operator norm risk \(k/n+ak/(n\sqrt{\rho_B})\), up to logarithmic factors.
With \(a=a_{\mca{G},k}\) and \(\rho_B=\rho_k\), balancing the center, outer-block, and tail errors gives the scales \(k_0\) and \(k_+\) in \eqref{eq:DPCov_OperatorScales}.
Since \(\mca{H}_\alpha\subset\mca{G}_\alpha\), the profile estimator also gives the same upper bounds over \(\mca{H}_\alpha\).
We keep the greedy construction for that class because its stronger entrywise constraint permits a simpler procedure.

\subsubsection{Row-tail class under Frobenius loss: peeling blocks}

Under Frobenius loss, the row--column \(\ell^1\) constraint yields a direct sparse approximation.
Indeed, \(A\in\mca{K}_k(a)\) implies \(\sum_{u,v} \abs{A_{uv}}\leq ka\), so retaining the \(kr\) largest entries gives a support \(S\subseteq[k]^2\) with \(\abs{S}\leq kr\) and \(\norm{A_{S^c}}_F^2 \leq ka^2/r\).
Because the locations of these entries are unknown, the peeling procedure selects them privately and releases only their noisy values.
When \(r\) is small, this avoids the privacy cost of releasing all \(k^2\) entries.

\begin{algorithm}[H]
  \caption{Peeling Frobenius Block Estimator}
  \label{alg:DP_Cov_G_FrobeniusRelease}
  \begin{algorithmic}
    \State \textbf{Input:} \(X\), square block \(B=I\times J\) of side length \(k\), radius \(a\), and budget \(\rho_B\).
    \State Choose \( r = \argmin_{1\leq r\leq k} (a^2/r + r/n + kr^2/(\rho_B n^2)) \).
    \State Set \(\widetilde{A}_{B,uv}=n^{-1} \sum_{i=1}^n \chi_L(X_{i,I(u)} X_{i,J(v)})\) for \( L=CK^2 \log n \).
    \State Starting with all \(k^2\) coordinates, perform \(kr\) successive selections by exponential mechanism  with score \(\abs{\widetilde{A}_{B,uv}}\) and privacy parameter \(\sqrt{\rho_B/(kr)}\); denote the selected support by \(\widehat{S}_B\).
    \State For \((u,v)\in\widehat{S}_B\), set \(\widehat{A}_{B,r}^{\mathrm{sp}}(u,v)=\widetilde{A}_{B,uv}+G_{B,uv}\), where \(G_{B,uv} \stackrel{\mathrm{iid}}{\sim}N(0,4L^2 k r/(\rho_B n^2))\), and set the remaining coordinates to zero.
    \State \textbf{return} \(\widehat{A}_{B,r}^{\mathrm{sp}}\).
  \end{algorithmic}
\end{algorithm}

Up to logarithmic factors, this block estimator has squared error \(k(a^2/r+r/n+kr^2/(\rho_B n^2))\) for a fixed sparsity index \(r\),
so we minimize the three terms for an optimal \(r_k\).
The three terms balance sparse approximation, sampling on the selected entries, and privacy noise from selection and release.
We set \(k_+=d/2\) and
balance these errors to yield the choice
\begin{equation}
  k_0
  \asymp
  \frac{d}{2}
  \wedge
  n^{1/(2\alpha+2)}
  \wedge
  \left(\frac{d}{\rho n^2}\right)^{-1/(2\alpha+3)}.
\end{equation}

\section{Minimax Lower Bounds}
\label{sec:LowerBound}

In this section, we present our main minimax lower bound results for covariance estimation under DP\@.
As one of our main technical contributions, we first introduce a novel DP van Trees inequality,
which is of independent interest for other private estimation problems beyond this paper.
Then, we use it to derive minimax lower bounds for covariance estimation under zCDP\@.

\subsection{A DP van Trees Inequality}

The key technique for proving the lower bounds is \cref{thm:VanTrees__zCDP}, a novel van Trees inequality under DP\@.
It extends the classical van Trees inequality to the differentially private setting by incorporating a Fisher information bound under zCDP\@.
The proof of \cref{thm:VanTrees__zCDP} is deferred to the supplementary material.

\begin{theorem}[DP van Trees Inequality]
  \label{thm:VanTrees__zCDP}
  Let $\hat{\theta}$ be a $\rho$-zCDP estimator computed from $n$ i.i.d.\ samples from $P_\theta$, $\theta \in \R^p$.
  Let $I_{x}(\theta)$ be the Fisher information matrix of a single sample from $P_\theta$.
  Let $\pi$ be a prior distribution on $\Theta \subseteq \R^p$ with Fisher information matrix $J_\pi$.
  Under common regularity conditions, we have
  \begin{equation}
    \label{eq:VanTrees__zCDP}
    \E_{\pi} \E_{\theta} \norm{\hat{\theta} - \theta}^2_2
    \geq \frac{p^2}{I + \Tr J_\pi},\quad
    I =
    \left\{
      C_\rho \rho n^2 \int_{\Theta} \norm{I_{x}(\theta)} \dd \pi(\theta)
    \right\}
    \wedge
    \left\{
      n \int_{\Theta} \Tr {I_{x}(\theta)} \dd \pi(\theta)
    \right\},
  \end{equation}
  where \(C_\rho=(e^{2\rho}-1)/\rho\).
\end{theorem}

Let us make some remarks on \cref{thm:VanTrees__zCDP} from the following perspectives.

\subsubsection{Cost of Privacy}
We note that there are three terms in the denominator of the lower bound in \cref{eq:VanTrees__zCDP}:
the terms $n \int_{\Theta} \Tr {I_{x}(\theta)} \dd \pi(\theta)$ and $\Tr J_\pi$ are standard in the classical van Trees inequality,
while the additional term $\rho n^2 \int_{\Theta} \norm{I_{x}(\theta)} \dd \pi(\theta)$ captures the restriction of the information due to privacy.
We highlight that this term involves the operator norm $\norm{I_{x}(\theta)}$ of the Fisher information matrix instead of the trace $\Tr I_{x}(\theta)$, which is generally smaller than \( \Tr {I_{x}(\theta)} \) by a factor of the parameter dimension \( p \).
Hence, the cost of privacy often has a stronger dependence on the dimension of the parameter compared to the non-private case.
This difference is commonly observed in private estimation problems~\citep{cai2021_CostPrivacy,cai2023_ScoreAttack}.

\subsubsection{Comparison with Other Lower Bound Techniques}
There are several existing techniques for deriving minimax lower bounds under DP\@.
One approach is to use private versions of classical information-theoretic inequalities such as Fano's inequality~\citep{duchi2014_LocalPrivacy, duchi2018_MinimaxOptimal,acharya2020_DifferentiallyPrivate,cai2021_CostPrivacy}.
However, these techniques often require delicate constructions and fail to yield tight lower bounds in many problems.

Another popular approach is the fingerprinting method~\citep{bassily2014_DifferentiallyPrivate} or the recently developed score attack~\citep{cai2023_ScoreAttack}.
However, these methods typically require that the prior distribution of the parameter has independent coordinates,
which limits their applicability in complex problems where such independence may not hold, such as covariance estimation~\citep{kamath2022_NewLower,narayanan2024_BetterSimpler,portella2024_LowerBounds}.
Notably, the use of van Trees inequality for deriving minimax lower bounds under DP was first explored in \citet{cai2024_OptimalFederated}, but it was still interpreted through the lens of score attack.
Moreover, these applications typically require cumbersome control of remainder terms and often suffer from logarithmic-factor gaps.

\subsubsection{Applications}
Because it is computationally convenient, \cref{thm:VanTrees__zCDP} can be used to derive minimax lower bounds under DP for various estimation problems, yielding tight lower bounds in many cases.
To use \cref{thm:VanTrees__zCDP}, it suffices to (i) compute and bound the Fisher information matrix with respect to the parameter of interest, and (ii) construct an appropriate prior distribution over the parameter space.
For more complex problems, the prior distribution may need to be carefully designed to balance $I$ and $J_\pi$,
as we will see in the proof of our main results in \cref{sec:LowerBound}.

Let us further illustrate the use of \cref{thm:VanTrees__zCDP} with simple examples of mean estimation and linear regression.
For mean estimation, we take $x \sim N(\mu, I_d)$ and thus $I_{x}(\mu) = I_d$.
Taking a prior distribution of $\mu$ with independent coordinates satisfying regularity conditions, we obtain the lower bound
\begin{math}
  \frac{d}{n} \vee \frac{d^2}{\rho n^2}.
\end{math}
For linear regression, we take $y = \ang{\beta,x} + \ep$ with $x \sim N(0,I_d)$ and $\ep \sim N(0,1)$ independent,
and we have $I_{(x,y)}(\beta) = I_d$.
Taking a similar prior distribution on $\beta$, we obtain the same lower bound as in mean estimation.
These examples easily recover the known minimax lower bounds for mean estimation and linear regression~\citep{cai2021_CostPrivacy, cai2023_ScoreAttack}.
Therefore, we believe that \cref{thm:VanTrees__zCDP} can serve as a general-purpose tool for deriving minimax lower bounds under DP\@.

\subsection{Class-Specific Minimax Lower Bounds}
\label{subsec:Covariance_LowerBound}

Using the DP van Trees inequality,
we can provide matching minimax lower bounds for the three classes and two loss norms.
Recall that $\mca{M}_\rho$ denotes the set of all $\rho$-zCDP estimators based on $n$ samples.
We begin with the separated-block class and then turning to the two smaller classes.

\subsubsection{Separated-block class}

The separated-block lower bound holds simultaneously for all Schatten losses between the Frobenius and operator norms.
Taking $q=\infty$ gives the operator-norm lower bound, whereas $q=2$ gives the normalized Frobenius lower bound.

\begin{theorem}[Lower bound in \( \mca{F}_\alpha \)]
  \label{thm:Lower__Operator}
  For every $q\in[2,\infty]$,
  \begin{equation}
    \inf_{\hat{\Sigma} \in \mca{M}_\rho}
    \sup_{\Sigma \in \mca{F}_\alpha}
    d^{-\frac{2}{q}} \E_{\Sigma} \normsch{\hat{\Sigma} - \Sigma}^2
    \gtrsim
    n^{-\frac{2\alpha}{2\alpha+1}} \wedge \frac{d}{n} + \xk{\frac{d}{\rho n^2}}^{\frac{\alpha}{\alpha+1}} \wedge \frac{d^3}{\rho n^2}.
  \end{equation}
\end{theorem}

\subsubsection{Pointwise-decay and row-tail classes}

The stronger local structure of $\mca{G}_\alpha$ and $\mca{H}_\alpha$ changes the privacy rate under operator loss.
Thanks to \cref{prop:ClassInclusions}, it suffices to show lower bounds for the class \( \mca{H}_\alpha \).
Thus the two smaller classes share a common hard submodel for the lower-bound argument, even though their upper bounds require different private local procedures.
We have the following two theorems.

\begin{theorem}[Operator lower under $\mca{G}_\alpha, \mca{H}_\alpha$]
  \label{thm:Lower__GH_Operator}
  For $\mca{J}\in\{\mca{G},\mca{H}\}$, we have
  \begin{equation}
    \inf_{\hat{\Sigma}\in\mca{M}_\rho}
    \sup_{P\in\mca{P}_{\alpha}^{\mca{J}}}
    \E_P \norm{\hat{\Sigma}-\Sigma(P)}^2
    \gtrsim
    n^{-\frac{2\alpha}{2\alpha+1}}
    \wedge\frac{d}{n}
    +
    \mca{Q}_{\alpha,d} \left(\frac{d}{\rho n^2}\right).
    \label{eq:Lower__GH_Operator}
  \end{equation}
\end{theorem}

\begin{theorem}[Frobenius lower bound under $\mca{G}_\alpha, \mca{H}_\alpha$]
  \label{thm:Lower__Frobenius}
   For $\mca{J}\in\{\mca{G},\mca{H}\}$, we have
  \begin{equation}
    \label{eq:Lower__GH_Frobenius}
    \inf_{\hat{\Sigma}\in\mca{M}_\rho}
    \sup_{P\in\mca{P}_{\alpha}^{\mca{J}}}
    \frac{1}{d}\E_P\norm{\hat{\Sigma}-\Sigma(P)}_F^2
    \gtrsim
    n^{-\frac{2\alpha+1}{2(\alpha+1)}} \wedge \frac{d}{n} +
    \xk{\frac{d}{\rho n^2}}^{\frac{2\alpha+1}{2\alpha+3}} \wedge \frac{d^3}{\rho n^2}.
  \end{equation}
\end{theorem}

\subsubsection{Proof Strategy}

All three constructions use \cref{thm:VanTrees__zCDP} with $x\sim N(0,\Sigma)$.
We regard the covariance parameter as an element of the Euclidean space of symmetric matrices, equipped with the Frobenius inner product, and represent its Fisher information in any orthonormal basis for this inner product.
It can be computed that the trace and operator norm of the Fisher information matrix are given by
\begin{equation*}
  \Tr I_x(\Sigma) =  \frac{1}{4} \zk{\Tr(\Sigma^{-2}) + \xk{\Tr \Sigma^{-1}}^2}, \quad \norm{I_{x}(\Sigma)} = \frac{1}{2} \norm{\Sigma^{-1}}^2.
\end{equation*}

For $\mca{F}_\alpha$, we need to control the operator norm of off-diagonal blocks.
To this end, we need to carefully design a prior distribution over matrices with bounded operator norm and small Fisher information trace.
This is accomplished in \cref{lem:Lower__PriorDistribution} below.

\begin{lemma}
  \label{lem:Lower__PriorDistribution}
  There is a distribution $\Upsilon_d$ over matrices $X \in \R^{d \times d}$ with continuously differentiable density $q(X)$ supported on \( \norm{X} \leq 1 \) such that the trace of the Fisher information matrix \( J_{\Upsilon_d} \) of the distribution satisfies
  \begin{equation}
    \label{eq:Lower__PriorDistribution_Fisher}
    \Tr J_{\Upsilon_d} = \sum_{i,j=1}^d \E_{\Upsilon_d} \zk{\frac{\partial}{\partial x_{ij}} \log q(X)}^2 \leq C d^3,
  \end{equation}
  where $C$ is an absolute constant.
\end{lemma}

The proof of \cref{lem:Lower__PriorDistribution} is done by constructing an explicit distribution based on truncated Gaussian random matrices.
We note that if we only consider distributions with independent entries, the upper bound in \cref{eq:Lower__PriorDistribution_Fisher} is at least of order \( d^4 \), which is too large for our purpose.

With  \cref{lem:Lower__PriorDistribution}, we introduce a tridiagonal block structure for the prior distribution of $\Sigma$,
setting diagonal blocks as $I_k$ and the off-diagonal blocks as $c k^{-\alpha} W_l$ for \( 1 \leq l \leq N_k \),
where \( W_l \) are i.i.d.\ from the distribution \( \Upsilon_k \).
Similar calculations then yield the desired lower bound in \cref{thm:Lower__Operator},
where the balancing of the terms also corresponds to the optimal block size in \cref{cor:DPCov_Operator}.

For the operator loss over $\mca{H}_\alpha$, we use rank-$r$ perturbations within blocks of width $k$,
using the same idea of low-rank approximation in \cref{subsubsec:pointwise-decay-upper-block}.
For every $1\leq r\leq k\leq\lfloor d/4\rfloor$, the resulting Gaussian submodel gives the privacy lower bound
$\min\xk{k^{-2\alpha}/r,dkr/(\rho n^2)}$.
Optimizing over $k$ and $r$ yields $\mca{Q}_{\alpha,d}(d/(\rho n^2))$ in \cref{thm:Lower__GH_Operator}.
Finally, the Frobenius loss over $\mca{H}_\alpha$ is simpler, we take $\Sigma_{ij}=c k^{-(\alpha+1)}u_{ij}$ for $1\leq\abs{i-j}\leq k$, where the $u_{ij}$ are independent with density $\cos^2(\pi t/2)$ on $[-1,1]$.

\section{Adaptive Estimators}
\label{sec:Adaptive}

In this section, we construct adaptive estimators for the unknown decay parameter \(\alpha\).
Following the constructions in \cref{sec:UpperBounds}, we first introduce a thresholding version of the blockwise tridiagonal estimator for \(\mca{F}_\alpha\).
This estimator is relatively simple and practical to use.
For the multiscale block estimators, we use Lepski's principle to select the tuning parameters from private estimates.

\subsection{Adaptive Blockwise Tridiagonal Estimators by Thresholding}
\label{subsec:Adaptive_Construction}

  We use the idea of block thresholding introduced in \cite{cai2012_AdaptiveCovariance}  to make the blockwise tridiagonal estimator in \cref{alg:DP_Cov_BlockDiagonal} adaptive.
  The full matrix is decomposed into dyadic shells as described in \cref{subsec:DyadicBlockStructure}.
  Then, we estimate successive dyadic increments and retain those whose estimated size exceeds the thresholding level.
  The thresholding level accounts for both sampling fluctuation and privacy noise.

Fix sufficiently large constants \(C_0,C,L_1\) and a sufficiently small constant \(c_0\), and set
\[
  \begin{aligned}
    k_0&=d\wedge\ceil{C_0 \log n},
    &k_m&=2^m k_0,\\
    M&=1+\max\{m\geq0:k_m \leq\min(c_0 n,d)\},
    &N_m&=\frac{d}{k_m}\quad(0\leq m<M).
  \end{aligned}
\]
The adaptive blockwise tridiagonal estimator is given in \cref{alg:DP_Cov_Adaptive}.

\begin{algorithm}[H]
  \caption{DP Adaptive Blockwise Tridiagonal Estimator (Operator Norm)}
  \label{alg:DP_Cov_Adaptive}
  \begin{algorithmic}
    \State \textbf{Input:} \(X\) and privacy parameter \(\rho\).
    \State Initialize \( \hat{\Sigma} = \bm{0}_{d \times d} \) and set \( \rho_0 = \rho / (2MN_0) \).
    \State Compute
    \begin{math}
      \hat{\Sigma}[B_{k_0;l,r}]
      =
      \mf{DPCovBlock}
      \left(
        X,B_{k_0;l,r};
        \rho_0,\sqrt{Ck_0}
      \right)
    \end{math}
    for all \( l \in [N_0] \) and \( r \in \{0,1\} \).
    \For{\( m = 1,\dots,M-1 \)}
      \State Set \( \rho_m = \rho / (M N_m) \) and
      \begin{math}
        \tau_m^2 = L_1 \xk{\frac{k_m + \log d}{n} + \frac{k_m^2 (k_m + \log d)}{\rho_m n^2} + \exp(-2k_m)}.
      \end{math}
      \For{\( l = 1,\dots,N_m-1 \)}
        \State Compute
        \begin{math}
          A
          =
          \mf{DPCovBlock}
          \left(
            X,B_{k_m;l,1};
            \rho_m,\sqrt{Ck_m}
          \right)[\Gamma_{m,l}].
        \end{math}
        \State Set
        \begin{math}
          \hat{\Sigma}[\Gamma_{m,l}] = A \cdot \ind{\norm{A} > \tau_m}.
        \end{math}
      \EndFor
    \EndFor
    \State Fill the lower triangular part by symmetry and \textbf{return} \( \hat{\Sigma}^{\mf{Ada}} = \hat{\Sigma} \).
  \end{algorithmic}
\end{algorithm}

\begin{theorem}[Operator-norm adaptation over \(\mca{F}_\alpha\)]
  \label{thm:DP_Cov_Adaptive}
  Let $\Sigma \in \mca{F}_\alpha$ for some $\alpha > 0$.
  Suppose that $\rho n^2 / d \gtrsim (\log n)^{2\alpha+3}$.
  The estimator \(\hat{\Sigma}^{\mf{Ada}}\) in
  \cref{alg:DP_Cov_Adaptive} is \(\rho\)-zCDP and satisfies
  \begin{equation}
    \E \norm{\hat{\Sigma}^{\mf{Ada}}  - \Sigma}^2
    \lesssim n^{-\frac{2\alpha}{2\alpha+1}} \wedge  \frac{d}{n}
    + \xk{\frac{d \log n}{\rho n^2  }}^{\frac{\alpha}{\alpha+1}} \wedge \frac{d^3 \log n}{\rho n^2}.
  \end{equation}
  The same upper bound applies to the normalized Frobenius loss because
  \[
    \frac{1}{d}\E\norm{\hat{\Sigma}^{\mf{Ada}}-\Sigma}_F^2
    \leq
    \E\norm{\hat{\Sigma}^{\mf{Ada}}-\Sigma}^2.
  \]
\end{theorem}

The norm used for thresholding should match the norm of the loss.
Therefore, for adaptation over \(\mca{H}_\alpha\), we replace the operator threshold in \cref{alg:DP_Cov_Adaptive} with
\begin{equation}
  \hat{\Sigma}[\Gamma_{m,l}]
  =
  A \cdot \ind{\norm{A}_{F}^2 > k_m \tau_m^2}.
\end{equation}

\begin{theorem}[Frobenius adaptation over \(\mca{H}_\alpha\)]
  \label{thm:DP_Cov_Adaptive_Frob}
  Let $\Sigma \in \mca{H}_\alpha$ for some $\alpha > 0$.
  Suppose that $\rho n^2 / d \gtrsim (\log n)^{2\alpha+4}$.
  The Frobenius-thresholded estimator is \(\rho\)-zCDP and satisfies
  \begin{equation}
    \frac{1}{d}\E \norm{\hat{\Sigma}^{\mf{Ada}}  - \Sigma}_F^2
    \lesssim n^{-\frac{2\alpha+1}{2\alpha+2}} \wedge  \frac{d}{n}
    + \xk{\frac{d \log n}{\rho n^2  }}^{\frac{2\alpha+1}{2\alpha+3}} \wedge \frac{d^3 \log n}{\rho n^2}.
  \end{equation}
\end{theorem}

For both thresholding procedures, the statistical term is attained without an adaptation loss, while the privacy term incurs a logarithmic factor.

\subsection{Adaptive Multiscale Block Estimators}
\label{subsec:Adaptive_GH}

For the multiscale block estimator, simple thresholding no longer works: the local releases themselves require tuning to the unknown smoothness \(\alpha\).
For example, the radius \(a\) in \cref{alg:DP_Cov_H_GreedyRelease} is of order \(k^{-\alpha}\) and determines the balance between approximation error and privacy noise.
Attaining the optimal local error therefore requires choosing this radius before deciding which blocks to retain.
Thresholding alone cannot resolve this dependence, since without knowing \(\alpha\), we do not yet have the appropriately tuned block estimates to threshold.

To adapt, we use Lepski's principle~\citep{lepskii1991_ProblemAdaptive}: among estimators indexed by nested smoothness classes, choose the largest smoothness level whose estimate remains compatible with every estimate based on a weaker smoothness assumption. Here compatibility means that the distance between two estimates does not exceed the sum of their error radii. We apply this principle to the local block releases at each dyadic scale. Specifically, we construct releases over a public grid of local radii covering the prescribed smoothness interval, with smaller radii corresponding to stronger smoothness assumptions. For each proposed radius, we first use public error bounds to determine whether to use a zero, structured, or naive dense block release, where the structured block releases are those in \cref{subsec:MultiscaleBlock}. Lepski's method then selects the radius by comparing the resulting private estimates. The privacy budget is allocated across scales, radii, and blocks, and the selected block estimates are assembled to form the final estimator.

In detail, let \(\mca{J}\in\{\mca{H},\mca{G}\}\) denote the covariance class under consideration. We first specify the radius grid, error bounds, and deviation radii. At a dyadic scale \(k\), set
  \[
    J_k=\ceil{(\overline{\alpha}-\underline{\alpha})\log_2k},
    \qquad
    a_{\mca{J},k,j}
    =
    C_{\mca{J}}k^{-\underline{\alpha}}2^{-j},
    \qquad
    0\leq j\leq J_k,
  \]
  and abbreviate \(a_j=a_{\mca{J},k,j}\).
The index \(j\) increases as the candidate radius decreases, so Lepski's method selects the largest compatible index.
For every \(\alpha\in[\underline{\alpha},\overline{\alpha}]\), the grid contains a radius between \(C_{\mca{J}}k^{-\alpha}\) and \(2C_{\mca{J}}k^{-\alpha}\), where \(C_{\mca{J}}\) depends only on the class and model bounds.
Let \(\rho_{k,j}\) denote the public budget allocated to each block at radius index \(j\), as specified in the algorithm.
Given radius \(a_j\), the leading squared error bounds for the zero, structured, and dense releases are, up to constants and logarithmic factors:
\[
  \renewcommand{\arraystretch}{1.2}
  \begin{array}{c@{\quad}c@{\quad}c@{\quad}c}
    \text{mode}
    &W_{k,j,0}
    &W_{k,j,\mathrm{str}}
    &W_{k,j,\mathrm{dense}}
    \\ \hline
    \mca{H}_{\mathrm{op}}
    &a_j^2
    &\displaystyle
    \frac{k}{n}
    +\min_{1\leq r\leq k}
    \left\{
      \frac{a_j^2}{r}+\frac{k^2r}{\rho_{k,j}n^2}
    \right\}
    &\displaystyle
    \frac{k}{n}+\frac{k^3}{\rho_{k,j}n^2}
    \\
    \mca{G}_{\mathrm{op}}
    &a_j^2
    &\displaystyle
    \frac{k}{n}+\frac{a_jk}{n\sqrt{\rho_{k,j}}}
    &\displaystyle
    \frac{k}{n}+\frac{k^3}{\rho_{k,j}n^2}
    \\
    \mca{G}_{\mathrm{F}}
    &ka_j^2
    &\displaystyle
    k\min_{1\leq r\leq k}
    \left\{
      \frac{a_j^2}{r}
      +\frac{r}{n}
      +\frac{kr^2}{\rho_{k,j}n^2}
    \right\}
    &\displaystyle
    \frac{k^2}{n}+\frac{k^4}{\rho_{k,j}n^2}
  \end{array}
\]
We publicly select the release minimizing this bound, then apply a logarithmic correction to obtain the deviation radius:
\[
  t_{k,j}\in\argmin_{t\in\{0,\mathrm{str},\mathrm{dense}\}}W_{k,j,t},
  \qquad
  D_{k,j}^2=C(\log n)^C W_{k,j,t_{k,j}},
\]
The correction yields simultaneous high-probability error bounds \(D_{k,j}\) for candidates whose radius constraints hold, as required by the Lepski comparisons.
In modes \(\mca{H}_{\mathrm{op}}\), \(\mca{G}_{\mathrm{op}}\), and \(\mca{G}_{\mathrm{F}}\), the structured branch is \cref{alg:DP_Cov_H_GreedyRelease}, \cref{alg:DP_Cov_G_ProfileRelease}, and \cref{alg:DP_Cov_G_FrobeniusRelease}, respectively.
In every mode, the dense branch is \(\mf{DPCovBlock}(X,B;\rho_{k,j},R_k)\) with \(R_k^2=CK^2(k+\log n)\).
The complete procedure is given in \cref{alg:DP_Cov_GH_Adaptive}.
\begin{algorithm}[htpb]
  \caption{DP Adaptive Multiscale Block Estimator}
  \label{alg:DP_Cov_GH_Adaptive}
  \begin{algorithmic}
    \State \textbf{Input:} \(X\), privacy parameter \(\rho\), interval \([\underline{\alpha},\overline{\alpha}]\), and mode \(\mca{G}_{\mathrm{op}}\), \(\mca{H}_{\mathrm{op}}\), or \(\mca{G}_{\mathrm{F}}\).
    \State Set \(J_{\mathrm{sq}}=\log_2d-1\) and initialize \(\widehat{\Sigma}\) on \(\mca{T}_1\) as in \cref{alg:DP_Cov_DyadicTemplate}, using budget \(\rho/4\).
    \For{each \(k\in\{1,2,4,\ldots,d/4\}\)}
      \If{the mode is operator norm and \(k>n/(\log n)^C\)}
        Break.
      \EndIf
      \State Form the public radius grid \((a_{\mca{J},k,j})_{j=0}^{J_k}\) defined above.
      \State Set \(\rho_{k,j}=\rho/\{4J_{\mathrm{sq}}N_{2k}(J_k+1)\}\) for \(0\leq j\leq J_k\).
      \For{\(j=0,\ldots,J_k\)}
        \State Select a branch \(t_{k,j}\) and deviation radius \(D_{k,j}\) by the public tuning rules above.
        \For{\(B\in\mca{B}_k\)}
          \State Run branch \(t_{k,j}\) with radius \(a_{\mca{J},k,j}\) and budget \(\rho_{k,j}\) to obtain a point candidate.
          \State Radially clip the candidate to the loss-norm ball of radius \(a_{\mca{J},k,0}\) in the operator modes or \(\sqrt{k}a_{\mca{J},k,0}\) in the Frobenius mode; denote the result by \(Z_{B,k,j}\).
        \EndFor
      \EndFor
      \State Set
      \[
        \widehat{j}_k
        =
        \max\left\{
              j:
              \max_{B\in\mca{B}_k}
              \norm{Z_{B,k,j}-Z_{B,k,i}}_\star
              \leq D_{k,j}+D_{k,i}
              \text{ for every }i<j
        \right\}.
      \]
      \State Set \(\widehat{\Sigma}[B]=Z_{B,k,\widehat{j}_k}\) for each \(B\in\mca{B}_k\).
    \EndFor
    \State Fill the lower triangular part of \( \widehat{\Sigma} \) by symmetry.
    \State \textbf{return} \(\Pi_{M_0}(\widehat{\Sigma})\).
  \end{algorithmic}
\end{algorithm}

Denote the outputs in modes \(\mca{G}_{\mathrm{op}}\), \(\mca{H}_{\mathrm{op}}\), and \(\mca{G}_{\mathrm{F}}\) by \(\widehat{\Sigma}_{\mca{G}}^{\mf{Ada}}\), \(\widehat{\Sigma}_{\mca{H}}^{\mf{Ada}}\), and \(\widehat{\Sigma}_{\mca{G},\mathrm{F}}^{\mf{Ada}}\), respectively.
These adaptive estimators attain the minimax rates up to polylogarithmic factors.

\begin{theorem}[Operator-norm adaptation over \(\mca{G}_\alpha\) and \(\mca{H}_\alpha\)]
  \label{thm:DP_Cov_GH_Adaptive}
  Fix $0<\underline{\alpha}\le\overline{\alpha}<\infty$ and $K,M_0,M>0$.
  For each \(\mca{J}\in\{\mca{G},\mca{H}\}\) and all sufficiently large \(n\), the corresponding estimator \(\hat{\Sigma}^{\mf{Ada}}_{\mca{J}}\) in
  \cref{alg:DP_Cov_GH_Adaptive} is $\rho$-zCDP and, uniformly over $\alpha\in[\underline{\alpha},\overline{\alpha}]$,
  \begin{equation}
    \sup_{P\in\mca{P}_{\alpha}^{\mca{J}}}
    \E_P \norm{\hat{\Sigma}^{\mf{Ada}}_{\mca{J}}-\Sigma(P)}^2
    \lesssim_{\log}
    n^{-\frac{2\alpha}{2\alpha+1}}
    \wedge\frac{d}{n}
    +
    \mca{Q}_{\alpha,d} \left(\frac{d}{\rho n^2}\right).
    \label{eq:DP_Cov_GH_Adaptive}
  \end{equation}
\end{theorem}

\begin{theorem}[Frobenius adaptation over \(\mca{G}_\alpha\)]
  \label{thm:DP_Cov_G_Frobenius_Adaptive}
  Under the conditions of \cref{thm:DP_Cov_GH_Adaptive}, the estimator
  \(\hat{\Sigma}^{\mf{Ada}}_{\mca{G},\mathrm{F}}\) in
  \cref{alg:DP_Cov_GH_Adaptive} is \(\rho\)-zCDP and, uniformly over
  \(\alpha\in[\underline{\alpha},\overline{\alpha}]\),
  \begin{equation}
    \sup_{P\in\mca{P}_\alpha^{\mca{G}}}
    \frac{1}{d}
    \E_P \norm{
      \hat{\Sigma}^{\mf{Ada}}_{\mca{G},\mathrm{F}}-\Sigma(P)
    }_F^2
    \lesssim_{\log}
    n^{-\frac{2\alpha+1}{2\alpha+2}}
    \wedge\frac{d}{n}
    +
    \left(\frac{d}{\rho n^2}\right)^{\frac{2\alpha+1}{2\alpha+3}}
    \wedge
    \frac{d^3}{\rho n^2}.
    \label{eq:DP_Cov_G_Frobenius_Adaptive}
  \end{equation}
\end{theorem}

\section{Estimating the Precision Matrix under DP}
\label{sec:Precision}

Estimating the precision matrix, i.e., the inverse of the covariance matrix, is also a fundamental problem in statistics and machine learning, with important applications in graphical models, portfolio optimization, and other areas.
As in the non-private setting~\citep{cai2010_OptimalRates}, an estimator of the precision matrix can be obtained by applying matrix inversion to the DP covariance estimator developed in this paper.
Following the classical framework, we assume that the precision matrix is bounded in operator norm, which is equivalent to requiring the smallest eigenvalue of the covariance matrix to be bounded away from zero.
Fix \(0<c_0<M_0\).
For \(\mca{J}\in\{\mca{F},\mca{G},\mca{H}\}\), define the uniformly conditioned covariance and distribution classes
\begin{equation*}
  \tilde{\mca{J}}_\alpha
  =
  \mca{J}_\alpha
  \cap
  \dk{\Sigma:\lambda_{\min}(\Sigma)\geq c_0},
  \qquad
  \tilde{\mca{P}}_\alpha^{\mca{J}}
  =
  \dk{P\in\mca{P}_\alpha^{\mca{J}}:\Sigma(P)\in\tilde{\mca{J}}_\alpha},
\end{equation*}
and write \(\Omega(P)=\Sigma(P)^{-1}\).

Since a covariance matrix estimator $\hat{\Sigma}$ may not be invertible, we regularize it via eigenvalue truncation.
Let the eigen-decomposition of $\hat{\Sigma}$ be $\hat{\Sigma} = \hat{U} \hat{\Lambda} \hat{U}^\T$, where $\hat{\Lambda} = \mr{Diag}(\hat{\lambda}_i)_{i \in [d]}$ is the diagonal matrix of eigenvalues.
Then, we define the precision matrix estimator as
\begin{equation}
  \label{eq:Precision_Estimator}
  \hat{\Omega} = \hat{U}\mr{Diag}\xk{\max(\hat{\lambda}_i,L_2^{-1})^{-1}}_{i \in [d]} \hat{U}^\T,
\end{equation}
for any fixed constant \(L_2\geq c_0^{-1}\).

\begin{theorem}
  \label{thm:Precision__MinimaxRates}
  Suppose that \( \rho n^2 / d  \gtrsim (\log d)^{2 (\alpha+1)} \). Then we have
  \begin{equation}
    \inf_{\hat{\Omega} \in \mca{M}_\rho}
    \sup_{P \in \tilde{\mca{P}}_\alpha^{\mca{F}}}
    \E_P \norm{\hat{\Omega} - \Omega(P)}^2
    \asymp
    n^{-\frac{2\alpha}{2\alpha+1}} \wedge \frac{d}{n} + \xk{\frac{d}{\rho n^2}}^{\frac{\alpha}{\alpha+1}} \wedge \frac{d^3}{\rho n^2}.
  \end{equation}
  In particular, the upper bound can be achieved by the estimators in \cref{eq:Precision_Estimator} based on the DP covariance matrix estimators.
  For each \(\mca{J}\in\{\mca{G},\mca{H}\}\), under the conditions of \cref{thm:Overview_GH_Operator}, the same conclusion holds with \(\tilde{\mca{P}}_\alpha^{\mca{F}}\) replaced by \(\tilde{\mca{P}}_\alpha^{\mca{J}}\), the right-hand side replaced by that of \cref{eq:Overview_GH_Operator}, and \(\asymp\) replaced by \(\asymp_{\log}\).
\end{theorem}

\cref{thm:Precision__MinimaxRates} shows that the minimax rates for estimating the precision matrix under DP constraints are the same as those for estimating the covariance matrix itself, similar to the non-private setting~\citep{cai2010_OptimalRates}.
We remark that the adaptive estimator in \cref{sec:Adaptive} can also be used in \cref{eq:Precision_Estimator} for adaptive estimation of the precision matrix.

The proof of \cref{thm:Precision__MinimaxRates} is contained in the supplementary material, and we give a brief sketch here.
For the upper bound, eigenvalue truncation moves the covariance estimator by at most its distance from \(\Sigma\), and the matrix inversion formula then gives
\begin{math}
  \norm{\hat{\Omega}-\Omega}
  \leq
  2L_2c_0^{-1}\norm{\hat{\Sigma}-\Sigma}.
\end{math}
Thus, the covariance upper bounds transfer directly by post-processing.
For the lower bound, any DP precision matrix estimator can be symmetrized, spectrally truncated to the uniformly conditioned parameter space, and inverted to obtain a DP covariance matrix estimator.
The inversion map is bi-Lipschitz on this parameter space, so the covariance lower bounds in \cref{thm:Lower__Operator} and \cref{thm:Lower__GH_Operator} transfer directly.

\section{Discussion}
\label{sec:Conclusion}

This paper develops minimax and adaptive theory for differentially private estimation over the nested covariance classes
\(\mca{H}_\alpha\subset\mca{G}_\alpha\subset\mca{F}_\alpha\)
under operator and Frobenius losses. The central statistical message is that privacy fundamentally changes how these classes compare. Under operator loss, all three classes share the same leading non-private algebraic rate, yet their privacy costs separate the broad separated-block class from the row-tail and pointwise-decay classes. For the latter two classes, the operator privacy cost exhibits an additional phase transition at \(\alpha=1/2\), whereas no analogous smoothness transition occurs under Frobenius loss. Thus, the decay exponent alone does not determine the difficulty of private estimation: the local covariance geometry and the loss function also play essential roles.

The estimator constructions reflect these geometric distinctions. Over \(\mca{F}_\alpha\), blockwise tridiagonal estimation balances truncation bias, sampling error, and privacy noise. Over \(\mca{G}_\alpha\) and \(\mca{H}_\alpha\), a dyadic block decomposition separates distance scales and supports profile, peeling, and greedy procedures tailored to each class. The lower bounds are obtained by combining a zCDP van Trees inequality with dependent matrix-valued priors tailored to the geometry of each class. These constructions identify the same scale and rank obstructions that govern the upper bounds. Under uniform spectral conditioning, the operator norm guarantees further extend to precision matrix estimation through spectral truncation and inversion.

Several important questions remain open. First, the minimax rates exhibit an unavoidable polynomial dependence on the ambient dimension \(d\), even in the presence of ordered off-diagonal decay, which may limit their applicability in extremely high-dimensional regimes. It would be valuable to determine whether alternative structural assumptions, including sparsity, low rank, or Toeplitz structure \citep{cai2016_EstimatingStructured}, can reduce this dimensional dependence. Differentially private estimation of covariance functions for functional data \citep{cai_NonparametricCovariance} is another promising direction.

Second, the private covariance estimators developed here may serve as building blocks for downstream procedures, including principal component analysis, uncertainty quantification, discriminant analysis, and graphical model estimation under privacy constraints. Establishing optimality guarantees for these downstream procedures would clarify how error in covariance estimation propagates through private multivariate analysis.

Third, extending the present central DP framework to distributed or federated settings is of substantial practical interest. In such settings, observations are partitioned across multiple users or institutions that may have heterogeneous sample sizes, covariance structures, communication constraints, and privacy requirements. These features introduce new statistical and algorithmic trade-offs that are absent from the centralized model.

Finally, it remains open whether the polylogarithmic overheads in the adaptive guarantees are intrinsic. In particular, it is unknown whether the oracle ratios bounded by fixed polylogarithmic factors for the row-tail procedures and the pointwise-decay operator procedure can be sharpened, or whether the logarithmic privacy overheads for the separated-block procedures and the pointwise-decay Frobenius procedure are unavoidable.

We hope that the results and techniques developed here will stimulate further research on differentially private multivariate analysis and high-dimensional statistics.

\clearpage
\setcounter{tocdepth}{2}
\tableofcontents
\clearpage

\appendix
\section{Proofs for Blockwise Tridiagonal Estimator and Its Adaptation}
\label{sec:Proofs_CovarianceEstimation}

This section analyzes the blockwise tridiagonal estimators used for the separated-block class \(\mca{F}_\alpha\) under squared operator loss and the pointwise-decay class \(\mca{H}_\alpha\) under normalized squared Frobenius loss.
For known smoothness, we first establish privacy and blockwise concentration bounds, and then use a blocking argument to convert local estimation errors into the global risks in \cref{thm:DPCov_Operator} and \cref{thm:DPCov_Frobenius}.
We subsequently analyze thresholding rules based on the dyadic block increments to prove the adaptive guarantees in \cref{thm:DP_Cov_Adaptive} and \cref{thm:DP_Cov_Adaptive_Frob}.

\subsection{Upper Bounds for Known Smoothness}

Let us introduce some further notation for convenience.

\subsubsection{Privacy Guarantee}

\begin{proposition}
  \label{prop:DP_Cov_Helper_Block_Privacy}
  \cref{alg:DP_Cov_Helper_Block} is $\rho_B$-zCDP\@.
\end{proposition}

\begin{proof}
  By \cref{lem:DPMechanism_Gaussian_zCDP}, it suffices to bound the sensitivity.
  Let $\caD$ and $\caD'$ be two neighboring datasets differing in $X_1$ and $X_1'$.
  The blockwise truncation in \cref{alg:DP_Cov_Helper_Block} gives
  \begin{align*}
    \norm{\tilde{\Sigma}_{B}(\caD) - \tilde{\Sigma}_{B}(\caD')}_F
    &\leq \frac{1}{n}
    \norm{\widetilde{X}_{1,I} \widetilde{X}_{1,J}^\T
    -\widetilde{X}_{1,I}'(\widetilde{X}_{1,J}')^\T}_F \\
    &\leq \frac{2R^2}{n}.
  \end{align*}
  The Gaussian noise variance in \cref{alg:DP_Cov_Helper_Block} dominates the variance required for this sensitivity.
\end{proof}

\begin{proposition}
  \label{prop:DP_Cov_BlockDiagonal_Privacy}
  \cref{alg:DP_Cov_BlockDiagonal} is $\rho$-zCDP\@.
\end{proposition}

\begin{proof}
  The estimator releases at most \(2N_k-1\) blocks, each with budget \(\rho_0=\rho/(2N_k)\).
  Hence \cref{prop:DP_Cov_Helper_Block_Privacy} and \cref{lem:DPMechanism_Composition} give total privacy cost at most \(\rho\), and symmetric completion and zero filling are post-processing.
\end{proof}

\subsubsection{Concentrations}

For the blockwise tridiagonal estimators, write \(R_k=CK\sqrt{k}\), with the same sufficiently large numerical constant as in \cref{alg:DP_Cov_BlockDiagonal}.
For a block \(I\subseteq[d]\) of size \(k\), let
\[
  \widetilde{X}_{i,I}
  =
  \chi_{R_k}(X_{i,I})
\]
as in \cref{alg:DP_Cov_Helper_Block}, and write
\begin{equation*}
  \tilde{\Sigma}_B
  =
  \frac{1}{n}\sum_{i=1}^n
  \widetilde{X}_{i,I} \widetilde{X}_{i,J}^\T.
\end{equation*}

\begin{lemma}[DP Block bound]
  \label{lem:BlockBound}
  Let $B = I \times J$ be a fixed square block of size $k$ and \( \hat{\Sigma}_B^{\mf{DP}} \) be the output of \cref{alg:DP_Cov_Helper_Block} with clipping radius \(R_k\).
  Suppose that $k + \log d \leq c n$ for some small enough constant $c$.
  When the numerical constant in \(R_k\) is sufficiently large, with probability at least $1 - C d^{-10} \exp(-k)$, we have
  \begin{equation}
    \norm{\hat{\Sigma}_B^{\mf{DP}} - \Sigma_B}^2 \lesssim \tau^2(k;n,\rho,d) \coloneqq \frac{k + \log d}{n} + \frac{k^2 (k + \log d)}{\rho_0 n^2} + \exp(-2k).
  \end{equation}
\end{lemma}
\begin{proof}
  Since \(\E X_i=0\), we have \(\Sigma_B=\E X_{i,I} X_{i,J}^\T\).
  Therefore,
  \begin{align*}
    \norm{\hat{\Sigma}_B^{\mf{DP}} - \Sigma_B}
    &\leq
    \norm{\tilde{\Sigma}_B-\Sigma_B}
    +
    \norm{\sigma_B M_B}.
  \end{align*}
  The first term is bounded in
  \cref{prop:A__BlockBound_CovTruncation}.
  For the second term, using \cref{lem:MatrixGaussianConcentration},
  with probability at least $1 - C d^{-10} \exp(-k)$, we have
  \begin{equation*}
    \norm{M_B}^2 \lesssim k + \log d.
  \end{equation*}
  Recalling that $\sigma_B^2 \asymp \frac{k^2}{\rho_0 n^2}$, we obtain the desired result.
\end{proof}

\begin{proposition}
  \label{prop:A__BlockBound_CovTruncation}
  Under the conditions in \cref{lem:BlockBound}, with probability at least $1 - C d^{-10} \exp(-k)$, we have
  \begin{equation}
    \norm{\tilde{\Sigma}_B-\Sigma_B}^2
    \lesssim
    \frac{k+\log d}{n}+\exp(-2k).
  \end{equation}
\end{proposition}
\begin{proof}
  We can write
  \begin{equation*}
    \norm{\tilde{\Sigma}_B-\Sigma_B}
    \leq
    \norm{\tilde{\Sigma}_B-\E\tilde{\Sigma}_B}
    +
    \norm{\E\tilde{\Sigma}_B-\Sigma_B}
    =I_1+I_2.
  \end{equation*}
  For $I_1$, we can take $z_i = (\widetilde{X}_{i,I}^\T, \widetilde{X}_{i,J}^\T)^\T \in \R^{2k}$ so that
  $\tilde{\Sigma}_B-\E\tilde{\Sigma}_B$ is a sub-matrix of
  \( \frac{1}{n} \sum_{i=1}^n z_i z_i^\T - \E z_i z_i^\T \).
  Now,
  \begin{align*}
    \norm{z_i}_{\psi_2} &= \sup_{\norm{u}=1} \norm{\ang{u,z_i}}_{\psi_2}
    \leq \sup_{\norm{u}=1} \zk{\norm{\ang{u_I,\widetilde{X}_{i,I}}}_{\psi_2} + \norm{\ang{u_J,\widetilde{X}_{i,J}}}_{\psi_2}} \\
    & \leq \sup_{\norm{u}=1}\norm{\ang{u_I,\widetilde{X}_{i,I}}}_{\psi_2} + \sup_{\norm{u}=1}\norm{\ang{u_J,\widetilde{X}_{i,J}}}_{\psi_2} \\
    & \leq 2 \norm{X_i}_{\psi_2},
  \end{align*}
  where the last inequality uses
  \(\abs{\ang{u_I,\chi_R(X_{i,I})}}\leq\abs{\ang{u_I,X_{i,I}}}\).
  Thus, we can use \cref{lem:StandardCovarianceEst} with $t = C(k + \log d)$ to get
  \begin{equation*}
    \norm{\tilde{\Sigma}_B-\E\tilde{\Sigma}_B}
    \leq
    \norm{\frac{1}{n} \sum_{i=1}^n z_i z_i^\T - \E z_i z_i^\T}
    \lesssim  \sqrt{\frac{k + \log d}{n}},
  \end{equation*}
  with probability at least $1 - 2 d^{-10} \exp(-k)$.

  For $I_2$, we can apply \cref{prop:SubGVectorCovTruncation} to get
  \begin{equation*}
    I_2
    =
    \norm{
      \E
      \chi_{R_k}(X_{1,I})
      \chi_{R_k}(X_{1,J})^\T
      -
      \E X_{1,I} X_{1,J}^\T
    }
    \lesssim
    k\exp(-C_0 k)
    \lesssim
    \exp(-k),
  \end{equation*}
  provided that the numerical constant in \(R_k\) is sufficiently large.
  Combining the two parts, we obtain the result.
\end{proof}

\begin{lemma}[Blockwise Expectation Control]
  \label{lem:BlockExpectationBound}
  Let $B$ be a fixed square block of size $k \leq cn$ for some small $c > 0$ and \( \hat{\Sigma}_B^{\mf{DP}} \) be the output of \cref{alg:DP_Cov_Helper_Block} with clipping radius \(R_k\).
  We have
  \begin{align}
    \E \norm{\hat{\Sigma}_B^{\mf{DP}} - \Sigma_B}^2 & \lesssim \frac{k}{n} + \frac{k^3}{\rho_0 n^2} + \exp(-2k), \\
    \E \norm{\hat{\Sigma}_B^{\mf{DP}} - \Sigma_B}_F^2 & \lesssim \frac{k^2}{n} + \frac{k^4}{\rho_0 n^2} + k \exp(-2k).
  \end{align}
\end{lemma}
\begin{proof}
  It suffices to prove a bound for the operator norm since $\norm{A}_F^2 \leq k \norm{A}^2$ for a matrix $A \in \R^{k \times k}$\@.
  As in the proof of \cref{lem:BlockBound}, we decompose
  \begin{equation*}
    \norm{\hat{\Sigma}_B^{\mf{DP}} - \Sigma_B}
    \leq
    \norm{\tilde{\Sigma}_B-\Sigma_B}
    +
    \norm{\sigma_B M_B}.
  \end{equation*}
  Following \cref{prop:A__BlockBound_CovTruncation} and using the fact that
  \begin{equation*}
    \E \norm{\tilde{\Sigma}_B-\E\tilde{\Sigma}_B}^2
    \lesssim \xk{\sqrt{\frac{k}{n}} + \frac{k}{n}}^2
    \lesssim \frac{k}{n},
  \end{equation*}
  where we use \cref{lem:StandardCovarianceEst} with \( p = 2k \) and \( q = 2 \),
  we have
  \begin{equation*}
    \E \norm{\tilde{\Sigma}_B-\Sigma_B}^2
    \lesssim
    \frac{k}{n}+\exp(-2k).
  \end{equation*}
  Finally, using \cref{lem:MatrixGaussianConcentration}, we have
  \begin{equation*}
    \E \norm{\sigma_B M_B}^2 \lesssim \sigma_B^2 k \lesssim \frac{k^3}{\rho_0 n^2}.
  \end{equation*}
  Summarizing the two parts, we obtain
  \begin{equation*}
    \E \norm{\hat{\Sigma}_B^{\mf{DP}} - \Sigma_B}^2 \lesssim \frac{k}{n} + \frac{k^3}{\rho_0 n^2} + \exp(-2k),
  \end{equation*}
  which gives the desired result.
\end{proof}

Taking the estimator $\hat{\Sigma}^{\mf{DP}}$ given by \cref{alg:DP_Cov_Helper_Block} with the full block $B = [d]^2$ and clipping radius \(R_d=CK\sqrt d\),
we have the following immediate corollary of \cref{lem:BlockExpectationBound}.

\begin{corollary}
  \label{cor:FullBlockDPCov}
  Suppose that $\Sigma$ is a general covariance matrix and let \( \hat{\Sigma}^{\mf{DP}} \) be the naive DP covariance estimator given by \cref{alg:DP_Cov_Helper_Block} with the full block $B = [d]^2$ and clipping radius \(R_d\).
  Then, as long as $d \leq c n$ for some small enough constant $c > 0$, we have
  \begin{align*}
    & \E \norm{\hat{\Sigma}^{\mf{DP}} - \Sigma}^2 \lesssim \frac{d}{n} + \frac{d^3}{\rho n^2}, \\
    & \frac{1}{d} \E \norm{\hat{\Sigma}^{\mf{DP}} - \Sigma}_F^2 \lesssim \frac{d}{n} + \frac{d^3}{\rho n^2}.
  \end{align*}
\end{corollary}

The following bound, while being very rough, will be useful in controlling estimation errors when the high-probability event in \cref{lem:BlockBound} fails.

\begin{proposition}[Rough Bound]
  \label{prop:BlockRoughBound}
  Under the setting of \cref{lem:BlockExpectationBound}, the output \( \hat{\Sigma}_B^{\mf{DP}} \) satisfies
  \begin{equation}
    \E \norm{\hat{\Sigma}_B^{\mf{DP}}}_F^4 \lesssim k^4(1 + \sigma_B^4),\quad \sigma_B^2 \lesssim \frac{k^2}{\rho_0 n^2}.
  \end{equation}
\end{proposition}
\begin{proof}
  We start with
  \begin{equation*}
    \E \norm{\hat{\Sigma}_B^{\mf{DP}}}_F^4
    \lesssim
    \E \norm{\tilde{\Sigma}_B}_F^4
    +
    \E \norm{\sigma_B M_B}_F^4.
  \end{equation*}
  The first term is roughly bounded by
  \begin{align*}
    \norm{\tilde{\Sigma}_B}_F & \leq \frac{1}{n} \sum_{i=1}^n \norm{\widetilde{X}_{i,I} \widetilde{X}_{i,J}^\T}_F
    = \frac{1}{n} \sum_{i=1}^n \norm{\widetilde{X}_{i,I}} \norm{\widetilde{X}_{i,J}} \\
    & \leq R_k^2 \lesssim k.
  \end{align*}
  The second term is bounded using \cref{lem:MatrixGaussianConcentration}, which gives
  \begin{equation*}
    \E \norm{\sigma_B M_B}_F^4 \leq \sigma_B^4 \E \norm{M_B}_F^4 \lesssim \sigma_B^4 k^2 \E \norm{M_B}^4
    \lesssim \sigma_B^4 k^4.
  \end{equation*}
\end{proof}

\subsubsection{Blocking Lemma}
\label{subsec:BlockingLemma}

\begin{figure}[h]
  \centering
  \begin{minipage}{0.2\textwidth}
    \centering
    \includegraphics[width=1\textwidth]{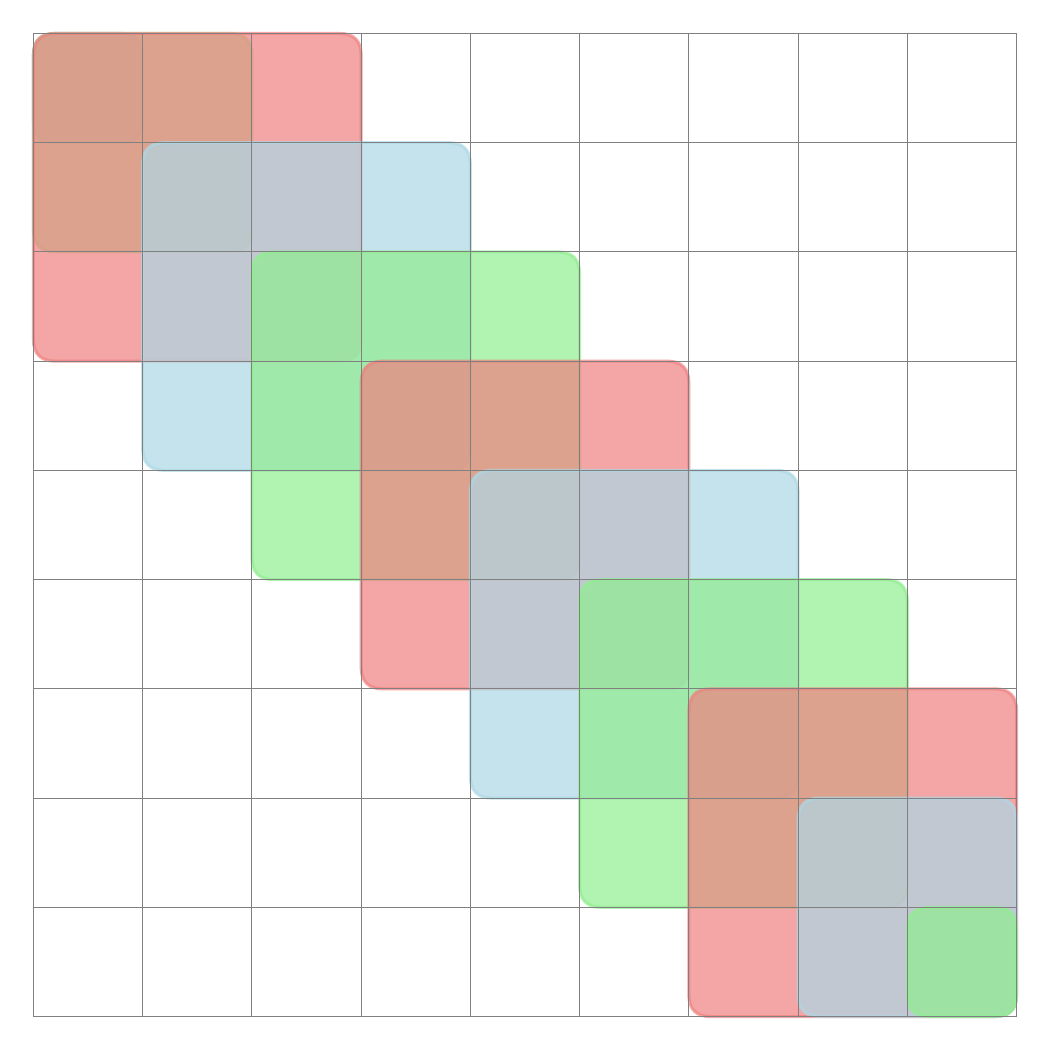}
  \end{minipage}%
  \begin{large}
    \( \longrightarrow \)
  \end{large}
  \begin{minipage}{0.2\textwidth}
    \centering
    \includegraphics[width=1\textwidth]{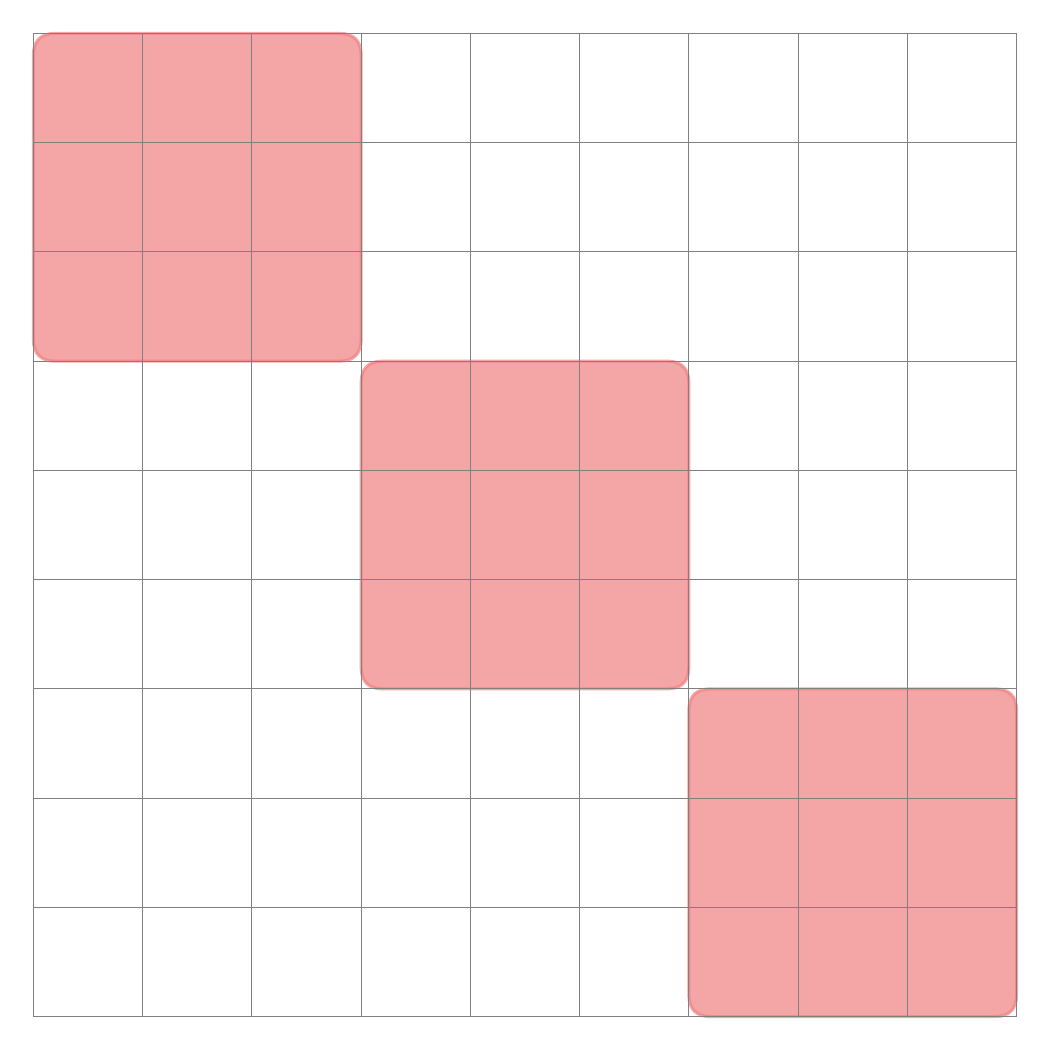}
  \end{minipage}%
  \begin{minipage}{0.2\textwidth}
    \centering
    \includegraphics[width=1\textwidth]{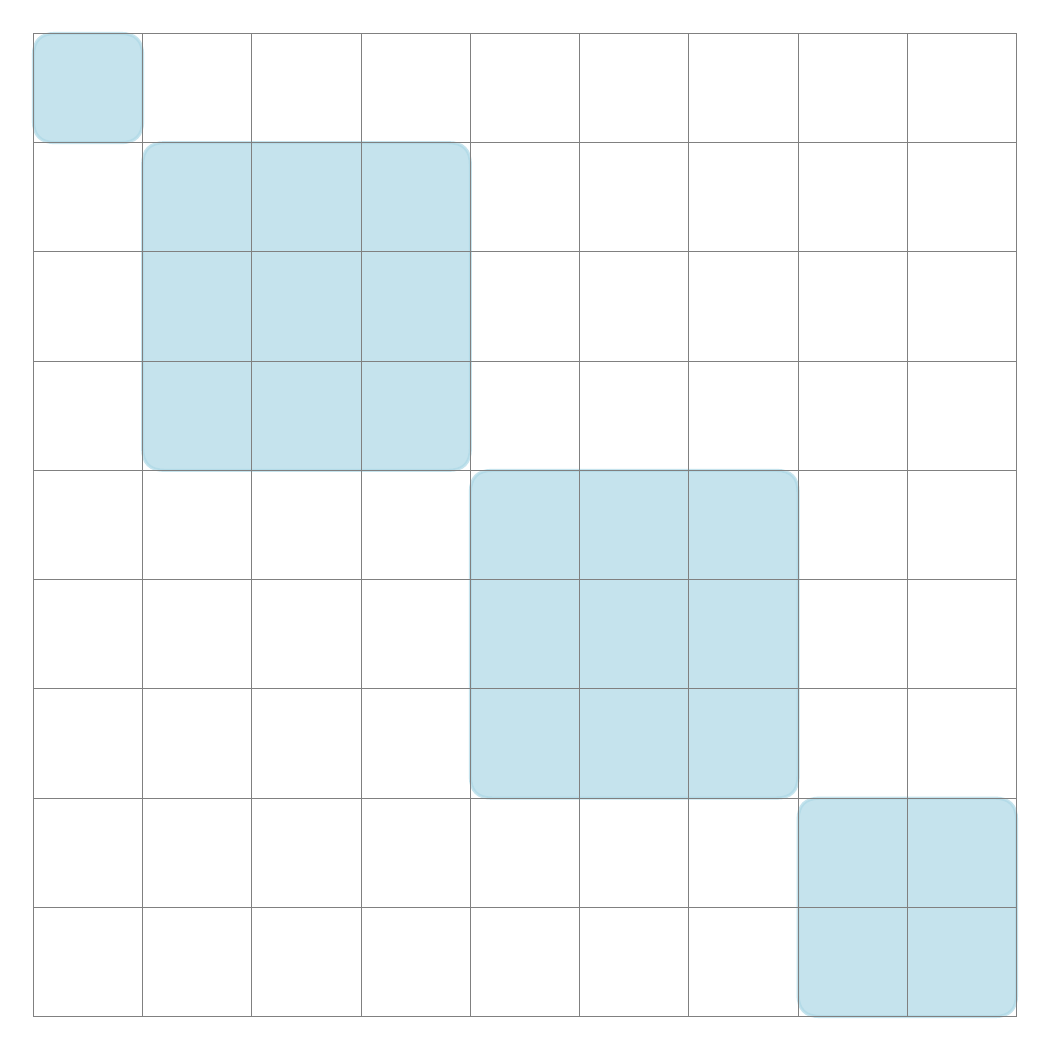}
  \end{minipage}%
  \begin{minipage}{0.2\textwidth}
    \centering
    \includegraphics[width=1\textwidth]{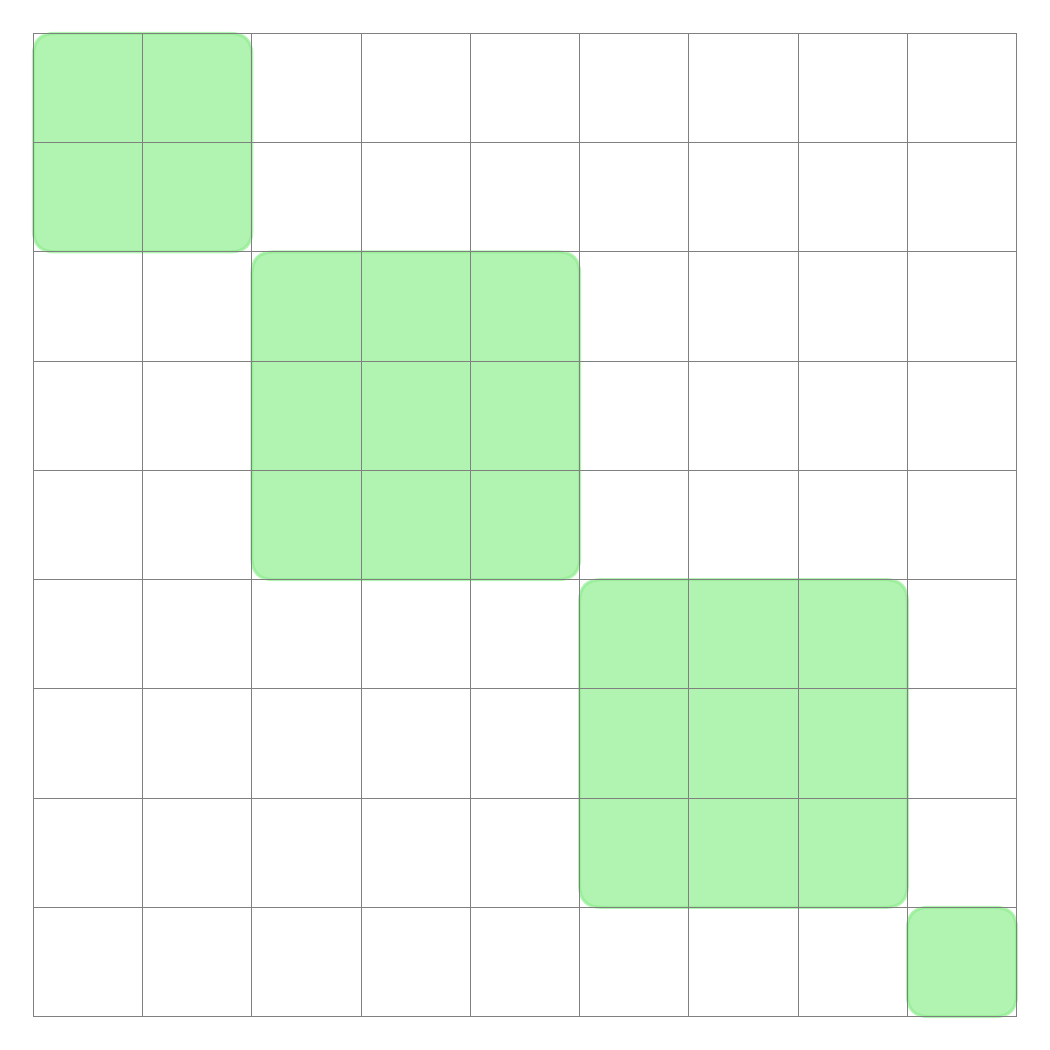}
  \end{minipage}

  \caption{
    Partitioning overlapping blocks into groups with disjoint blocks.
  }
  \label{fig:BlockPartition}
\end{figure}

To handle the overlapping blocks in the blockwise tridiagonal estimator, we introduce the following blocking lemma.

\begin{lemma}[Block Partition Lemma]
  \label{lem:BlockPartitionLemma}
  Let $A^{(1)},\dots,A^{(d+m)} \in \R^{d \times d}$ be matrices such that $A^{(k)}$ and $A^{(k')}$ have disjoint blocks if $\abs{k-k'} \geq m$.
  Then, we have
  \begin{equation}
    \norm{\sum_{k=1}^{d+m} A^{(k)}} \leq m \max_{k \in [d+m]} \norm{A^{(k)}}.
  \end{equation}
\end{lemma}
\begin{proof}
  See \cref{fig:BlockPartition}.
  We partition the set $\dk{1,\dots,d+m}$ into $m$ groups $S_k = \dk{ k, k+m, k+2m, \dots}$ for $k \in [m]$.
  Then, for each group $S_k$, the matrices $\dk{A^{(i)} : i \in S_k}$ have disjoint blocks,
  so we have
  \begin{equation*}
    \norm{\sum_{i \in S_k} A^{(i)}} = \max_{i \in S_k} \norm{A^{(i)}} \leq \max_{k \in [d+m]} \norm{A^{(k)}}.
  \end{equation*}
  Consequently, we have
  \begin{equation*}
    \norm{\sum_{k=1}^{d+m} A^{(k)}} = \norm{\sum_{k=1}^m \sum_{i \in S_k} A^{(i)}}
    \leq \sum_{k=1}^m \norm{\sum_{i \in S_k} A^{(i)}} \leq m \max_{k \in [d+m]} \norm{A^{(k)}}.
  \end{equation*}
\end{proof}

\begin{corollary}[Blockwise Tridiagonal Partition]
  \label{cor:TriDiagonalNormBound}
  Let \(\bigsqcup_l I_l=[d]\) and let
  \(A=\sum_{\abs{l-l'}\leq1}A_{l,l'}\), where \(A_{l,l'}\) is supported on \(I_l \times I_{l'}\).
  Then,
  \begin{equation}
    \norm{A} \leq 4 \max_{\abs{l-l'}\leq1} \norm{A_{l,l'}}.
  \end{equation}
\end{corollary}
\begin{proof}
  We can take
  \begin{equation*}
    A^{(l)} = A_{l,l}+A_{l,l+1}+A_{l+1,l}.
  \end{equation*}
  Then they have disjoint blocks if $\abs{l-l'} \geq 2$, and we have
  \begin{equation*}
    \norm{A^{(l)}} \leq 2 \max_{\abs{i-j}\leq1} \norm{A_{i,j}}.
  \end{equation*}
  Applying \cref{lem:BlockPartitionLemma} with $m=2$ gives the desired result.
\end{proof}

\subsubsection{Proof of \cref{thm:DPCov_Frobenius}}
\label{subsec:Proof_DPCov_Frobenius}
Recall the blockwise tridiagonal region \(\mca{T}_k\) from
\cref{eq:BlockwiseTridiagonalRegion}.
Then, we have
\begin{align*}
  \norm{\hat{\Sigma}^{\mf{DP}} - \Sigma}_F^2 =
  \norm{\hat{\Sigma}^{\mf{DP}}_{\mca{T}_k} - \Sigma_{\mca{T}_k}}_F^2 + \norm{\Sigma_{\mca{T}_k^\complement}}_F^2.
\end{align*}
Using the hypothesis that $\Sigma \in \mca{H}_\alpha$, we have
\begin{equation*}
  \norm{\Sigma_{\mca{T}_k^\complement}}_F^2
  \lesssim \sum_{i,j:\abs{i-j} > k} \abs{i-j}^{-2(\alpha+1)} \lesssim d k^{-(2\alpha + 1)}.
\end{equation*}

On the other hand, applying \cref{lem:BlockExpectationBound}, we have
\begin{align*}
  \norm{\hat{\Sigma}^{\mf{DP}}_{\mca{T}_k} - \Sigma_{\mca{T}_k}}_F^2
  &= \sum_{l \in [N_k]} \norm{\hat{\Sigma}^{\mf{DP}}_{B_{k;l,0}} - \Sigma_{B_{k;l,0}}}_F^2
  + 2 \sum_{l < N_k} \norm{\hat{\Sigma}^{\mf{DP}}_{B_{k;l,1}} - \Sigma_{B_{k;l,1}}}_F^2 \\
  & \lesssim N_k \xk{\frac{k^2}{n} + \frac{k^4}{\rho_0 n^2} + k \exp(-2k)}  \\
  & \lesssim \frac{d k}{n} + \frac{d^2 k^2}{\rho n^2} + d \exp(-2k),
\end{align*}
where we recall that \( N_k = d/k \) and \( \rho_0 = \rho / N_k \).
Combining the two estimates, we have
\begin{equation*}
  \frac{1}{d}\E \norm{\hat{\Sigma}^{\mf{DP}} - \Sigma}_F^2
  \lesssim \frac{k}{n} + \frac{d k^2}{\rho n^2} + k^{-(2\alpha+1)},
\end{equation*}
where we note that the $\exp(-2k)$ is dominated by $k^{-(2\alpha+1)}$,
and the last term vanishes if $k \geq d/2$.

\subsubsection{Proof of \cref{thm:DPCov_Operator}}
\label{subsec:Proof_DPCov_Operator}
Let us use the same notation as in the proof of \cref{thm:DPCov_Frobenius}.
Still, we can write
\begin{equation*}
  \hat{\Sigma}^{\mf{DP}} - \Sigma = \hat{\Sigma}^{\mf{DP}}_{\mca{T}_k} - \Sigma_{\mca{T}_k} + \Sigma_{\mca{T}_k^\complement},
\end{equation*}
so
\begin{equation*}
  \norm{\hat{\Sigma}^{\mf{DP}} - \Sigma}
  \leq \norm{\hat{\Sigma}^{\mf{DP}}_{\mca{T}_k} - \Sigma_{\mca{T}_k}} + \norm{\Sigma_{\mca{T}_k^\complement}}.
\end{equation*}

For the first term, recalling the blockwise tridiagonal structure and applying \cref{cor:TriDiagonalNormBound},
we get
\begin{align*}
  \norm{\hat{\Sigma}^{\mf{DP}}_{\mca{T}_k} - \Sigma_{\mca{T}_k}}
  &\leq
  4\max\xk{
    \max_{l\leq N_k} \norm{\hat{\Sigma}^{\mf{DP}}_{B_{k;l,0}}-\Sigma_{B_{k;l,0}}},
    \max_{l<N_k} \norm{\hat{\Sigma}^{\mf{DP}}_{B_{k;l,1}}-\Sigma_{B_{k;l,1}}}
  },
\end{align*}
Let us use \cref{lem:BlockBound} with union bound over all blocks (at most $d$) in the band and denote the event by $E$,
which holds with probability at least $1 - C d^{-9} \exp(-k)$.
Under the event $E$, we have
\begin{align*}
  \max_{\substack{r\in\{0,1\}\\l+r\leq N_k}}
  \norm{\hat{\Sigma}^{\mf{DP}}_{B_{k;l,r}}-\Sigma_{B_{k;l,r}}}^2
  &\lesssim
  \frac{k+\log d}{n}
  +\frac{k^2(k+\log d)}{\rho_0 n^2}
  +\exp(-2k).
\end{align*}
Combining the above bound with \cref{prop:Proof_OffBandBound}, and using \( (a+b)^2 \lesssim a^2 + b^2 \), we get on the event \( E \)
\begin{equation*}
  \norm{\hat{\Sigma}^{\mf{DP}} - \Sigma}^2
  \lesssim \frac{k + \log d}{n} + \frac{d k (k + \log d)}{\rho n^2} + k^{-2\alpha}.
\end{equation*}

It remains to bound the expectation on the event $E^\complement$.
Using the Cauchy--Schwarz inequality, we have
\begin{equation*}
  \E \zk{\bm{1}_{E^\complement} \norm{\hat{\Sigma}^{\mf{DP}} - \Sigma}^2}
  \leq \sqrt{\bbP(E^\complement)} \sqrt{\E \norm{\hat{\Sigma}^{\mf{DP}} - \Sigma}^4}.
\end{equation*}
For the last term, we can use
\begin{align*}
  \E \norm{\hat{\Sigma}^{\mf{DP}} - \Sigma}^4 \lesssim \E \norm{\hat{\Sigma}^{\mf{DP}}}^4 + \norm{\Sigma}^4 \lesssim 1 + \E \norm{\hat{\Sigma}^{\mf{DP}}}^4.
\end{align*}
Now, we bound the last term very roughly by
\begin{align*}
  \norm{\hat{\Sigma}^{\mf{DP}}}^4
  \leq \norm{\hat{\Sigma}^{\mf{DP}}}^4_F
  =
  \left\{
    \sum_{l\leq N_k} \norm{\hat{\Sigma}^{\mf{DP}}_{B_{k;l,0}}}_F^2
    +
    2\sum_{l<N_k} \norm{\hat{\Sigma}^{\mf{DP}}_{B_{k;l,1}}}_F^2
  \right\}^2
  \lesssim
  d\sum_{\substack{r\in\{0,1\}\\l+r\leq N_k}}
  \norm{\hat{\Sigma}^{\mf{DP}}_{B_{k;l,r}}}_F^4,
\end{align*}
so with \cref{prop:BlockRoughBound}, we have
\begin{equation*}
  \E  \norm{\hat{\Sigma}^{\mf{DP}}}^4
  \lesssim
  d\sum_{\substack{r\in\{0,1\}\\l+r\leq N_k}}
  \E\norm{\hat{\Sigma}^{\mf{DP}}_{B_{k;l,r}}}_F^4
  \lesssim d^2 k^4(1 + \sigma_B^4)
  \lesssim d^6 \zk{1 + \xk{\frac{kd}{\rho n^2}}^2}.
\end{equation*}
Consequently, we have
\begin{align*}
  \E \zk{\bm{1}_{E^\complement} \norm{\hat{\Sigma}^{\mf{DP}} - \Sigma}^2}
  &\lesssim
  \left\{
    d^{-9} e^{-k}\,d^2 k^4
    \left[1+\left(\frac{kd}{\rho n^2}\right)^2\right]
  \right\}^{1/2}\\
  &\lesssim
  d^{-3/2} e^{-k/2}
  \left(1+\frac{kd}{\rho n^2}\right),
\end{align*}
which is negligible compared to the previous bound.
Therefore, the theorem follows.

\paragraph{Bounding the Off-Band Terms}

\begin{proposition}
  \label{prop:Proof_OffBandBound}
  Let \(\Sigma\in\mca{F}_\alpha\), and let \(\mca{T}_k\) be the blockwise tridiagonal region in \cref{eq:BlockwiseTridiagonalRegion}.
  Then,
  \begin{equation*}
    \norm{\Sigma_{\mca{T}_k^\complement}}^2 \lesssim k^{-2\alpha} \ind{k < d/2}.
  \end{equation*}
\end{proposition}
\begin{proof}
  When \(k\geq d/2\), we have \(\mca{T}_k^\complement=\emptyset\), so the result is immediate.
  Use the dyadic construction in \cref{subsec:DyadicBlockStructure} with \(k_0=k\) and \(M=\lceil\log_2(d/k)\rceil\).
  Then
  \begin{equation*}
    \mca{T}_k^\complement
    =
    \bigsqcup_{m=1}^{M}
    \left(\mca{T}_{k_m} \setminus\mca{T}_{k_{m-1}}\right),
  \end{equation*}
  where each increment has a blockwise tridiagonal incidence pattern.
  Since every \(\Gamma_{m,l}\) consists of two \(k_{m-1}\)-off-diagonal blocks, the definition of \(\mca{F}_\alpha\) gives
  \begin{equation*}
    \norm{\Sigma[\Gamma_{m,l}]}
    \leq
    2k_{m-1}^{-\alpha}
    \lesssim
    2^{-\alpha m} k^{-\alpha}.
  \end{equation*}
  Combining this with \cref{cor:TriDiagonalNormBound}, we get
  \begin{equation*}
    \norm{\Sigma[\mca{T}_{k_m} \setminus\mca{T}_{k_{m-1}}]}
    \leq
    4\max_l \norm{\Sigma[\Gamma_{m,l}]}
    \lesssim
    2^{-\alpha m} k^{-\alpha}.
  \end{equation*}
  Summing over $m$ yields
  \begin{align*}
    \norm{\Sigma_{\mca{T}_k^\complement}}
    &\leq
    \sum_{m=1}^M
    \norm{\Sigma[\mca{T}_{k_m} \setminus\mca{T}_{k_{m-1}}]}
    \lesssim
    k^{-\alpha} \sum_{m=1}^M2^{-\alpha m}
    \lesssim
    k^{-\alpha}.
  \end{align*}
\end{proof}

\subsubsection{Schatten Norms}

The operator-norm upper bound also yields a consequence for the Schatten-\(q\) norm when \(q\geq2\).
Since
\[
  \normsch{A}^2 \leq d^{2/q} \norm{A}^2
\]
for every \(A\in\R^{d\times d}\), we obtain the following result.

\begin{corollary}[Schatten-norm consequence over the separated-block class]
  \label{cor:DPCov_F_Schatten}
  Let \(q\geq2\) and suppose that the conditions of \cref{cor:DPCov_Operator} hold.
  The estimator \(\hat{\Sigma}_{\mca{F},\mathrm{op}}\) in \cref{cor:DPCov_Operator} satisfies
  \(\sup_{\Sigma\in\mca{F}_\alpha}d^{-2/q} \E\normsch{\hat{\Sigma}_{\mca{F},\mathrm{op}}-\Sigma}^2 \lesssim n^{-2\alpha/(2\alpha+1)}+\left(d/(\rho n^2)\right)^{\alpha/(\alpha+1)}\).
  In particular, \(q=2\) gives the structured Frobenius rate over \(\mca{F}_\alpha\).
\end{corollary}

\subsection{Error Bounds for Block Thresholding}
\label{subsec:Proofs_BlockThresholding}

Set $\hat{\Sigma}=\hat{\Sigma}^{\mf{Ada}}$.

\begin{lemma}[Block Thresholding]
  \label{lem:BlockThresholding}
  Fix \(\Gamma_{m,l}\) and suppose that
  \(k_m+\log d\leq c_0 n\) for a sufficiently small \(c_0>0\).
  Then, with probability at least $1-Cd^{-10}\exp(-k_m)$,
  \begin{equation}
    \norm{\hat{\Sigma}[\Gamma_{m,l}] - \Sigma[\Gamma_{m,l}]} \lesssim \min(\norm{\Sigma[\Gamma_{m,l}]}, \tau_m),
  \end{equation}
  where
  \begin{equation*}
    \tau_m^2 = L_1 \xk{\frac{k_m + \log d}{n} + \frac{k_m^2 (k_m + \log d)}{\rho_m n^2} + \exp(-2k_m)}.
  \end{equation*}
\end{lemma}
\begin{proof}
  Let
  \[
    A_0
    =
    \mf{DPCovBlock}
    \left(
      X,B_{k_m;l,1};
      \rho_m,\sqrt{Ck_m}
    \right)
  \]
  and
  \(A=A_0[\Gamma_{m,l}]\).
  The algorithm sets
  \(\hat{\Sigma}[\Gamma_{m,l}]=A\ind{\norm{A}>\tau_m}\).
  Denote \(B=B_{k_m;l,1}\) and \(\Gamma=\Gamma_{m,l}\).
  Using \cref{lem:BlockBound}, we have
  \( \norm{A_0- \Sigma_B} \leq \frac{1}{4}\tau_m \) with probability at least $1 - C d^{-10} \exp(-k_m)$ as long as we take $L_1$ large enough.
  Since \(\Gamma_{m,l}=B_{k_m;l,1} \setminus B_{k_{m-1};2l,1}\),
  for any matrix \(M\),
  \begin{equation*}
    \norm{M[\Gamma_{m,l}]}
    \leq
    \norm{M[B_{k_m;l,1}]}
    +
    \norm{M[B_{k_{m-1};2l,1}]}
    \leq
    2\norm{M[B_{k_m;l,1}]}.
  \end{equation*}
  Hence,
  \begin{equation*}
    \norm{A - \Sigma_\Gamma} \leq 2 \norm{A_0 - \Sigma_B} \leq \frac{1}{2}\tau_m.
  \end{equation*}
  Consequently,
  \begin{equation*}
    \norm{A} > \tau_m \implies \norm{\Sigma_\Gamma} > \frac{1}{2}\tau_m,\quad \text{and} \quad
    \norm{A} \leq \tau_m \implies \norm{\Sigma_\Gamma} \leq \frac{3}{2}\tau_m.
  \end{equation*}
  Therefore, if \( \norm{A} > \tau_m \), we have
  \begin{equation*}
    \norm{A\cdot \ind{\norm{A} > \tau_m} - \Sigma_\Gamma} = \norm{A - \Sigma_\Gamma} \leq \frac{1}{2}\tau_m \leq \norm{\Sigma_\Gamma}.
  \end{equation*}
  On the other hand, if \( \norm{A} \leq \tau_m \), we have
  \begin{equation*}
    \norm{A\cdot \ind{\norm{A} > \tau_m} - \Sigma_\Gamma} = \norm{\Sigma_\Gamma} \leq \frac{3}{2}\tau_m.
  \end{equation*}
  Thus,
  \begin{equation*}
    \norm{A\cdot \ind{\norm{A} > \tau_m} - \Sigma_\Gamma} \leq \min(\norm{\Sigma_\Gamma}, 2\tau_m).
  \end{equation*}

\end{proof}

\subsubsection{Operator Loss: Proof of \cref{thm:DP_Cov_Adaptive}}

Let the largest retained blockwise tridiagonal region be \(\mca{T}_{k_{M-1}}\).
Apply a union bound to the initial blocks \(B_{k_0;l,0}\) and \(B_{k_0;l,1}\) using \cref{lem:BlockBound}, and to the increments \(\Gamma_{m,l}\) using \cref{lem:BlockThresholding}; denote the resulting event by \(E\),
which satisfies $\bbP(E) \geq 1 - C d^{-8} \exp(-k_0)$.

On $E$, decompose
\begin{equation*}
  \norm{ \hat{\Sigma} - \Sigma}
  \leq \norm{\hat{\Sigma}_{\mca{T}_{k_{M-1}}} - \Sigma_{\mca{T}_{k_{M-1}}}} + \norm{\Sigma_{\mca{T}_{k_{M-1}}^\complement}}.
\end{equation*}
The second term is bounded by \cref{prop:Proof_OffBandBound}, which gives
\begin{equation}
  \label{eq:Proof_Adaptive_OffBandBound}
  \norm{\Sigma_{\mca{T}_{k_{M-1}}^\complement}}\lesssim k_{M-1}^{-\alpha} \ind{k_{M} < d}.
\end{equation}

By iterating \cref{eq:DyadicTridiagonalIncrement}, we decompose
\begin{equation*}
  \mca{T}_{k_{M-1}}
  =
  \mca{T}_{k_0}
  \sqcup
  \bigsqcup_{m=1}^{M-1}
  \left(\mca{T}_{k_m} \setminus\mca{T}_{k_{m-1}}\right),
\end{equation*}
Both \(\mca{T}_{k_0}\) and every increment have blockwise tridiagonal incidence patterns.
Decomposing the terms yields
\begin{equation}
  \label{eq:Proof_Adaptive_GammaBlockDecomp}
  \norm{\hat{\Sigma}_{\mca{T}_{k_{M-1}}} - \Sigma_{\mca{T}_{k_{M-1}}}}
  \leq
  \norm{\hat{\Sigma}_{\mca{T}_{k_0}} - \Sigma_{\mca{T}_{k_0}}}
  +
  \sum_{m=1}^{M-1}
  \norm{
    \hat{\Sigma}_{\mca{T}_{k_m} \setminus\mca{T}_{k_{m-1}}}
    -
    \Sigma_{\mca{T}_{k_m} \setminus\mca{T}_{k_{m-1}}}
  }.
\end{equation}
\cref{cor:TriDiagonalNormBound} and \cref{lem:BlockBound} yield
\begin{align*}
  \norm{\hat{\Sigma}_{\mca{T}_{k_0}} - \Sigma_{\mca{T}_{k_0}}}
  \lesssim \tau_0 = \xk{\frac{k_0 + \log d}{n}}^{1/2} + \xk{\frac{k_0^2 (k_0 + \log d)}{\rho_0 n^2}}^{1/2} + \exp(-k_0).
\end{align*}
Similarly, \cref{cor:TriDiagonalNormBound} and \cref{lem:BlockThresholding} yield
\begin{equation}
  \norm{
    \hat{\Sigma}_{\mca{T}_{k_m} \setminus\mca{T}_{k_{m-1}}}
    -
    \Sigma_{\mca{T}_{k_m} \setminus\mca{T}_{k_{m-1}}}
  }
  \lesssim \max_{l<N_m} \norm{\hat{\Sigma}[\Gamma_{m,l}] - \Sigma[\Gamma_{m,l}]}
  \lesssim \max_{l<N_m} \min(\norm{\Sigma[\Gamma_{m,l}]}, \tau_m).
\end{equation}
Since \(\Sigma\in\mca{F}_\alpha\) and \(\Gamma_{m,l}\) is the union of two \(k_{m-1}\)-off-diagonal blocks, we have
\begin{equation*}
  \max_{l<N_m} \norm{\Sigma[\Gamma_{m,l}]}
  \lesssim
  k_{m-1}^{-\alpha}
  \lesssim
  2^{-\alpha m} k_0^{-\alpha}.
\end{equation*}
Since \(k_0 \asymp\log n\), \(k_0+\log d\asymp k_0\),
\(k_m=2^mk_0\), \(N_m\lesssim d/k_m\), and
\(\rho_m=\rho/(MN_m)\), we have
\begin{align*}
  \tau_m &\lesssim \xk{\frac{k_m + \log d}{n} + \frac{M d k_m (k_m + \log d)}{\rho n^2} + \exp(-2k_m)}^{1/2} \\
  & \lesssim \frac{1}{\sqrt{n}} k_m^{\hf} + \xk{\frac{M d}{ \rho n^2}}^{\hf} k_m + \exp(-k_m) \\
  & \lesssim \frac{\sqrt{k_0}}{\sqrt{n}} 2^{m/2} + \xk{\frac{M d k_0^2}{\rho n^2}}^{\hf} 2^m.
\end{align*}
Plugging the two bounds into \cref{eq:Proof_Adaptive_GammaBlockDecomp}, we obtain
\begin{equation*}
  \norm{\hat{\Sigma}_{\mca{T}_{k_{M-1}}} - \Sigma_{\mca{T}_{k_{M-1}}}}
  \lesssim \tau_0 + \sum_{m=1}^{M-1} \min(2^{-\alpha m} k_0^{-\alpha}, \tau_m)
\end{equation*}
The bias term decreases with $m$, whereas $\tau_m$ increases.
Splitting the sum at some $0 \leq m_* \leq M-1$, we have
\begin{align*}
  \norm{\hat{\Sigma}_{\mca{T}_{k_{M-1}}} - \Sigma_{\mca{T}_{k_{M-1}}}} &
  \lesssim \sum_{0 \leq m \leq m_*} \tau_m + \sum_{m_* < m \leq M} k_0^{-\alpha} 2^{-\alpha m} \\
  & \lesssim \frac{\sqrt{k_0}}{\sqrt{n}} 2^{m_*/2} + \xk{\frac{M d k_0^2}{\rho n^2}}^{\hf} 2^{m_*} + k_0^{-\alpha} 2^{-\alpha m_*} \ind{m_* \leq M}.
\end{align*}
For some $k \geq k_0$, taking $m_* = \max \dk{m < M : k_m \leq k}$ so that $k_0 2^{m_*} \asymp k$, we can write
\begin{equation}
  \label{eq:Proof_Adaptive_BandBound}
  \norm{\hat{\Sigma}_{\mca{T}_{k_{M-1}}} - \Sigma_{\mca{T}_{k_{M-1}}}}
  \lesssim \sqrt{\frac{k}{n}} + \xk{\frac{M d k^2}{ \rho n^2}}^{\hf} + k^{-\alpha} \ind{k \leq k_{M-1}}.
\end{equation}
Combining \cref{eq:Proof_Adaptive_OffBandBound} and
\cref{eq:Proof_Adaptive_BandBound} gives, on $E$,
\begin{equation}
  \label{eq:Proof_Ada_MainBound}
  \E \zk{\bm{1}_E \norm{ \hat{\Sigma} - \Sigma}^2} \lesssim \frac{k}{n} + \frac{M d k^2}{\rho n^2}
  + k^{-2\alpha} \cdot \ind{k \leq k_{M-1}} + k_{M-1}^{-2\alpha} \cdot \ind{k_{M} < d}
  \quad  \forall k \geq k_0.
\end{equation}

To optimize \cref{eq:Proof_Ada_MainBound}, use
$k_0 \asymp \log n$, $k_{M-1} \asymp \min(d,n)$, and
\(M \lesssim \log n \wedge \log d \lesssim \log n\).
The assumptions \(d\gtrsim\log n\) and
\(\rho n^2/d\gtrsim(\log n)^{2\alpha+3}\) ensure that each candidate
scale in the optimization is at least of order \(k_0\).
Taking
\begin{equation*}
  k \asymp \min \xk{ n^{\frac{1}{2\alpha+1}}, \xk{\frac{\rho n^2}{M d}}^{\frac{1}{2\alpha+2}}, d},
\end{equation*}
we find that
\begin{align*}
  \E \zk{\bm{1}_E \norm{ \hat{\Sigma} - \Sigma}^2}
  &\lesssim n^{-\frac{2\alpha}{2\alpha+1}}\wedge \frac{d}{n} +
  \xk{\frac{d \log n}{\rho n^2  }}^{\frac{\alpha}{\alpha+1}} \wedge \frac{d^3 \log n}{\rho n^2}
\end{align*}

It remains to control the exceptional event.
By Cauchy--Schwarz,
\begin{equation*}
  \E \zk{\bm{1}_{E^\complement} \norm{\hat{\Sigma} - \Sigma}^2}
  \leq \zk{\bbP(E^\complement) \E \norm{\hat{\Sigma} - \Sigma}^4}^{\hf}
  \lesssim \zk{d^{-8} \exp(-k_0) \E \norm{\hat{\Sigma} - \Sigma}^4}^{\hf}.
\end{equation*}
Moreover,
\begin{equation*}
  \E \norm{\hat{\Sigma} - \Sigma}^4 \lesssim \E \norm{\hat{\Sigma}}^4 + \norm{\Sigma}^4 \lesssim 1 + \E \norm{\hat{\Sigma}}_F^4.
\end{equation*}
Decomposing $\hat{\Sigma}$ into blocks and using \cref{prop:BlockRoughBound} gives
\begin{align*}
  \norm{\hat{\Sigma}}_F^4
  & \lesssim d^2 {\sum_{B} \norm{\hat{\Sigma}_{B}}_F^4}
  \lesssim d^2 \sum_B \zk{1 + \xk{\frac{M d^3}{\rho n^2}}^2} \\
  & \lesssim d^4 \zk{1 + \xk{\frac{M d^3}{\rho n^2}}^2},
\end{align*}
so
\begin{equation*}
  \E \zk{\bm{1}_{E^\complement} \norm{\hat{\Sigma} - \Sigma}^2}
  \lesssim \exp(-k_0/2) \xk{1 + {\frac{M d}{\rho n^2}}}
  \lesssim \frac{1}{n} \xk{1 + {\frac{M d}{\rho n^2}}},
\end{equation*}
which is dominated by the main-event rate.
Combining the two events concludes the proof.

\subsubsection{Frobenius Loss: Proof of \cref{thm:DP_Cov_Adaptive_Frob}}

Since \(M\lesssim d\), the standing regime gives
\(M/(\rho n^2)=o(1)\).

\begin{lemma}[Frobenius Thresholding]
  \label{lem:BlockThresholding_Frob}
  Under the same setting as in \cref{lem:BlockThresholding}, we have
  \begin{equation}
    \E \norm{\hat{\Sigma}[\Gamma_{m,l}] - \Sigma[\Gamma_{m,l}]}^2_{F}
    \lesssim \min(\norm{\Sigma[\Gamma_{m,l}]}^2_F, k_m \tau_m^2) + d^{-2} \exp(-k_m/4),
  \end{equation}
  and
  \begin{equation}
    \E \norm{\hat{\Sigma}[B] - \Sigma[B]}_F^2
    \lesssim k_0 \tau_0^2 + d^{-2} \exp(-k_0/4)
    \qquad
    \forall B\in\left\{B_{k_0;l,0},B_{k_0;l,1} \right\}.
  \end{equation}
\end{lemma}
\begin{proof}
  The proof of \cref{lem:BlockThresholding} gives an event $E$ with
  probability at least $1 - C d^{-10} \exp(-k_m)$ on which
  \begin{equation*}
    \bm{1}_E \norm{\hat{\Sigma}[\Gamma_{m,l}] - \Sigma[\Gamma_{m,l}]}^2_F \lesssim \min(\norm{\Sigma[\Gamma_{m,l}]}^2_F, k_m \tau_m^2).
  \end{equation*}
  On $E^\complement$, \cref{prop:BlockRoughBound} gives
  \begin{align*}
    \E \zk{\bm{1}_{E^\complement} \norm{\hat{\Sigma}[\Gamma_{m,l}] - \Sigma[\Gamma_{m,l}]}^2_F}
    &\lesssim \zk{\bbP(E^\complement) \E \norm{\hat{\Sigma}[\Gamma_{m,l}] - \Sigma[\Gamma_{m,l}]}^4}^{\hf} \\
    & \lesssim k_m^2 (1+\sigma_B^2) d^{-5} \exp(-k_m/2) \\
    & \lesssim d^{-2} \exp(-k_m/4)
  \end{align*}
  Combining the two bounds proves the claim for \(\Gamma_{m,l}\).
  The proof for \(B\in\{B_{k_0;l,0},B_{k_0;l,1}\}\) is similar using \cref{lem:BlockBound}.
\end{proof}

With the convention \(\Gamma_{0,l}=B_{k_0;l,1}\), the Frobenius error decomposes as
\begin{align*}
  \E \norm{ \hat{\Sigma} - \Sigma}_F^2
  &=
  \sum_{l=1}^{N_0}
  \E\norm{\hat{\Sigma}[B_{k_0;l,0}]-\Sigma[B_{k_0;l,0}]}_F^2\\
  &\quad+
  2\sum_{m=0}^{M-1} \sum_{l<N_m}
  \E\norm{\hat{\Sigma}[\Gamma_{m,l}]-\Sigma[\Gamma_{m,l}]}_F^2
  +
  \E\norm{\Sigma_{\mca{T}_{k_{M-1}}^\complement}}_F^2.
\end{align*}
For $\Sigma \in \mca{H}_\alpha$,
\begin{equation*}
  \norm{\Sigma[\Gamma_{m,l}]}^2_F \lesssim k_m^{-2\alpha},\quad 0 \leq m \leq M-1,
\end{equation*}
and similarly
\begin{equation*}
  \norm{\Sigma_{\mca{T}_{k_{M-1}}^\complement}}_F^2 \lesssim d k_M^{-(2\alpha+1)} \asymp N_M k_M^{-2\alpha},
\end{equation*}
so the off-band remainder can be absorbed into the \( m = M \) term in the tail sum below.
Using \cref{lem:BlockThresholding_Frob} and splitting the sums at
$0 \leq m_* \leq M-1$ gives
\begin{align*}
  \E \norm{ \hat{\Sigma} - \Sigma}_F^2
  &\lesssim \sum_{0 \leq m \leq m_*} N_m k_m \tau_m^2 + \sum_{m_* < m \leq M} N_m k_{m}^{-2\alpha} + R \\
  &\lesssim d \sum_{0 \leq m \leq m_*} \tau_m^2 + d \sum_{m_* < m \leq M} k_{m}^{-(2\alpha+1)} + R,
\end{align*}
where
\begin{equation*}
  R \coloneqq d^{-2} \sum_{m=0}^{M-1} N_m \exp(-k_m/4).
\end{equation*}
Using \( N_m \lesssim d/k_m \), \( k_m = 2^m k_0 \), and \( k_0 \asymp \log n \), we have
\begin{equation*}
  R \lesssim d^{-1} \sum_{m=0}^{M-1} 2^{-m} \exp(-2^m k_0/4)
  \lesssim d^{-1} \exp(-k_0/4).
\end{equation*}
With \(k\asymp k_02^{m_*}\), the estimates for $\tau_m$ and $k_m$ give
\begin{align*}
  \frac{1}{d} \E \norm{ \hat{\Sigma} - \Sigma}_F^2
  &\lesssim
  \frac{k_0}{n}2^{m_*}
  +\frac{Mdk_0^2}{\rho n^2}2^{2m_*}\\
  &\quad+
  k_0^{-(2\alpha+1)}2^{-(2\alpha+1)m_*} \ind{m_*\leq M}
  +d^{-2} e^{-k_0/4}\\
  &\lesssim
  \frac{k}{n}
  +\frac{Mdk^2}{\rho n^2}
  +k^{-(2\alpha+1)} \ind{k\leq k_{M-1}}.
\end{align*}
Under the theorem assumption \( \rho n^2 / d \gtrsim (\log n)^{2\alpha+4} \), we have \( M / (\rho n^2) \lesssim (\log n)^{-(2\alpha+3)} \),
so the exponential remainder above is negligible compared to the main terms.
Together with \(d\gtrsim\log n\), the same assumption ensures that each
candidate scale in the optimization is at least of order \(k_0\).
Choosing
\begin{equation*}
  k \asymp \min \xk{ n^{\frac{1}{2\alpha+2}}, \xk{\frac{\rho n^2}{M d}}^{\frac{1}{2\alpha+3}}, d}
\end{equation*}
gives the desired result.

  \clearpage
  \section{Multiscale Block Estimator for \texorpdfstring{\(\mca{H}_\alpha\) and \(\mca{G}_\alpha\)}{H-alpha and G-alpha}}
\label{sec:Proofs_GH_UpperBounds}

In this section we prove upper bounds for the multiscale block estimator when the smoothness parameter is known.
The multiscale block construction handles covariance classes whose off-diagonal decay is imposed either entrywise or through row tails.
In \cref{subsec:Proofs_GH_LocalBlocks}, at each dyadic scale, we establish a greedy block oracle for the pointwise-decay class \(\mca{H}_\alpha\) under squared operator loss, together with profile and peeling block oracles for the row-tail class \(\mca{G}_\alpha\) under squared operator and normalized squared Frobenius losses, respectively.
Privacy accounting and dyadic aggregation then convert these local oracle bounds into the guarantees for known smoothness in \cref{thm:DPCov_GH_Operator} and \cref{thm:DPCov_G_Frobenius}.

Throughout this section, \(C>0\) denotes a sufficiently large constant that
may change from line to line, and we set \(L=CK^2\log n\).
In every local result below, we additionally assume that \(q\),
\(\rho_B^{-1}\), and \(\beta_B^{-1}\), when present, are bounded by fixed
powers of \(n\), as required for the dyadic construction.
The results below under the operator norm require the following condition on
the local scale:
\begin{equation}
  k+\log(q+1)
  \leq
  \frac{n}{(\log n)^C}.
  \label{eq:Supp_DPCov_GH_OperatorScaleCondition}
\end{equation}
The concentration, net, and private-selection bounds used below are collected
in \cref{lem:SubGaussianProduct}, \cref{lem:SphereNet}, and
\cref{lem:FiniteExponentialMechanism}.

\subsection{Privacy Guarantee}

\begin{proposition}
  \label{prop:Supp_DPCov_GH_LocalPrivacy}
  For a fixed block \(B\), each call to
  \(\mf{DPCovBlock}\), the greedy release in
  \cref{alg:DP_Cov_H_GreedyRelease}, the formal profile release in
  \cref{alg:Supp_DPCov_G_ProfileReleaseFormal}, or the peeling release in
  \cref{alg:DP_Cov_G_FrobeniusRelease}, with public tuning parameters and
  budget \(\rho_B\), is \(\rho_B\)-zCDP.
  Consequently, applying any one of these procedures to \(q\) blocks from
  the same sample is jointly \(q\rho_B\)-zCDP.
\end{proposition}

\begin{proof}
  The dense release is \(\rho_B\)-zCDP by
  \cref{prop:DP_Cov_Helper_Block_Privacy}.
  For the greedy release, conditionally on the preceding transcript, the
  score and coefficient sensitivities at every round are both at most
  \(2L/n\).
  The exponential-mechanism selection and Gaussian coefficient release each
  cost \(\rho_B/(2r)\)-zCDP, so adaptive composition over the \(r\) rounds
  costs \(\rho_B\).

  For the profile release, suppose that the public profile collection has
  cardinality \(D\).
  Each round in a fixed profile costs
  \(\rho_B/(D T_m)\)-zCDP; hence the \(T_m\) rounds cost \(\rho_B/D\), and
  composition over the \(D\) profiles costs \(\rho_B\).
  For the peeling release, put \(s=k r\).
  The \(s\) exponential-mechanism selections together cost
  \(\rho_B/2\)-zCDP.
  Conditionally on the selection transcript, the selected \(s\)-vector has
  Euclidean sensitivity at most \(2L\sqrt{s}/n\), so its Gaussian release
  costs the remaining \(\rho_B/2\).
  All subsequent averaging, feasible-set construction, and projection are
  post-processing.
  Using fresh randomization at every conditional release, adaptive composition
  gives the per-block guarantees above, while sequential
  composition over the \(q\) released blocks gives \(q\rho_B\)-zCDP.
\end{proof}

\begin{proposition}
  \label{prop:Supp_DPCov_GH_GlobalPrivacy}
  Every estimator constructed from \cref{alg:DP_Cov_DyadicTemplate} using
  either the zero release or one of the local releases in
  \cref{prop:Supp_DPCov_GH_LocalPrivacy} is \(\rho\)-zCDP.
\end{proposition}

\begin{proof}
  The dense center uses budget \(\rho/4\) and is \(\rho/4\)-zCDP by
  \cref{prop:DP_Cov_BlockDiagonal_Privacy}.
  At a dyadic outer scale \(k\), there are
  \(\abs{\mca{B}_k}=3(N_{2k}-1)\leq3N_{2k}\) blocks, each released with
  budget \(\rho_k=\rho/\xkx{4J(k_0)N_{2k}}\).
  Thus \cref{prop:Supp_DPCov_GH_LocalPrivacy} bounds the privacy cost of one
  scale by \(3\rho/\xkx{4J(k_0)}\).
  There are at most \(J(k_0)\) active scales, so the center and outer
  releases together cost at most
  \[
    \frac{\rho}{4}
    +
    J(k_0)\frac{3\rho}{4J(k_0)}
    =
    \rho.
  \]
  Symmetric completion, block assignment, and spectral projection are
  post-processing.
\end{proof}

\subsection{Local Block Oracles}
\label{subsec:Proofs_GH_LocalBlocks}

\subsubsection{A Robust Dense Block Deviation Bound}

\begin{proposition}[Robust dense block deviation bound]
  \label{prop:Supp_DPCov_GH_DenseDeviation}
  Let \(A_1,\ldots,A_q\) be \(k\times k\) covariance cross-blocks from the
  same sample, and fix \(0<\rho_B\leq\rho\) and \(0<\beta_B<1/4\).
  For each \(b\leq q\), set \(R^2=CK^2(k+\log n)\) and
  \(\widetilde{A}_b=\mf{DPCovBlock}(X,B_b;\rho_B,R)\).

  If \cref{eq:Supp_DPCov_GH_OperatorScaleCondition} holds, set
  \[
    R_{\mathrm{dense}}^2
    =
    C\log^C n
    \xk{
      \frac{k}{n}+\frac{k^3}{\rho_B n^2}
    }.
  \]
  Then
  \[
    \bbP\dk{
      \max_{b\leq q}\norm{\widetilde{A}_b-A_b}
      \leq R_{\mathrm{dense}}
    }
    \geq1-\beta_B,
  \]
  and the same rate holds in expectation.

  At every scale, set
  \begin{equation}
    \xkx{R_{\mathrm{dense}}^{\mathrm{F}}(\rho_B)}^2
    =
    C\log^C n
    \xk{
      \frac{k^2}{n}+\frac{k^4}{\rho_B n^2}
    }.
    \label{eq:DP_Cov_G_FrobeniusDeviationDense}
  \end{equation}
  Then
  \[
    \bbP\dk{
      \max_{b\leq q}\norm{\widetilde{A}_b-A_b}_F
      \leq R_{\mathrm{dense}}^{\mathrm{F}}(\rho_B)
    }
    \geq1-\beta_B,
  \]
  and the same rate holds in expectation.
\end{proposition}

\begin{proof}[Proof of \cref{prop:Supp_DPCov_GH_DenseDeviation}]
  The fixed-block decomposition and the corresponding truncation,
  concentration, and Gaussian-noise bounds are established in Appendix A;
  see \cref{lem:BlockBound}, \cref{prop:A__BlockBound_CovTruncation}, and
  \cref{lem:BlockExpectationBound}.  Here we only use three features of that
  argument:
  the privacy parameter is \(\rho_B\), the radius depends on \(k+\log n\),
  and the bounds hold simultaneously over the \(q\) blocks.

  With \(R\) as above, write
  \[
    Y_{i,b}
    =
    \chi_R(X_{i,I_b})\chi_R(X_{i,J_b})^\T,
    \qquad
    A_b^R=\E Y_{i,b},
    \qquad
    \overline{A}_b=\frac{1}{n}\sum_{i=1}^nY_{i,b}.
  \]
  The released block is \(\widetilde{A}_b=\overline{A}_b+G_b\), where
  the entries of \(G_b\) have variance
  \(\sigma_B^2=18R^4/(\rho_B n^2)\).

  \emph{Operator loss.}
  For \(t\geq0\), put \(s_t=k+\log(q+1)+t\).
  The truncation-bias part of the Appendix A argument applies to the present
  radius.  A union bound over the \(q\) blocks then gives
  \[
    \max_{b\leq q}\norm{A_b^R-A_b}\lesssim n^{-c},
    \qquad
    \max_{b\leq q}\norm{A_b^R-A_b}_F^2\lesssim k n^{-2c}.
  \]
  The empirical-fluctuation and Gaussian-noise bounds, again with a union
  bound over the \(q\) blocks, give
  \[
    \bbP\dk{
      \max_{b\leq q}\norm{\overline{A}_b-A_b^R}
      >
      C\xk{
        \sqrt{\frac{s_t}{n}}+\frac{s_t}{n}
      }
    }
    \leq2e^{-t}.
  \]
  In addition, the Gaussian-noise bound from Appendix A gives
  \[
    \bbP\dk{
      \max_{b\leq q}\norm{G_b}
      >C\sigma_B\sqrt{s_t}
    }
    \leq2e^{-t}.
  \]
  Use the decomposition
  \[
    \widetilde{A}_b-A_b
    =
    (\overline{A}_b-A_b^R)+(A_b^R-A_b)+G_b.
  \]
  Set \(t=\log(4/\beta_B)\).
  Under \cref{eq:Supp_DPCov_GH_OperatorScaleCondition}, the linear Bernstein
  term is controlled, and the clipping bias is negligible.
  Combining these bounds with the displayed decomposition proves the operator
  deviation bound.  Integrating these tails, as in
  \cref{lem:BlockExpectationBound}, gives the expected risk.

  \emph{Frobenius loss.}
  Put \(u_t=\log(q k^2+1)+t\).
  Applying the coordinatewise empirical bound gives
  \[
    \bbP\dk{
      \max_{\substack{b\leq q\\u,v\leq k}}
      \abs{(\overline{A}_b-A_b^R)_{uv}}
      >
      C\xk{\sqrt{\frac{u_t}{n}}+\frac{u_t}{n}}
    }
    \leq2e^{-t}.
  \]
  The Gaussian chi-square step from Appendix A, followed by a union bound,
  gives
  \[
    \bbP\dk{
      \max_{b\leq q}\norm{G_b}_F^2
      >C\sigma_B^2\xkx{k^2+\log(q+1)+t}
    }
    \leq e^{-t}.
  \]
  Since \(\norm{M}_F^2\leq k^2\max_{u,v}\abs{M_{uv}}^2\), use the displayed
  decomposition, set \(t=\log(4/\beta_B)\), and apply
  \(R^4\lesssim_{\log}k^2\).  This proves the Frobenius deviation bound
  without requiring \cref{eq:Supp_DPCov_GH_OperatorScaleCondition};
  Integrating the two stochastic tail bounds, as in
  \cref{lem:BlockExpectationBound}, gives the expected risk.
\end{proof}

\subsubsection{\texorpdfstring{\(\mca{H}_\alpha\)}{H-alpha} under Operator Loss: Greedy Blocks}

\begin{proposition}[Greedy block oracle for pointwise-decay blocks]
  \label{prop:Supp_DPCov_H_GreedyOracle}
  Assume that \cref{eq:Supp_DPCov_GH_OperatorScaleCondition} holds.
  Let \(A_1,\ldots,A_q\in\mca{E}_k(a)\), give each block budget
  \(0<\rho_B\leq\rho\), and fix \(0<\beta_B<1/4\).
  For every public \(1\leq r\leq k\), let
  \(\widehat{A}_{b,r}^{\mathrm{gr}}\) be the output of
  \cref{alg:DP_Cov_H_GreedyRelease} on block \(B_b\) when its rank-selection
  step is replaced by this fixed \(r\).
  With probability at least \(1-\beta_B\), these outputs satisfy
  \[
    \max_{b\leq q}
    \norm{\widehat{A}_{b,r}^{\mathrm{gr}}-A_b}^2
    \lesssim_{\log}
      \frac{a^2}{r}
      +\frac{k+\log(q+1)}{n}
      +\frac{k^2 r}{\rho_B n^2}.
  \]
  The corresponding expected risk satisfies
  \begin{equation}
    \E\max_{b\leq q}
    \norm{\widehat{A}_{b,r}^{\mathrm{gr}}-A_b}^2
    \lesssim_{\log}
    \frac{a^2}{r}
    +
    \frac{k+\log(q+1)}{n}
    +
    \frac{k^2 r}{\rho_B n^2}.
    \label{eq:Supp_DPCov_H_GreedyPointRisk}
  \end{equation}
\end{proposition}

\begin{proof}[Proof of \cref{prop:Supp_DPCov_H_GreedyOracle}]
  Fix \(r\) and use the notation
  \(B_t,H_t,\widehat{c}_t,\overline{B}_r\) from
  \cref{alg:DP_Cov_H_GreedyRelease}.

  \emph{Uniform stochastic bounds.}
  Put
  \[
    \tau_r^2
    =
    C(\log n)^C
    \xk{
      \frac{k+\log(q+1)}{n}
      +
      \frac{k^2 r}{\rho_B n^2}
    }.
  \]
  For \(H=\xi uv^\T\), write
  \[
    \overline{m}_{b,H}
    =
    \frac{1}{n}\sum_{i=1}^n
      \chi_L \xk{
      \xi(u^\T X_{i,I_b})(v^\T X_{i,J_b})
    },
    \qquad
    m_{b,H}=\langle A_b,H\rangle_F.
  \]
  Since there are at most \(2q9^{2k}\) block-direction pairs,
  \cref{lem:SubGaussianProduct}, \cref{lem:SphereNet}, scalar Bernstein's
  inequality, and a union bound give
  \begin{equation}
    \delta_B^2
    \coloneqq
    \max_{b,H}
    \abs{\overline{m}_{b,H}-m_{b,H}}^2
    \leq
    \tau_r^2.
    \label{eq:Proof_H_GreedyEmpirical}
  \end{equation}
  outside an event of probability at most \(\beta_B/4\).
  Condition \cref{eq:Supp_DPCov_GH_OperatorScaleCondition} absorbs the square
  of the linear Bernstein term.

  Apply the conditional exponential-mechanism utility bound in
  \cref{lem:FiniteExponentialMechanism} with failure probability
  \(\beta_B/(4qr)\) at every adaptive selection.
  Although the directions are selected adaptively, conditioning on the
  transcript before each round allows the conditional exponential-mechanism
  bound to be applied.  The tower property and a union bound over the
  \(qr\) block-round pairs then yield the simultaneous event
  \[
    \max_H s_t(X,H)-s_t(X,H_t)
    \lesssim
    \frac{L}{n}\sqrt{\frac{r}{\rho_B}}
    \xkx{k+\log(q r/\beta_B+2)}.
  \]
  The signed net satisfies
  \[
    \max_{H\in\mca{W}_k}
    \langle A_b-B_{t-1},H\rangle_F
    \geq
    \frac{1}{2}\norm{A_b-B_{t-1}}.
  \]
  Thus, for
  \(R_{t-1}=A_b-B_{t-1}\) and
  \(c_t=\langle R_{t-1},H_t \rangle_F\),
  \begin{equation}
    c_t
    \geq
    \frac{1}{2}\norm{R_{t-1}}-\tau_r.
    \label{eq:Proof_H_GreedyDirection}
  \end{equation}
  A Gaussian union bound for the noisy coefficients, with failure
  probability at most \(\beta_B/4\), gives
  \begin{equation}
    \max_{b,t}
    \abs{\widehat{c}_t-c_t}
    \leq
    \tau_r.
    \label{eq:Proof_H_GreedyCoefficient}
  \end{equation}
  Thus \cref{eq:Proof_H_GreedyDirection} and
  \cref{eq:Proof_H_GreedyCoefficient}
  hold simultaneously for all blocks and rounds; the total failure
  probability is at most \(3\beta_B/4<\beta_B\).

  \emph{Residual-energy inequality.}
  Since \(\norm{H_t}_F=1\), the exact energy identity
  \[
    \norm{R_t}_F^2
    =
    \norm{R_{t-1}}_F^2-c_t^2+
    \abs{\widehat{c}_t-c_t}^2
  \]
  implies
  \[
    \sum_{t=1}^r c_t^2
    \leq
    \norm{A_b}_F^2+
    \sum_{t=1}^r
    \abs{\widehat{c}_t-c_t}^2.
  \]
  Moreover, \cref{eq:Proof_H_GreedyDirection} gives
  \(\norm{R_{t-1}}^2 \lesssim c_t^2+\tau_r^2\).
  Hence Jensen's inequality and the preceding energy identity yield
  \begin{equation}
    \begin{aligned}
      \norm{A_b-\overline{B}_r}^2
      &\leq
      \frac{1}{r}\sum_{t=1}^r\norm{R_{t-1}}^2\\
      &\lesssim
      \frac{1}{r}\sum_{t=1}^r c_t^2+\tau_r^2\\
      &\leq
      \frac{\norm{A_b}_F^2}{r}
      +
      \frac{1}{r}\sum_{t=1}^r
      \abs{\widehat{c}_t-c_t}^2
      +
      \tau_r^2\\
      &\lesssim
      \frac{\norm{A_b}_F^2}{r}+\tau_r^2.
    \end{aligned}
    \label{eq:Proof_H_GreedyRisk}
  \end{equation}
  Since \(\norm{A_b}\leq M_0\), the radial projection increases this error
  by at most a numerical factor.

  \emph{Expected risk.}
  For the expected risk, let \(u\geq0\) and put
  \[
    \tau_r^2(u)
    =
    C(\log n+u)^C
    \xk{
      \frac{k+\log(q+1)}{n}
      +
      \frac{k^2 r}{\rho_B n^2}
    }.
  \]
  Use failure probability \(e^{-u}/(4qr)\) at every conditional
  selection and the same tail parameter in the Bernstein and Gaussian
  bounds.
  Repeating the preceding argument gives
  \[
    \bbP\dk{
      \max_{b\leq q}
      \norm{\widehat{A}_{b,r}^{\mathrm{gr}}-A_b}^2
      >
      C\xk{\frac{a^2}{r}+\tau_r^2(u)}
    }
    \leq e^{-u}.
  \]
  Integrating this inequality over \(u\geq0\) proves
  \cref{eq:Supp_DPCov_H_GreedyPointRisk}.
\end{proof}

\subsubsection{\texorpdfstring{\(\mca{G}_\alpha\)}{G-alpha} under Operator Loss: Profile Blocks}

The argument is organized into five steps.
In \cref{para:Supp_DPCov_G_ProfileDualGeometry}, we reduce the operator norm
to a finite collection of dyadic profile discrepancies.
In \cref{para:Supp_DPCov_G_ProfileFiniteProbes}, we replace the infinite probe
classes by finite nets and obtain the corresponding entropy bound.
In \cref{para:Supp_DPCov_G_ProfilePrimalFitting}, we construct a center for
each profile through entropy-regularized updates and control the resulting
regret.
In \cref{para:Supp_DPCov_G_ProfilePrivateFits}, we define the clipped
measurements, adaptive horizons, private probe-selection rule, common feasible
set, and formal estimator.
Finally, \cref{para:Supp_DPCov_G_ProfileOracleGuarantee} combines these
ingredients to prove the profile-fitting and block-oracle bounds.

\begin{proposition}[Profile block oracle for row-tail blocks]
  \label{prop:Supp_DPCov_G_ProfileOracle}
  Under the standing sub-Gaussian model, let
  \(B_b=I_b\times J_b\), \(b\leq q\), be square index blocks of side length
  \(k\), with covariance blocks \(A_b\in\mca{K}_k(a)\), where \(a=O(1)\).
  Assume that \cref{eq:Supp_DPCov_GH_OperatorScaleCondition} holds.
  Give each block budget
  \(0<\rho_B\leq\rho\), and fix
  \(0<\beta_B\leq n^{-2}\).
  With probability at least \(1-\beta_B\), the point outputs satisfy
  \[
    \max_{b\leq q}
    \norm{\widehat{A}_b^{\mathrm{prof}}-A_b}^2
    \lesssim_{\log}
      \frac{k+\log(q+1)}{n}
      +\frac{a k}{n\sqrt{\rho_B}}.
  \]
  Their expected risk satisfies
  \begin{equation}
    \E\max_{b\leq q}
    \norm{\widehat{A}_b^{\mathrm{prof}}-A_b}^2
    \lesssim_{\log}
    \frac{k+\log(q+1)}{n}
    +
    \frac{a k}{n\sqrt{\rho_B}}.
    \label{eq:Supp_DPCov_G_ProfilePointRisk}
  \end{equation}
\end{proposition}

\paragraph{Dual Geometry}
\label{para:Supp_DPCov_G_ProfileDualGeometry}
Let
\(\mca{D}_k=\dkx{2^j:0\leq j\leq\ceil{\log_2 k}}\) and
\(D=\abs{\mca{D}_k}\).
The following lemma reduces operator loss to these dyadic profile
discrepancies.

\begin{lemma}[Dual-probe geometry]
  \label{lem:Proof_G_ProfileGeometry}
  For \(s\in\mca{D}_k\), define
  \[
    \mca{V}_s
    =
    \dk{
      u\in\R^k:
      \norm{u}_0\leq s,\
      \norm{u}_2 \leq1,\
      \norm{u}_\infty \leq\sqrt{2/s}
    }.
  \]
  For \(m=2^t \in\mca{D}_k\), let
  \[
    \mca{H}_m
    =
    \dk{
      \xi uv^\T:
      u\in\mca{V}_s,\ v\in\mca{V}_r,\
      \xi\in\dkx{-1,1},\
      s,r\in\mca{D}_k,\ \max(s,r)=m
    }.
  \]
  Then every \(E\in\R^{k\times k}\) satisfies
  \begin{equation}
    \norm{E}
    \leq
    D^2
    \max_{m\in\mca{D}_k}
    \sup_{H\in\mca{H}_m}
    \langle E,H\rangle_F.
    \label{eq:Proof_G_WitnessDecomposition}
  \end{equation}
\end{lemma}

\begin{proof}
  Let \(x,y\in\R^k\) be unit left and right singular vectors such that
  \[
    \norm{E}=x^\T E y.
  \]
  For a unit vector \(w\), write
  \(\abs{w}_{(1)}\geq\cdots\geq\abs{w}_{(k)}\) for the decreasing
  rearrangement of its coordinate magnitudes.
  For every \(j\in[k]\),
  \[
    j\abs{w}_{(j)}^2
    \leq
    \sum_{\ell=1}^j\abs{w}_{(\ell)}^2
    \leq1,
  \]
  and hence \(\abs{w}_{(j)}\leq j^{-1/2}\).

  For every integer \(h\geq0\) with \(2^h\leq k\), let \(w^{(h)}\)
  retain the coordinates of \(w\) whose ranks lie between \(2^h\) and
  \(\min\xkx{2^{h+1}-1,k}\), and set all other coordinates to zero.
  These groups are disjoint and exhaustive, so
  \[
    w=\sum_{h:2^h\leq k}w^{(h)}.
  \]
  Moreover,
  \[
    \norm{w^{(h)}}_0\leq2^h,
    \qquad
    \norm{w^{(h)}}_2\leq1,
    \qquad
    \norm{w^{(h)}}_\infty
    \leq
    \abs{w}_{(2^h)}
    \leq
    2^{-h/2}
    \leq
    \sqrt{\frac{2}{2^h}}.
  \]
  Thus \(w^{(h)}\in\mca{V}_{2^h}\).
  Applying this decomposition to both singular vectors gives
  \[
    x=\sum_{h:2^h\leq k}x^{(h)},
    \qquad
    y=\sum_{\ell:2^\ell\leq k}y^{(\ell)},
  \]
  with \(x^{(h)}\in\mca{V}_{2^h}\) and
  \(y^{(\ell)}\in\mca{V}_{2^\ell}\).
  Each sum has at most \(D\) terms.

  Expanding the bilinear form and applying the triangle inequality yield
  \[
    \norm{E}
    \leq
    \sum_{h:2^h\leq k}
    \sum_{\ell:2^\ell\leq k}
    \abs{\langle E,x^{(h)}(y^{(\ell)})^\T\rangle_F}.
  \]
  For each pair \((h,\ell)\), choose
  \(\xi_{h\ell}\in\dkx{-1,1}\) to match the sign of the corresponding
  inner product, with either sign when it is zero.
  Since
  \[
    \xi_{h\ell}x^{(h)}(y^{(\ell)})^\T
    \in
    \mca{H}_{2^{\max(h,\ell)}},
  \]
  every summand is bounded by
  \[
    \max_{m\in\mca{D}_k}
    \sup_{H\in\mca{H}_m}\langle E,H\rangle_F.
  \]
  There are at most \(D^2\) pairs, which proves
  \cref{eq:Proof_G_WitnessDecomposition}.
\end{proof}

Accordingly, for \(A,A'\in\mca{K}_k(a)\) and
\(m\in\mca{D}_k\), define the profile discrepancy
\[
  d_m(A,A')
  =
  \sup_{H\in\mca{H}_m}
  \abs{\langle A-A',H\rangle_F}.
\]
Then \cref{lem:Proof_G_ProfileGeometry} gives
\[
  \norm{A-A'}
  \leq
  D^2\max_{m\in\mca{D}_k}d_m(A,A').
\]

\paragraph{Finite Probe Classes}
\label{para:Supp_DPCov_G_ProfileFiniteProbes}
We replace each probe class by a finite set.
Take internal Euclidean \(\delta\)-nets, with
\(\delta=n^{-20}/4\), over every support in each profile, and let
\(\widetilde{\mca{H}}_m\) be the resulting set of signed rank-one probes.
For a support of size \(s\), the Euclidean unit ball has a
\(\delta\)-net of size at most \((1+2/\delta)^s\), while the support has
at most \(\binom{k}{s}\) choices.
Summing over the dyadic pairs \((s,r)\) with \(\max(s,r)=m\) therefore
gives
\begin{equation}
  \log\abs{\widetilde{\mca{H}}_m}
  \lesssim_{\log}
  m.
  \label{eq:Proof_G_ProfileEntropy}
\end{equation}
For \(A,A'\in\mca{K}_k(a)\), Schur's inequality gives
\(\norm{A-A'}\leq2a\).
Approximating the two factors of \(H=\xi uv^\T\) on their fixed
supports therefore changes its pairing with \(A-A'\) by at most
\(4a\delta=a n^{-20}\).
Hence
\begin{equation}
  d_m(A,A')
  \leq
  \max_{H\in\widetilde{\mca{H}}_m}
  \abs{\langle A-A',H\rangle_F}
  +
  a n^{-20}.
  \label{eq:Proof_G_ProfileNetApproximation}
\end{equation}
The term \(a n^{-20}\) will be included in the final radius.

\paragraph{Primal Fitting}
\label{para:Supp_DPCov_G_ProfilePrimalFitting}
The probe class for each profile is finite.  It remains to construct a matrix
center that performs well against the probes selected from that class.
For a fixed profile \(m\) and horizon \(T\), view the matrices in
\(\mca{K}_k(a)\) as decisions in an online linear optimization problem.
A signed lift with row and column slacks makes these constraints compatible
with entropy regularization, and the averaged iterates form the profile
center.  The deterministic regret bound in
\cref{lem:Proof_G_ProfileRegret} remains valid for adaptive probe sequences
and gives the fitting bound used below.

When \(a=0\), the feasible set contains only the zero matrix.
For \(a>0\), let \(\mca{Z}_k(a)\) consist of tuples
\(z=(P^+,P^-,r,c)\), where
\(P^+,P^-\in\R_+^{k\times k}\) and \(r,c\in\R_+^k\), satisfying
\[
  \begin{aligned}
    \sum_{j=1}^k(P^+_{ij}+P^-_{ij})+r_i&=a,
    &i&\in[k],\\
    \sum_{i=1}^k(P^+_{ij}+P^-_{ij})+c_j&=a,
    &j&\in[k].
  \end{aligned}
\]
The nonnegative matrices \(P^+\) and \(P^-\) represent the positive and negative parts of a candidate.
The slack variables \(r\) and \(c\) turn the row and column upper bounds into equalities with fixed total mass.
This representation lets us apply entropy regularization while preserving both sets of constraints.
Let \(\pi(z)=P^+-P^-\), so \(\pi(z)\in\mca{K}_k(a)\).
Put \(\phi(x)=x\log(x/a)\), with \(\phi(0)=0\), and define
\[
  \begin{split}
  \Phi(z)
  ={}&
  \sum_{i=1}^k
    \zk{
    \sum_{j=1}^k\xkx{\phi(P^+_{ij})+\phi(P^-_{ij})}
    +\phi(r_i)
    }\\
  &+
  \sum_{j=1}^k
    \zk{
    \sum_{i=1}^k\xkx{\phi(P^+_{ij})+\phi(P^-_{ij})}
    +\phi(c_j)
    }.
  \end{split}
\]
When \(z'\) lies in the relative interior, let
\[
  D_\Phi(z,z')
  =
  \Phi(z)-\Phi(z')-\langle\nabla\Phi(z'),z-z'\rangle
\]
denote the Bregman divergence.
Set
\[
  G(H)=(-H,H,0,0),
  \qquad
  \langle G(H),z\rangle=-\langle\pi(z),H\rangle_F.
\]
For a public horizon \(T\geq1\), let
\[
  \eta_{m,T}
  =
  \sqrt{\frac{2km\log(2k+1)}{T}}
\]
and initialize the lift \(z_1=z^\circ\) by
\[
  P^{\circ,+}_{ij}=P^{\circ,-}_{ij}
  =r_i^\circ=c_j^\circ
  =\frac{a}{2k+1}.
\]
After observing \(H_t\), update
\begin{equation}
  A_t=\pi(z_t),
  \qquad
  z_{t+1}
  =
  \argmin_{z\in\mca{Z}_k(a)}
  \xkx{
    \eta_{m,T} \langle G(H_t),z\rangle
    +D_\Phi(z,z_t)
  }.
  \label{eq:Proof_G_ProfilePrimalUpdate}
\end{equation}

\begin{lemma}[Entropy-regularized primal fitting]
  \label{lem:Proof_G_ProfileRegret}
  Fix \(a>0\), \(m\in\mca{D}_k\), and \(T\geq1\).
  For every adaptive sequence \(H_t\in\widetilde{\mca{H}}_m\), the
  deterministic update in \cref{eq:Proof_G_ProfilePrimalUpdate} satisfies
  \begin{equation}
    \max_{A\in\mca{K}_k(a)}
    \sum_{t=1}^T \langle A,H_t \rangle_F
    -
    \sum_{t=1}^T \langle A_t,H_t \rangle_F
    \lesssim
    a\sqrt{\frac{k T\log(2k)}{m}}.
    \label{eq:Proof_G_ProfileRegret}
  \end{equation}
\end{lemma}

\begin{proof}
  \emph{Signed lift.}
  Every \(A\in\mca{K}_k(a)\) has the lift
  \[
    P^+=A_+,
    \qquad
    P^-=A_-,
    \qquad
    r_i=a-\sum_j \abs{A_{ij}},
    \qquad
    c_j=a-\sum_i \abs{A_{ij}}.
  \]
  Thus it is enough to compare the lifted iterates with an arbitrary
  \(z\in\mca{Z}_k(a)\).

  \emph{Strong convexity.}
  Let \(x,y\in\R_+^N\) have total mass \(a\), and put
  \(y_\theta=y+\theta(x-y)\).
  Define \(\Phi_N(x)=\sum_{\ell=1}^N \phi(x_\ell)\).
  Taylor's formula and Cauchy--Schwarz give
  \[
    \begin{aligned}
    D_{\Phi_N}(x,y)
    &=
    \int_0^1(1-\theta)
    \sum_{\ell=1}^N
    \frac{(x_\ell-y_\ell)^2}{y_{\theta,\ell}}
    \,\mathrm{d}\theta\\
    &\geq
    \int_0^1(1-\theta)
    \frac{\xk{\sum_\ell \abs{x_\ell-y_\ell}}^2}
    {\sum_\ell y_{\theta,\ell}}
    \,\mathrm{d}\theta
    =
    \frac{1}{2a}\norm{x-y}_1^2.
    \end{aligned}
  \]
  The same inequality extends to the boundary by lower semicontinuity.
  For a difference \(h=(\Delta P^+,\Delta P^-,\Delta r,\Delta c)\), define
  \[
    R_i(h)
    =
    \sum_j \xkx{\abs{\Delta P^+_{ij}}+\abs{\Delta P^-_{ij}}}
    +\abs{\Delta r_i},
  \]
  \[
    C_j(h)
    =
    \sum_i \xkx{\abs{\Delta P^+_{ij}}+\abs{\Delta P^-_{ij}}}
    +\abs{\Delta c_j},
    \qquad
    \norm{h}_\diamond^2
    =
    \sum_i R_i(h)^2+\sum_j C_j(h)^2.
  \]
  Applying the preceding inequality to every row and column yields, for
  \(z'\) in the relative interior,
  \begin{equation}
    D_\Phi(z,z')
    \geq
    \frac{1}{2a}\norm{z-z'}_\diamond^2.
    \label{eq:Proof_G_ProfileEntropyStrongConvexity}
  \end{equation}

  \emph{Dual norm of a profile probe.}
  For \(g\) in the lifted product space, define
  \[
    \norm{g}_{\diamond,*}
    =
    \sup_{\norm{h}_\diamond\leq1}\abs{\langle g,h\rangle}.
  \]
  Consider \(H=\xi uv^\T \in\mca{H}_m\).
  For every \(h\), the row and column bounds give
  \[
    \begin{aligned}
    \abs{\langle G(H),h\rangle}
    &\leq
    \xk{\sum_i \max_j \abs{H_{ij}}^2}^{1/2}
    \xk{\sum_i R_i(h)^2}^{1/2},\\
    \abs{\langle G(H),h\rangle}
    &\leq
    \xk{\sum_j \max_i \abs{H_{ij}}^2}^{1/2}
    \xk{\sum_j C_j(h)^2}^{1/2}.
    \end{aligned}
  \]
  Hence the dual norm associated with \(\norm{\cdot}_\diamond\) satisfies
  \[
    \norm{G(H)}_{\diamond,*}
    \leq
    \norm{u}_2\norm{v}_\infty
    \wedge
    \norm{u}_\infty\norm{v}_2
    \leq
    \sqrt{\frac{2}{m}},
  \]
  because the larger support parameter of \(u\) and \(v\) equals \(m\).

  \emph{Mirror-descent inequality.}
  Compactness of \(\mca{Z}_k(a)\) and strict convexity of the objective give a
  unique minimizer at every update.
  The initialization \(z_1=z^\circ\) lies in the relative interior of
  \(\mca{Z}_k(a)\).
  Every subsequent minimizer remains there: otherwise, moving slightly
  toward \(z^\circ\) would decrease the objective because the right
  derivative of \(x\log x\) at zero is \(-\infty\).
  Combining the resulting first-order condition with the three-point
  identity for \(D_\Phi\) gives, for every \(z\in\mca{Z}_k(a)\),
  \[
    \eta_{m,T} \langle G(H_t),z_{t+1}-z\rangle
    \leq
    D_\Phi(z,z_t)-D_\Phi(z,z_{t+1})-D_\Phi(z_{t+1},z_t).
  \]
  Adding \(\eta_{m,T}\langle G(H_t),z_t-z_{t+1}\rangle\), applying
  \cref{eq:Proof_G_ProfileEntropyStrongConvexity}, and completing the square
  yield
  \[
    \begin{aligned}
    \eta_{m,T} \langle G(H_t),z_t-z\rangle
    \leq{}&
    D_\Phi(z,z_t)-D_\Phi(z,z_{t+1})\\
    &+
    \eta_{m,T} \norm{G(H_t)}_{\diamond,*}
    \norm{z_t-z_{t+1}}_\diamond
    -\frac{1}{2a}\norm{z_t-z_{t+1}}_\diamond^2\\
    \leq{}&
    D_\Phi(z,z_t)-D_\Phi(z,z_{t+1})
    +\frac{a\eta_{m,T}^2}{2}\norm{G(H_t)}_{\diamond,*}^2.
    \end{aligned}
  \]
  Summing over \(t\), using the dual-norm bound, and dividing by
  \(\eta_{m,T}\) give
  \begin{equation}
    \sum_{t=1}^T \langle G(H_t),z_t-z\rangle
    \leq
    \frac{D_\Phi(z,z^\circ)}{\eta_{m,T}}
    +\frac{a\eta_{m,T} T}{m}.
    \label{eq:Proof_G_ProfileOMDRegret}
  \end{equation}

  \emph{Initial divergence and conclusion.}
  Each row and column of a lifted comparator is a nonnegative vector of mass
  \(a\) with \(2k+1\) coordinates.
  Its relative entropy from the uniform vector is at most \(a\log(2k+1)\), so
  \[
    D_\Phi(z,z^\circ)
    \leq
    2a k\log(2k+1).
  \]
  Substituting the definition of \(\eta_{m,T}\) into \cref{eq:Proof_G_ProfileOMDRegret} gives
  \[
    \sum_{t=1}^T \langle G(H_t),z_t-z\rangle
    \leq
    2\sqrt2\,a
    \sqrt{\frac{k T\log(2k+1)}{m}}.
  \]
  Lift an arbitrary \(A\in\mca{K}_k(a)\) as above and use
  \(\langle G(H_t),z_t-z(A)\rangle=\langle A-A_t,H_t \rangle_F\).
  Maximizing over \(A\) proves \cref{eq:Proof_G_ProfileRegret}.
\end{proof}

\paragraph{Private Profile Fits}
\label{para:Supp_DPCov_G_ProfilePrivateFits}
The private probe selection uses the following number of rounds and privacy
allocation.
For a square block \(B=I\times J\) and \(H=\xi uv^\T\), define the clipped
empirical dual measurement
\[
  \widehat{\mu}_B(H)
  =
  \frac{1}{n}\sum_{i=1}^n
    \chi_L \xk{
    \xi(u^\T X_{i,I})(v^\T X_{i,J})
  }.
\]
Allocate \(\rho_B/D\) to every profile and set
\[
  T_m
  =
  \left\lfloor
    \frac{a n\sqrt{k\rho_B/D}}{m^{3/2}}
  \right\rfloor.
\]
For \(T_m\geq1\), set \(\varepsilon_m=\sqrt{2\rho_B/(D T_m)}\).
Choose \(C\) sufficiently large and use the common public radius
\begin{equation}
  R^2
  =
  C\log^C n
  \xk{
    \frac{k}{n}
    +
    \frac{a k}{n\sqrt{\rho_B}}
  }.
  \label{eq:Proof_G_ProfileRadius}
\end{equation}
Given profile centers \((\overline{A}_m)_{m\in\mca{D}_k}\), define the
common feasible set by
\begin{equation}
  \mca{A}_B^{\mathrm{prof}}(a)
  =
  \dk{
    A\in\mca{K}_k(a):
    \max_{m\in\mca{D}_k}d_m(A,\overline{A}_m)\leq R
  }.
  \label{eq:Proof_G_ProfileFeasibleSet}
\end{equation}

\begin{algorithm}[H]
  \caption{Profile Block Estimator (formal)}
  \label{alg:Supp_DPCov_G_ProfileReleaseFormal}
  \begin{algorithmic}
    \State \textbf{Input:} \(X\), square block \(B=I\times J\) of side
    length \(k\), radius \(a\), and budget \(\rho_B\).
    \For{\(m\in\mca{D}_k\)}
      \If{\(T_m=0\)}
        \State Set \(\overline{A}_m=0\).
      \Else
        \State Initialize \(z_1=z^\circ\).
        \For{\(t=1,\ldots,T_m\)}
          \State Set the primal candidate \(A_t=\pi(z_t)\).
          \State Use the score
          \(H\mapsto\widehat{\mu}_B(H)-\langle A_t,H\rangle_F\).
          \State Select \(H_t\in\widetilde{\mca{H}}_m\) by the exponential
          mechanism with privacy parameter \(\varepsilon_m\).
          \State Update \(z_{t+1}\) by
          \cref{eq:Proof_G_ProfilePrimalUpdate} with \(T=T_m\).
        \EndFor
        \State Set \(\overline{A}_m=T_m^{-1} \sum_{t=1}^{T_m} A_t\).
      \EndIf
    \EndFor
    \State Form \(\mca{A}_B^{\mathrm{prof}}(a)\) according to
    \cref{eq:Proof_G_ProfileFeasibleSet}.
    \State Set \(\widehat{A}_B^{\mathrm{prof}}\) to the unique
    minimum-Frobenius-norm element of \(\mca{A}_B^{\mathrm{prof}}(a)\) when
    the set is nonempty, and to zero otherwise.
    \State \textbf{return} \(\widehat{A}_B^{\mathrm{prof}}\).
  \end{algorithmic}
\end{algorithm}

\paragraph{Oracle Guarantee}
\label{para:Supp_DPCov_G_ProfileOracleGuarantee}
We now combine the preceding ingredients to bound the error of each profile center.
The bound has three sources: fluctuation of the empirical measurements, error in private probe selection, and the deterministic fitting error.
More rounds reduce the fitting error but increase the selection error under a fixed privacy budget.
The choice of \(T_m\) balances these two terms.
For block \(B_b\), write \(\overline{A}_{b,m}\) for its profile-\(m\) center.

\begin{lemma}[Profile fitting error]
  \label{lem:Proof_G_FixedProfile}
  Under the assumptions of \cref{prop:Supp_DPCov_G_ProfileOracle}, with
  probability at least \(1-\beta_B\), simultaneously for every \(b\leq q\)
  and \(m\in\mca{D}_k\),
  \begin{equation}
    d_m(A_b,\overline{A}_{b,m})^2
    \lesssim_{\log}
    \frac{m+\log(q+1)}{n}
    +
    \frac{a\sqrt{D k m}}{n\sqrt{\rho_B}}.
    \label{eq:Proof_G_ProfileRadiusSquared}
  \end{equation}
\end{lemma}

\begin{proof}
  \emph{Uniform score error.}
  By \cref{lem:SubGaussianProduct}, scalar Bernstein's inequality,
  \cref{eq:Proof_G_ProfileEntropy}, and a union bound, with probability at
  least \(1-\beta_B/2\),
  \begin{equation}
    \abs{\widehat{\mu}_{B_b}(H)-\langle A_b,H\rangle_F}
    \lesssim_{\log}
    \sqrt{\frac{m+\log(q+1)}{n}}
    \label{eq:Proof_G_ProfileScoreEvent}
  \end{equation}
  simultaneously over \(b\leq q\), \(m\in\mca{D}_k\), and
  \(H\in\widetilde{\mca{H}}_m\).
  Condition \cref{eq:Supp_DPCov_GH_OperatorScaleCondition} absorbs the linear
  Bernstein term.

  \emph{Case \(T_m\geq1\).}
  Conditionally on the preceding transcript, \(A_t\) is fixed and the score
  sensitivity is \(2L/n\).
  The public horizons are polynomial in \(n\), so the union over blocks,
  profiles, and rounds contributes only logarithmic factors.
  Applying the exponential-mechanism bound in
  \cref{lem:FiniteExponentialMechanism} conditionally at each round, followed
  by the tower property and a union bound, yields an event of probability at
  least \(1-\beta_B/2\), simultaneously over all \(b,m,t\).
  On the intersection of the empirical-score and
  exponential-mechanism events, the deterministic approximation in
  \cref{eq:Proof_G_ProfileNetApproximation} gives
  \[
    d_m(A_b,A_t)
    -
    \langle A_b-A_t,H_t \rangle_F
    \lesssim_{\log}
    \frac{m}{n}\sqrt{\frac{D T_m}{\rho_B}}
    +
    \sqrt{\frac{m+\log(q+1)}{n}}.
  \]
  By convexity of \(d_m(A_b,\cdot)\), averaging over \(t\) and applying
  \cref{eq:Proof_G_ProfileRegret} with comparator \(A_b\) yield
  \[
    d_m(A_b,\overline{A}_{b,m})
    \lesssim_{\log}
    \sqrt{\frac{m+\log(q+1)}{n}}
    +
    a\sqrt{\frac{k}{mT_m}}
    +
    \frac{m}{n}\sqrt{\frac{D T_m}{\rho_B}}.
  \]
  For \(T_m\geq1\), the rounded and unrounded horizons differ by at most a
  factor of two.
  Substituting this choice and squaring proves
  \cref{eq:Proof_G_ProfileRadiusSquared}.

  \emph{Case \(T_m=0\).}
  If \(T_m=0\), then \(d_m(A_b,0)\leq a\).
  The defining inequality \(a n\sqrt{k\rho_B/D}<m^{3/2}\), together with
  \(m\leq2k\), gives
  \[
    a^2
    \lesssim
    \frac{a\sqrt{D k m}}{n\sqrt{\rho_B}},
  \]
  so the same bound holds in the zero-horizon branch.
\end{proof}

\begin{proof}[Proof of \cref{prop:Supp_DPCov_G_ProfileOracle}]
  \emph{Feasible-set reconciliation.}
  Since \(m\leq2k\) and \(D\lesssim\log n\),
  \cref{eq:Proof_G_ProfileRadiusSquared} gives, simultaneously over
  all blocks and profiles,
  \[
    d_m(A_b,\overline{A}_{b,m})^2
    \lesssim_{\log}
    \frac{k+\log(q+1)}{n}
    +
    \frac{a k}{n\sqrt{\rho_B}}.
  \]
  Choosing the fixed constant in \cref{eq:Proof_G_ProfileRadius}
  sufficiently large makes the public radius dominate this bound.
  Hence, with probability at least \(1-\beta_B\),
  \[
    A_b \in\mca{A}_b^{\mathrm{prof}}(a)
    \qquad
    \text{for every }b\leq q.
  \]
  If \(A,A'\) belong to the same feasible set, then
  \(d_m(A,A')\leq2R\) for every profile.
  The dual-probe decomposition \cref{eq:Proof_G_WitnessDecomposition}
  therefore gives \(\norm{A-A'}\leq2D^2 R\), and hence
  \[
    \operatorname{diam}_{\mathrm{op}}^2
    \xkx{\mca{A}_b^{\mathrm{prof}}(a)}
    \lesssim_{\log}
      \frac{k+\log(q+1)}{n}
      +
      \frac{a k}{n\sqrt{\rho_B}}.
  \]

  \emph{Expected risk.}
  On the preceding event, the diameter bound controls the maximum squared
  error of the point outputs.
  On the exceptional event, both the truth and a nonzero output belong
  to \(\mca{K}_k(a)\), while the fallback output is zero.
  Schur's inequality therefore bounds the squared loss by \(4a^2\);
  the condition \(\beta_B\leq n^{-2}\) absorbs its contribution.
  This proves \cref{eq:Supp_DPCov_G_ProfilePointRisk}.
\end{proof}

\subsubsection{\texorpdfstring{\(\mca{G}_\alpha\)}{G-alpha} under Frobenius Loss: Peeling Blocks}

The row--column \(\ell^1\) bound gives a \(k r\)-sparse approximation; the
remaining argument controls private support selection and noisy coefficient
release.

\begin{lemma}[Sparse approximation of a row-tail block]
  \label{lem:Supp_DPCov_G_FrobeniusSparse}
  If \(A\in\mca{K}_k(a)\) and \(1\leq r\leq k\), then
  \[
    \sum_{u,v}\abs{A_{uv}}
    \leq
    k a,
  \]
  and there is a set \(S\subseteq[k]^2\) with \(\abs{S}\leq k r\) such that
  \[
    \norm{A_{S^c}}_F^2
    \leq
    \frac{k a^2}{r}.
  \]
\end{lemma}

\begin{proof}[Proof of \cref{lem:Supp_DPCov_G_FrobeniusSparse}]
  Summing the defining row bounds of \(\mca{K}_k(a)\) gives
  \[
    \sum_{u,v} \abs{A_{uv}}\leq k a.
  \]
  Hence, if \(S\) indexes the \(k r\) largest entries of \(A\) in
  absolute value, every entry outside \(S\) has magnitude at most
  \(a/r\).
  Therefore,
  \[
    \norm{A_{S^c}}_F^2
    \leq
    \frac{a}{r}\sum_{(u,v)\notin S} \abs{A_{uv}}
    \leq
    \frac{k a^2}{r}.
  \]
\end{proof}

\begin{proposition}[Peeling oracle bound]
  \label{prop:Supp_DPCov_G_FrobeniusOracle}
  Suppose \(A_b\in\mca{K}_k(a)\) with \(a=O(1)\) for every \(b\leq q\).
  Give each block budget \(0<\rho_B\leq\rho\).
  For each public \(1\leq r\leq k\), the joint sparse estimator defined in
  \cref{alg:DP_Cov_G_FrobeniusRelease} satisfies
  \begin{equation}
    \max_{b\leq q}
    \E\norm{\widehat{A}_{b,r}^{\mathrm{sp}}-A_b}_F^2
    \lesssim_{\log}
    k\xk{
      \frac{a^2}{r}
      +
      \frac{r}{n}
      +
      \frac{k r^2}{\rho_B n^2}
    }.
    \label{eq:DP_Cov_G_FrobeniusOracle}
  \end{equation}
\end{proposition}

\begin{proposition}[Peeling deviation radius]
  \label{prop:Supp_DPCov_G_FrobeniusDeviation}
  Under the setting of \cref{prop:Supp_DPCov_G_FrobeniusOracle}, fix
  \(0<\beta_B<1/4\).
  Set
  \begin{equation}
    R_{\mathrm{sp}}^2(a,r,\rho_B)
    =
    C\log^C n\,k
    \xk{
      \frac{a^2}{r}
      +
      \frac{r}{n}
      +
      \frac{k r^2}{\rho_B n^2}
    }.
    \label{eq:DP_Cov_G_FrobeniusDeviationSparse}
  \end{equation}
  With probability at least \(1-\beta_B\), the sparse estimator in
  \cref{alg:DP_Cov_G_FrobeniusRelease} satisfies simultaneously for all
  \(b\leq q\),
  \[
    \norm{\widehat{A}_{b,r}^{\mathrm{sp}}-A_b}_F
    \leq
    R_{\mathrm{sp}}(a,r,\rho_B).
  \]
\end{proposition}

\begin{proof}[Proofs of \cref{prop:Supp_DPCov_G_FrobeniusOracle} and
  \cref{prop:Supp_DPCov_G_FrobeniusDeviation}]
  Write \(s=k r\), and let \(T_b\) contain the \(s\) largest entries of
  \(A_b\) in absolute value.
  By \cref{lem:Supp_DPCov_G_FrobeniusSparse},
  \[
    \norm{(A_b)_{T_b^c}}_F^2
    \leq
    \frac{k a^2}{r}.
  \]

  \emph{Uniform stochastic bounds.}
  At peeling round \(t\), let \((U_{b,t},V_{b,t})\) denote the coordinate
  selected for block \(b\).
  Define
  \[
    \delta^2
    =
    \max_{b,u,v}
    \abs{\widetilde{A}_{b,uv}-A_{b,uv}}^2,
    \qquad
    g^2
    =
    \max_{b\leq q}
    \sum_{(u,v)\in\widehat{S}_b}G_{b,uv}^2,
  \]
  and
  \[
    \gamma
    =
    \max_{\substack{b\leq q\\t\leq s}}
    \xkx{
      \max_{(u,v)\ \mathrm{available}}
      \abs{\widetilde{A}_{b,uv}}
      -
      \abs{\widetilde{A}_{b,U_{b,t}V_{b,t}}}
    }.
  \]
  On a common event,
  \[
    \delta^2\lesssim_{\log}\frac{1}{n},
    \qquad
    \gamma\lesssim_{\log}\frac{\sqrt{s}}{n\sqrt{\rho_B}},
    \qquad
    g^2\lesssim_{\log}\frac{s^2}{\rho_B n^2}.
  \]
  By \cref{lem:SubGaussianProduct}, scalar Bernstein's inequality controls the
  clipped products in \(\delta\), while a Gaussian chi-square bound controls
  \(g\).
  For \(\gamma\), condition on the preceding peeling transcript at each
  round, apply the exponential-mechanism utility bound in
  \cref{lem:FiniteExponentialMechanism}, and then use the tower property and
  a union bound.
  The logarithmic clipping bias is included in the bound for \(\delta\).
  Allocate the total failure probability \(\beta_B\) among the coordinate,
  peeling, and refill events.
  The union bounds range over the \(q\) blocks, \(k^2\) coordinates, and at
  most \(k^2\) peeling rounds, and hence contribute only logarithmic factors.
  Hence all three bounds hold simultaneously with probability at least
  \(1-\beta_B\).

  \emph{Deterministic peeling inequality.}
  Pair the coordinates in \(T_b\setminus\widehat{S}_b\) bijectively
  with those in \(\widehat{S}_b\setminus T_b\).
  A missing coordinate \(i\) is still available when its paired false
  coordinate \(j\) is selected, so
  \[
    \abs{A_{b,i}}-\abs{A_{b,j}}
    \leq
    2\delta+\gamma.
  \]
  Using \(x^2-y^2\leq(x+y)(x-y)\) and
  \(\sum_{u,v}\abs{A_{b,uv}}\leq k a\) therefore gives
  \[
    \norm{(A_b)_{\widehat{S}_b^c}}_F^2
    \leq
    \frac{k a^2}{r}+k a(2\delta+\gamma).
  \]
  On the selected coordinates,
  \[
    \norm{
      (\widehat{A}_{b,r}^{\mathrm{sp}}-A_b)_{\widehat{S}_b}
    }_F^2
    \leq
    2s\delta^2+2g^2.
  \]
  Combining the last two displays gives
  \begin{equation}
    \norm{\widehat{A}_{b,r}^{\mathrm{sp}}-A_b}_F^2
    \leq
    \frac{k a^2}{r}
    +
    k a(2\delta+\gamma)
    +
    2s\delta^2
    +
    2g^2.
    \label{eq:Proof_G_FrobeniusMaster}
  \end{equation}

  Substitute these three stochastic bounds into
  \cref{eq:Proof_G_FrobeniusMaster}, and use
  \[
    \frac{k a}{\sqrt{n}}
    \leq
    \frac{k}{2}
    \xk{
      \frac{a^2}{r}
      +
      \frac{r}{n}
    },
    \qquad
    \frac{k a\sqrt{k r}}{n\sqrt{\rho_B}}
    \leq
    \frac{1}{2}
    \xk{
      \frac{k a^2}{r}
      +
      \frac{k^2 r^2}{\rho_B n^2}
    }.
  \]
  After increasing the fixed constant in the public sparse radius, this proves
  the deviation statement.

  For the expectation bound, keep the tail parameter free in the Bernstein
  and Gaussian inequalities.
  For the adaptive peeling steps, integrate conditionally and apply the
  tower property.
  This gives
  \[
    \E\delta^2\lesssim_{\log}\frac{1}{n},
    \qquad
    \E\gamma\lesssim_{\log}
    \frac{\sqrt{s}}{n\sqrt{\rho_B}},
    \qquad
    \E g^2\lesssim_{\log}
    \frac{s^2}{\rho_B n^2}.
  \]
  Taking expectations in \cref{eq:Proof_G_FrobeniusMaster}, using
  \(\E\delta\leq(\E\delta^2)^{1/2}\), and applying the same two Young
  inequalities proves
  \cref{eq:DP_Cov_G_FrobeniusOracle}.
\end{proof}

\subsection{Dyadic Aggregation and Rate Reductions}

The local bounds must be combined with the error from the dense center and the bias from discarding the tail.
Within each scale, the block arrangement controls operator error through the largest local error, while squared Frobenius errors add over disjoint blocks.
After aggregation, choosing the center and terminal scales reduces to balancing these contributions.
Throughout the remainder of this section, put \(x=d/(\rho n^2)\).
For each construction below, \(k_0\) denotes the center size supplied to
\cref{alg:DP_Cov_DyadicTemplate}.
At every dyadic outer scale, \(q\lesssim d\), the number of scales is
\(O(\log d)\), and
\(\rho_k^{-1}\lesssim d\log d/\rho\).
Under the standing regime, all local budget and block-count logarithms are
\(O(\log n)\).
For the known smoothness high-probability arguments, take
\(\beta_B=n^{-2}\).
The next two lemmas control the dense center and aggregate its error
with the outer blocks.
The remaining lemmas reduce the tuning problems to scalar balances used here
and in the adaptive proofs.

\begin{lemma}[Dense center error]
  \label{lem:Proof_GH_BaseRelease}
  For every dyadic center size \(k_0\geq1\) satisfying
  \(k_0+\log d\leq cn\), and every sufficiently large fixed \(C_1\), the
  dense estimator on \(\mca{T}_{k_0}\) in
  \cref{alg:DP_Cov_DyadicTemplate}, with symmetric embedding
  \(\widehat{\Sigma}^0\), satisfies
  \begin{align}
    \E\norm{
      \widehat{\Sigma}^0
      -\Sigma_{\mca{T}_{k_0}}
    }^2
    &\lesssim_{\log}
    \frac{k_0}{n}+x k_0^2,
    \label{eq:Proof_GH_BaseOperator}\\
    \frac{1}{d}\E\norm{
      \widehat{\Sigma}^0
      -\Sigma_{\mca{T}_{k_0}}
    }_F^2
    &\lesssim_{\log}
    \frac{k_0}{n}+x k_0^2.
    \label{eq:Proof_GH_BaseFrobenius}
  \end{align}
\end{lemma}

\begin{proof}
  \emph{Blockwise bound.}
  The center calls \cref{alg:DP_Cov_BlockDiagonal} with block size
  \(k_0\), total budget \(\rho/4\), and radius
  \(R_0=K\sqrt{C_1(k_0+\log n)}\).
  Each center block receives budget
  \(\rho_0=\rho/(8N_{k_0})\asymp\rho k_0/d\).
  Apply the covariance, clipping, and Gaussian-noise arguments from
  \cref{lem:BlockBound} with radius \(R_0\) and tail parameter
  \(k_0+\log d+t\).
  A union bound over at most \(2N_{k_0}\) blocks and integration over
  \(t\geq0\) give
  \begin{align*}
    \E
    \max_{\substack{l\in[N_{k_0}],\ r\in\dkx{0,1}\\
    l+r\leq N_{k_0}}}
    \norm{\widehat{\Sigma}[B_{k_0;l,r}]-\Sigma[B_{k_0;l,r}]}^2
    &\lesssim_{\log}
    \frac{k_0}{n}
    +
    \frac{(k_0+\log n)^2(k_0+\log d)}{\rho_0 n^2}\\
    &\qquad
    +e^{-c(k_0+\log n)}\\
    &\lesssim_{\log}
    \frac{k_0}{n}+x k_0^2.
  \end{align*}
  \emph{Operator aggregation.}
  The center blocks have a blockwise tridiagonal incidence pattern, so
  \cref{cor:TriDiagonalNormBound} proves
  \cref{eq:Proof_GH_BaseOperator}.

  \emph{Frobenius aggregation.}
  Let \(E_B=\widehat{\Sigma}[B]-\Sigma[B]\) for a center block \(B\).
  Each \(E_B\) has rank at most \(k_0\), so
  \(\norm{E_B}_F^2\leq k_0\norm{E_B}^2\).
  Disjointness of the upper-triangular center blocks, symmetry, and
  \(N_{k_0}k_0/d\asymp1\) therefore give
  \[
    \frac{1}{d}\E\norm{
      \widehat{\Sigma}^0-\Sigma_{\mca{T}_{k_0}}
    }_F^2
    \lesssim
    \frac{N_{k_0}k_0}{d}
    \E\max_B\norm{E_B}^2
    \lesssim_{\log}
    \frac{k_0}{n}+x k_0^2,
  \]
  which proves \cref{eq:Proof_GH_BaseFrobenius}.
\end{proof}

\begin{lemma}[Dyadic square-block error aggregation]
  \label{lem:Proof_DyadicSquareError}
  Fix center and terminal scales \(k_0\leq k_+\) satisfying
  \(k_0+\log d\leq cn\).
  Put
  \(E_B=\widehat{\Sigma}_B-\Sigma_B\) for \(B\in\mca{B}_k\).
  Let \(\widetilde{\Sigma}\) be obtained by the block assignments in
  \cref{alg:DP_Cov_DyadicTemplate}, before the final projection, and
  let \(\widehat{\Sigma}=\Pi_{M_0}(\widetilde{\Sigma})\).
  Then
  \begin{align}
    \E\norm{\widehat{\Sigma}-\Sigma}^2
    &\lesssim_{\log}
      \frac{k_0}{n}+x k_0^2
      +
      \norm{\Sigma-\Sigma_{\mca{T}_{k_+}}}^2
      +
      J(k_0)
      \sum_k
      \E\max_{B\in\mca{B}_k}\norm{E_B}^2,
    \label{eq:Proof_GH_OperatorAggregation}\\
    \frac{1}{d}\E\norm{\widehat{\Sigma}-\Sigma}_F^2
    &\lesssim_{\log}
      \frac{k_0}{n}+x k_0^2
      +
      \frac{1}{d}\norm{\Sigma-\Sigma_{\mca{T}_{k_+}}}_F^2
      +
      \frac{2}{d}
      \sum_k \sum_{B\in\mca{B}_k}
      \E\norm{E_B}_F^2.
    \label{eq:Proof_GH_FrobeniusAggregation}
  \end{align}
\end{lemma}

\begin{proof}
  \emph{Error decomposition.}
  Put
  \[
    E_0
    =
    \widehat{\Sigma}^0
    -\Sigma_{\mca{T}_{k_0}},
    \qquad
    E_+=\Sigma_{\mca{T}_{k_+}}-\Sigma.
  \]
  For each scale \(k\), let \(E_k\) be the symmetric matrix satisfying
  \(E_k[B]=E_B\) and \(E_k[B^\T]=E_B^\T\) for every
  \(B\in\mca{B}_k\), with all other entries set to zero.
  The dyadic construction in \cref{subsec:DyadicBlockStructure}
  partitions \(\mca{T}_{k_+}\) into the center support and the active outer
  supports, which are disjoint.
  Hence
  \[
    \widetilde{\Sigma}-\Sigma=E_0+\sum_k E_k+E_+.
  \]

  \emph{Operator loss.}
  For fixed \(k\), form the graph whose vertices are the intervals
  \(I_{k,\ell}\) and whose edges are the blocks in \(\mca{B}_k\).
  Each interval is incident to at most three increment blocks.
  Thus the graph has uniformly bounded maximum degree.
  A greedy edge coloring therefore partitions \(\mca{B}_k\) into a
  fixed number of matchings.
  The embedded errors within a matching have disjoint row and
  column supports and become block diagonal after a simultaneous
  permutation, so the norm of their sum is the largest block norm.
  Summing over the constant number of matchings gives
  \[
    \norm{E_k}
    \lesssim
    \max_{B\in\mca{B}_k}\norm{E_B}.
  \]
  Hence
  \[
    \begin{aligned}
      \norm{\widetilde{\Sigma}-\Sigma}^2
      &\lesssim
      \norm{E_0}^2+\norm{E_+}^2
      +J(k_0)\sum_k \norm{E_k}^2\\
      &\lesssim
      \norm{E_0}^2+\norm{E_+}^2
      +J(k_0)\sum_k
      \max_{B\in\mca{B}_k}\norm{E_B}^2.
    \end{aligned}
  \]
  This proves the operator contribution in
  \cref{eq:Proof_GH_OperatorAggregation}.

  \emph{Frobenius loss.}
  Disjointness of the center and active outer supports gives
  \[
    \norm{\widetilde{\Sigma}-\Sigma}_F^2
    =
    \norm{E_0}_F^2
    +
    2\sum_k \sum_{B\in\mca{B}_k}\norm{E_B}_F^2
    +
    \norm{E_+}_F^2,
  \]
  which proves \cref{eq:Proof_GH_FrobeniusAggregation}.
  The center bounds in both conclusions follow from
  \cref{lem:Proof_GH_BaseRelease}.

  \emph{Projection.}
  Spectral clipping changes operator error by at most a numerical
  factor and is nonexpansive in Frobenius norm.
\end{proof}

The next lemma handles the integer constraint on the local rank parameter.
It separates an interior balance from the endpoint choices, so subsequent tuning arguments can retain the finite-dimensional regimes.

\begin{lemma}[Integer three-regime reduction]
  \label{lem:Proof_GH_ThreeRegimeReduction}
  For \(A,B>0\) and an integer \(k\geq1\), put
  \[
    Q_k(A,B)
    =
    \max_{1\leq r\leq k}
    \xk{
      \frac{A}{r}\wedge B r
    }.
  \]
  Then
  \begin{equation}
    Q_k(A,B)
    \asymp
    A\wedge\sqrt{AB}\wedge B k
    \asymp
    A
    \wedge
    \min_{1\leq r\leq k}
    \xk{\frac{A}{r}+B r}
    \wedge
    B k.
    \label{eq:Proof_GH_ThreeRegimeReduction}
  \end{equation}
\end{lemma}

\begin{proof}
  Clip the continuous balance to the feasible interval:
  \[
    r_\star
    =
    \min\xk{
      k,\max\xk{1,\sqrt{A/B}}
    },
  \]
  and choose an integer \(r\in[1,k]\) within a factor two of
  \(r_\star\).
  This choice gives the lower bound
  \[
    Q_k(A,B)
    \gtrsim
    A\wedge\sqrt{AB}\wedge B k.
  \]
  Conversely, for every feasible \(r\),
  \[
    \frac{A}{r}\wedge B r
    \leq
    A\wedge\sqrt{AB}\wedge B k,
  \]
  which proves the first equivalence in
  \cref{eq:Proof_GH_ThreeRegimeReduction}.
  The same clipped choice gives the upper bound for the second
  expression, while \(A/r+B r\geq2\sqrt{AB}\) gives its lower bound.
\end{proof}

For the peeling construction, we must also compare the optimized sparse release with the zero and dense releases.
The following reduction separates the sampling and privacy contributions to this comparison, making their dependence on the block scale explicit.

\begin{lemma}[Three-branch reduction for Frobenius loss]
  \label{lem:Proof_G_FrobeniusThreeBranchReduction}
  For \(a,y>0\) and an integer \(k\geq1\), define
  \[
    T(a,k;y)
    =
    a^2
    \wedge
    \min_{1\leq r\leq k}
    \xk{
      \frac{a^2}{r}+\frac{r}{n}+y r^2
    }
    \wedge
    \xk{\frac{k}{n}+y k^2},
  \]
  and put
  \[
    U(a,k)
    =
    a^2
    \wedge
    \frac{a}{\sqrt{n}}
    \wedge
    \frac{k}{n},
    \qquad
    V(a,k;y)
    =
    a^2
    \wedge
    a^{4/3}y^{1/3}
    \wedge
    y k^2.
  \]
  Then
  \begin{equation}
    T(a,k;y)
    \lesssim
    U(a,k)+V(a,k;y).
    \label{eq:Proof_G_FrobeniusThreeBranchReduction}
  \end{equation}
  Moreover, if
  \[
    a\sqrt{n}
    \wedge
    (a^2/y)^{1/3}
    \lesssim_{\log}
    k,
  \]
  then the middle branch alone satisfies
  \begin{equation}
    \min_{1\leq r\leq k}
    \xk{
      \frac{a^2}{r}+\frac{r}{n}+y r^2
    }
    \lesssim_{\log}
      U(a,k)+V(a,k;y)+\frac{1}{n}+y.
    \label{eq:Proof_G_FrobeniusMiddleReduction}
  \end{equation}
  If \(a_k\leq C_a k^{-\alpha}\), then
  \begin{align}
    \sup_{k\leq d/4} U(a_k,k)
    &\lesssim
    n^{-\frac{2\alpha+1}{2\alpha+2}}
    \wedge\frac{d}{n},
    \label{eq:Proof_G_FrobeniusSamplingBound}\\
    \sup_{k\leq d/4} V(a_k,k;y)
    &\lesssim
    y^{\frac{2\alpha+1}{2\alpha+3}}
    \wedge y d^2.
    \label{eq:Proof_G_FrobeniusPrivacyBound}
  \end{align}
  The constants can be chosen uniformly when \(\alpha\) ranges over a
  fixed compact subset of \((0,\infty)\).
\end{lemma}

\begin{proof}
  Put
  \[
    r_{\mathrm{s}}=a\sqrt{n},
    \qquad
    r_{\mathrm{p}}=(a^2/y)^{1/3},
    \qquad
    s=r_{\mathrm{s}}\wedge r_{\mathrm{p}}.
  \]
  If \(s<1\), the zero branch proves
  \cref{eq:Proof_G_FrobeniusThreeBranchReduction}: when
  \(r_{\mathrm{s}}<1\), \(U(a,k)=a^2\), and when
  \(r_{\mathrm{p}}<1\), \(V(a,k;y)=a^2\).
  Taking \(r=1\) also bounds the middle branch by
  \(a^2+n^{-1}+y\), proving
  \cref{eq:Proof_G_FrobeniusMiddleReduction} in this case.
  If \(s>k\), both balancing indices exceed \(k\), so
  \(U(a,k)=k/n\), \(V(a,k;y)=yk^2\), and the dense branch gives the
  desired bound.
  If also \(s\lesssim_{\log}k\), take \(r=k\).
  When \(s=r_{\mathrm{s}}\), \(a^2/k\lesssim_{\log}k/n\); when
  \(s=r_{\mathrm{p}}\), \(a^2/k\lesssim_{\log}yk^2\).
  Thus the same choice proves
  \cref{eq:Proof_G_FrobeniusMiddleReduction}.
  In the remaining case \(1\leq s\leq k\), take
  \(r=\ceil{s}\in[1,k]\), so \(s\leq r\leq2s\).
  When \(s=r_{\mathrm{s}}\leq r_{\mathrm{p}}\), the inequality
  \(y a n^{3/2}\leq1\) gives
  \[
    \frac{a^2}{r}+\frac{r}{n}+y r^2
    \lesssim
    \frac{a}{\sqrt{n}}.
  \]
  When \(s=r_{\mathrm{p}}\leq r_{\mathrm{s}}\), the reverse inequality
  \(y a n^{3/2}\geq1\) gives
  \[
    \frac{a^2}{r}+\frac{r}{n}+y r^2
    \lesssim
    a^{4/3} y^{1/3}.
  \]
  This proves \cref{eq:Proof_G_FrobeniusThreeBranchReduction} and
  \cref{eq:Proof_G_FrobeniusMiddleReduction},
  including the integer endpoints \(r=1,k\).

  Since \(U(a,k)\) and \(V(a,k;y)\) are nondecreasing in \(a\), it
  suffices to take \(a_k=C_a k^{-\alpha}\).
  For the sampling part,
  \[
    U(a_k,k)
    \leq
    \frac{C_a k^{-\alpha}}{\sqrt{n}}
    \wedge
    \frac{k}{n}.
  \]
  The two terms balance at
  \(k\asymp n^{1/(2\alpha+2)}\), giving the first rate in
  \cref{eq:Proof_G_FrobeniusSamplingBound}; the bound
  \(U(a_k,k)\leq k/n\leq d/n\) gives its dimension-endpoint term.
  Similarly,
  \[
    V(a_k,k;y)
    \leq
    C_a^{4/3} k^{-4\alpha/3} y^{1/3}
    \wedge
    y k^2.
  \]
  These terms balance at
  \(k\asymp y^{-1/(2\alpha+3)}\), giving the first rate in
  \cref{eq:Proof_G_FrobeniusPrivacyBound}; the bound
  \(V(a_k,k;y)\leq y k^2\leq y d^2\) gives its dimension-endpoint term.
  Both comparisons are uniform on compact smoothness intervals.
\end{proof}

\subsection{Upper Bounds for Known Smoothness}

\subsubsection{Operator Loss}

\begin{proof}[Proof of \cref{thm:DPCov_GH_Operator}]
  Recall \(x=d/(\rho n^2)\), and write
  \[
    R_{\mathrm{stat}}
    =
    n^{-2\alpha/(2\alpha+1)} \wedge\frac{d}{n},
    \qquad
    u_k=k^{-2\alpha},
    \qquad
    v_k=x k.
  \]
  Here \(u_k\) is the squared structural radius of a scale-\(k\) outer block,
  while \(v_k\) is its effective privacy cost.

  \emph{Center.}
  Let \(k_0\) and \(k_+\) be the public scales in the
  construction.
  By the standing regime, the cutoff defining \(k_+\) is no smaller than the
  cutoff defining \(k_0\) for all sufficiently large \(n\).
  Dyadic rounding changes each nontrivial cutoff by at most a factor of two.
  Hence
  \[
    \frac{k_0}{n}+x k_0^2
    \lesssim
      R_{\mathrm{stat}}+\mca{Q}_{\alpha,d}(x).
  \]
  The statistical cutoff controls \(k_0/n\), while the privacy cutoff is
  determined by the crossing of
  \(x k^2\) and \(k^{-(2\alpha+1)}\), the dense endpoint in
  \cref{eq:Rate_Q}.
  If the dimension cutoff is active, use \(k=\lfloor d/4\rfloor\)
  in that definition.
  Moreover, \(k_0\lesssim n^{1/(2\alpha+1)}=o(n)\) and
  \(\log d=O(\log n)\), so the center condition in
  \cref{lem:Proof_GH_BaseRelease} holds.

  \emph{Outer geometry and budgets.}
  Every block in \(\mca{B}_k\) lies at distance between \(k\) and
  \(4k\) from the diagonal, and hence
  \[
    \Sigma_B\in\mca{K}_k(a_{\mca{G},k})
    \quad\text{under }\mca{G}_\alpha,
    \qquad
    \Sigma_B\in\mca{E}_k(a_{\mca{H},k})
    \quad\text{under }\mca{H}_\alpha,
    \qquad
    a_{\mca{J},k}\lesssim k^{-\alpha}.
  \]
  At an active scale, \(q=3(N_{2k}-1)\) and
  \[
    \rho_k
    =
    \frac{\rho}{4J(k_0)N_{2k}}
    \asymp
    \frac{\rho k}{dJ(k_0)}.
  \]
  At every active scale,
  \(k\leq k_+/2\lesssim n^{1/(2\alpha+1)}=o(n)\).
  Hence \cref{eq:Supp_DPCov_GH_OperatorScaleCondition} holds at every active
  scale for all sufficiently large \(n\).

  \emph{Scale-wise rate reduction.}
  Put
  \[
    s_k=\frac{k+\log(3N_{2k})}{n}.
  \]
  The radius and budget relations above give
  \[
    a_{\mca{J},k}^2\lesssim u_k,
    \qquad
    \frac{k^2}{\rho_k n^2}\lesssim_{\log}v_k,
    \qquad
    \frac{a_{\mca{G},k}k}{n\sqrt{\rho_k}}
    \lesssim_{\log}
    \sqrt{u_k v_k}.
  \]
  For \(\mca{H}_\alpha\), the fixed-rank bound in
  \cref{eq:Supp_DPCov_H_GreedyPointRisk} and the public choice of rank in
  \cref{alg:DP_Cov_H_GreedyRelease} give
  \[
    \E\max_{B\in\mca{B}_k}
    \norm{\widehat{\Sigma}_B-\Sigma_B}^2
    \lesssim_{\log}
    s_k
    +
    \min_{1\leq r\leq k}
    \xk{
      \frac{u_k}{r}+v_k r
    }.
  \]
  For \(\mca{G}_\alpha\),
  \cref{eq:Supp_DPCov_G_ProfilePointRisk} instead gives
  \[
    \E\max_{B\in\mca{B}_k}
    \norm{\widehat{\Sigma}_B-\Sigma_B}^2
    \lesssim_{\log}
    s_k+\sqrt{u_k v_k}.
  \]

  \emph{Outer privacy comparison.}
  At every active scale, the two scale cutoffs imply
  \[
    \sqrt{u_k v_k}
    \vee
    \min_{1\leq r\leq k}
    \xk{
      \frac{u_k}{r}+v_k r
    }
    \lesssim
    R_{\mathrm{stat}}+Q_k(u_k,v_k).
  \]
  Choose \(C\) large enough to absorb the fixed dyadic-rounding factor.
  The terminal privacy cutoff gives \(v_k\leq Cu_k\).
  If \(u_k\leq Cv_k k^2\), then
  \[
    Q_k(u_k,v_k)
    \asymp
    u_k\wedge\sqrt{u_k v_k}\wedge v_k k
    \asymp
    \sqrt{u_k v_k}.
  \]
  The inequalities \(v_k\leq Cu_k\) and \(u_k\leq Cv_k k^2\) place the
  continuous balance \(r=(u_k/v_k)^{1/2}\) in \([1,k]\) up to fixed
  factors, so a feasible integer rank within another fixed factor gives
  \[
    \min_{1\leq r\leq k}
    \xk{
      \frac{u_k}{r}+v_k r
    }
    \lesssim
    \sqrt{u_k v_k}.
  \]
  If \(u_k>Cv_k k^2\), the privacy cutoff cannot define \(k_0\), and the
  dimension cutoff leaves no outer scale.
  Hence every active scale is controlled by the statistical cutoff.
  In that case, \(u_k/k\lesssim n^{-1}\leq R_{\mathrm{stat}}\) and
  \(Q_k(u_k,v_k)\asymp v_k k\).
  The preceding display then follows from
  \[
    \sqrt{u_k v_k}
    \leq
    \frac12\xk{\frac{u_k}{k}+v_k k}.
  \]
  Taking \(r=k\) gives the same upper bound for the rank minimum.

  The terminal scale is either statistically or dimension limited, and in
  both cases \(s_k\lesssim_{\log}R_{\mathrm{stat}}\), while
  \(Q_k(u_k,v_k)\leq\mca{Q}_{\alpha,d}(x)\).
  Consequently, for both classes,
  \[
    \sup_{k_0\leq k<k_+}
    \E\max_{B\in\mca{B}_k}
    \norm{\widehat{\Sigma}_B-\Sigma_B}^2
    \lesssim_{\log}
    R_{\mathrm{stat}}+\mca{Q}_{\alpha,d}(x).
  \]
  The assertion is vacuous when there is no active scale.

  \emph{Tail and aggregation.}
  For both \(\mca{G}_\alpha\) and \(\mca{H}_\alpha\), the row-tail bound
  gives
  \[
    \norm{\Sigma-\Sigma_{\mca{T}_{k_+}}}^2
    \lesssim k_+^{-2\alpha}.
  \]
  This term vanishes when \(k_+=d/2\); otherwise the statistical or
  outer privacy cutoff gives
  \(k_+^{-2\alpha}\lesssim_{\log}
  R_{\mathrm{stat}}+\mca{Q}_{\alpha,d}(x)\).

  Apply \cref{lem:Proof_DyadicSquareError} with center
  \(k_0\) and terminal scale \(k_+\).
  Since there are at most \(J(k_0)\) active scales, its outer term is
  bounded by
  \[
    J(k_0)\sum_k
    \E\max_{B\in\mca{B}_k}
    \norm{\widehat{\Sigma}_B-\Sigma_B}^2
    \leq
    J(k_0)^2
    \sup_k
    \E\max_{B\in\mca{B}_k}
    \norm{\widehat{\Sigma}_B-\Sigma_B}^2.
  \]
  The factor \(J(k_0)^2\) is absorbed by \(\lesssim_{\log}\).
  Combining the center, outer, and tail bounds proves
  \cref{eq:DPCov_GH_Operator}.
  The privacy guarantee follows from
  \cref{prop:Supp_DPCov_GH_GlobalPrivacy}, completing the proof.
\end{proof}

\subsubsection{Frobenius Loss}

\begin{proof}[Proof of \cref{thm:DPCov_G_Frobenius}]
  Recall \(x=d/(\rho n^2)\), and let
  \[
    R_{\mathrm{stat}}
    =
    n^{-\frac{2\alpha+1}{2\alpha+2}}
    \wedge\frac{d}{n},
    \qquad
    R_{\mathrm{priv}}
    =
    x^{\frac{2\alpha+1}{2\alpha+3}}
    \wedge x d^2.
  \]

  \emph{Center.}
  By the definition of \(k_0\) and dyadic rounding,
  \begin{equation}
    \frac{k_0}{n}+x k_0^2
    \lesssim_{\log}
    R_{\mathrm{stat}}+R_{\mathrm{priv}}.
    \label{eq:Proof_G_FixedF_CenterRisk}
  \end{equation}
  Also, \(k_0\lesssim n^{1/(2\alpha+2)}=o(n)\), so the center
  condition in \cref{lem:Proof_GH_BaseRelease} holds.
  The statistical and privacy cutoffs are the crossings of \(k/n\) and
  \(x k^2\), respectively, with \(k^{-2\alpha-1}\).
  If the dimension cutoff is active, the corresponding endpoint terms are
  \(d/n\) and \(x d^2\).

  \emph{Peeling shells.}
  By the dyadic construction and the definition of \(\rho_k\),
  \begin{equation}
    \abs{\mca{B}_k}=3(N_{2k}-1),
    \qquad
    \rho_k
    =
    \frac{\rho}{4J(k_0)N_{2k}}
    \asymp
    \frac{\rho k}{dJ(k_0)}.
    \label{eq:Proof_G_FixedF_Budget}
  \end{equation}
  For every \(B\in\mca{B}_k\), the covariance cross-block \(\Sigma_B\)
  belongs to \(\mca{K}_k(a_k)\), where
  \(a_k=a_{\mca{G},k}\lesssim k^{-\alpha}\).
  In particular, \(a_k=O(1)\), as required by
  \cref{prop:Supp_DPCov_G_FrobeniusOracle}.

  \emph{Rate reduction.}
  By \cref{eq:Proof_G_FixedF_Budget},
  \(k/(\rho_k n^2)\lesssim J(k_0)x\).
  Set
  \[
    y=CJ(k_0)x,
    \qquad
    s_k
    =
    a_k\sqrt{n}
    \wedge
    \xk{\frac{a_k^2}{y}}^{1/3}.
  \]
  If the dimension cutoff defines \(k_0\), there are no outer shells.
  Otherwise, the statistical or privacy cutoff in the definition of \(k_0\)
  gives \(s_k\lesssim_{\log}k\) at every outer scale \(k\geq k_0\).
  Thus \cref{eq:Proof_G_FrobeniusMiddleReduction} applies.
  Since \(n^{-1}\lesssim R_{\mathrm{stat}}\) and
  \(y\lesssim_{\log}x\lesssim_{\log}R_{\mathrm{priv}}\), the two endpoint
  remainders in \cref{eq:Proof_G_FrobeniusMiddleReduction} are also absorbed.
  Apply \cref{prop:Supp_DPCov_G_FrobeniusOracle} jointly over all blocks at
  scale \(k\), using the publicly chosen minimizer \(r_k\).
  The middle reduction, together with
  \cref{eq:Proof_G_FrobeniusSamplingBound} and
  \cref{eq:Proof_G_FrobeniusPrivacyBound}, gives
  \begin{equation}
    \begin{aligned}
      \max_{B\in\mca{B}_k}
      \E\norm{\widehat{\Sigma}_B-\Sigma_B}_F^2
      &\lesssim_{\log}
      k\min_{1\leq r\leq k}
      \xk{
        \frac{a_k^2}{r}+\frac{r}{n}+y r^2
      }\\
      &\lesssim_{\log}
      k\xkx{R_{\mathrm{stat}}+R_{\mathrm{priv}}}.
    \end{aligned}
    \label{eq:Proof_G_FixedF_LocalRisk}
  \end{equation}

  \emph{Aggregation.}
  Apply \cref{lem:Proof_DyadicSquareError} with \(k_+=d/2\).
  Since \(\abs{\mca{B}_k}=3(N_{2k}-1)\),
  \(N_{2k} k/d=1/2\), and there are at most \(J(k_0)\) scales,
  \cref{eq:Proof_G_FixedF_LocalRisk} gives
  \[
    \frac{2}{d}
    \sum_k \sum_{B\in\mca{B}_k}
    \E\norm{\widehat{\Sigma}_B-\Sigma_B}_F^2
    \lesssim_{\log}
    R_{\mathrm{stat}}+R_{\mathrm{priv}}.
  \]
  The center contribution is absorbed by
  \cref{eq:Proof_G_FixedF_CenterRisk}.
  Since projection cannot increase the Frobenius error, this proves
  \cref{eq:DPCov_G_Frobenius}.
  The privacy guarantee follows from
  \cref{prop:Supp_DPCov_GH_GlobalPrivacy}, completing the proof.
\end{proof}

  \clearpage
  \section{Adaptive Multiscale Block Estimator}
\label{sec:Proofs_GH_Adaptive}

We now make the multiscale block estimators adaptive to the unknown smoothness parameter over the pointwise-decay class \(\mca{H}_\alpha\) and the row-tail class \(\mca{G}_\alpha\).
The proof first shows that Lepski selection attains the error bound of a suitable candidate radius at each scale.
We then apply this comparison to the private block releases and aggregate their errors across scales.
Aggregating the selected local releases across the dyadic scales yields the guarantees under operator and Frobenius losses in \cref{thm:DP_Cov_GH_Adaptive} and \cref{thm:DP_Cov_G_Frobenius_Adaptive}.

\subsection{Privacy Guarantee}

\begin{proposition}
  \label{prop:Supp_DPCov_GH_AdaptivePrivacy}
  In each of its three modes, \cref{alg:DP_Cov_GH_Adaptive} is
  \(\rho\)-zCDP.
\end{proposition}

\begin{proof}
  The center release uses budget \(\rho/4\) and is
  \(\rho/4\)-zCDP by
  \cref{prop:DP_Cov_BlockDiagonal_Privacy}.
  For a fixed block and radius, the zero branch has no privacy cost, while
  every nonzero branch is \(\rho_{k,j}\)-zCDP by
  \cref{prop:Supp_DPCov_GH_LocalPrivacy}.
  Composition over the radius grid therefore costs
  \[
    (J_k+1)\rho_{k,j}
    =
    \frac{\rho}{4J_{\mathrm{sq}}N_{2k}}.
  \]
  Since \(\abs{\mca{B}_k}=3(N_{2k}-1)\leq3N_{2k}\), composition over the
  blocks costs at most \(3\rho/(4J_{\mathrm{sq}})\) at scale \(k\).
  There are at most \(J_{\mathrm{sq}}\) active scales, so the outer releases
  cost at most \(3\rho/4\).
  Together with the center, the total cost is at most \(\rho\).
  Radial clipping, Lepski selection, symmetric completion, and the final
  spectral projection are post-processing.
\end{proof}

\subsection{Lepski Selection and Local Error Bounds}

We separate the selection argument from the particular block release used to construct the candidates.
First, radial clipping keeps every candidate in a known ball, which bounds the loss when the simultaneous deviation event fails.
The Lepski argument then compares the selected candidate with an index up to which all deviation bounds are valid.
These two steps convert the local deviation bounds into an expected risk bound.

For a norm \(\norm{\cdot}_\star\), let
\[
  \Gamma_{\star,B_0}(Y)
  =
  Y\min\left\{1,\frac{B_0}{\norm{Y}_\star}\right\},
\]
with multiplier one at \(Y=0\).
Whenever \(\norm{A}_\star\leq B_0\),
\[
  \norm{\Gamma_{\star,B_0}(Y)-A}_\star
  \leq
  2\norm{Y-A}_\star.
\]
Indeed, this is immediate when \(\norm{Y}_\star\leq B_0\); otherwise,
with \(c=B_0/\norm{Y}_\star\),
\[
  \norm{cY-A}_\star
  \leq
  c\norm{Y-A}_\star+(1-c)\norm{A}_\star
  \leq
  2\norm{Y-A}_\star,
\]
because \((1-c)B_0\leq\norm{Y}_\star-B_0\leq\norm{Y-A}_\star\).

\begin{lemma}[Scale-wise bounded-candidate Lepski selector]
  \label{lem:DP_Cov_GH_PointLepski}
  Let \(A_1,\ldots,A_q\) and the candidates
  \(Z_{b,j}\), \(b\leq q\), \(0\leq j\leq J\), belong to the
  \(\norm{\cdot}_\star\)-ball of radius \(B_0\).
  Let \(D_0\geq\cdots\geq D_J\geq0\) be public and define
  \[
    \widehat{j}
    =
    \max\left\{
      j:
      \max_{b\leq q}\norm{Z_{b,j}-Z_{b,i}}_\star
      \leq D_j+D_i
      \text{ for every }i<j
    \right\}.
  \]
  Return \(\widehat{A}_b=Z_{b,\widehat{j}}\), \(b\leq q\).
  If, for some \(j_*\leq J\), an event \(\Omega_{j_*}\) satisfies
  \[
    \max_{b\leq q}\norm{Z_{b,j}-A_b}_\star
    \leq D_j,
    \qquad 0\leq j\leq j_*,
  \]
  then, on \(\Omega_{j_*}\),
  \[
    \widehat{j}\geq j_*,
    \qquad
    \max_{b\leq q}
    \norm{\widehat{A}_b-A_b}_\star
    \leq3D_{j_*}.
  \]
  Consequently, if
  \(\bbP(\Omega_{j_*}^{\complement})\leq\delta\), then
  \[
    \E\max_{b\leq q}
    \norm{\widehat{A}_b-A_b}_\star^2
    \leq
    9D_{j_*}^2+4\delta B_0^2.
  \]
\end{lemma}

\begin{proof}
  On \(\Omega_{j_*}\), for every \(i<j_*\),
  \[
    \max_{b\leq q}\norm{Z_{b,j_*}-Z_{b,i}}_\star
    \leq D_{j_*}+D_i.
  \]
  Thus \(j_*\) is admissible and \(\widehat{j}\geq j_*\).
  If \(\widehat{j}=j_*\), the candidate bound gives the conclusion directly.
  Otherwise, admissibility of \(\widehat{j}\), applied with \(i=j_*\), gives
  \[
    \max_{b\leq q}\norm{Z_{b,\widehat{j}}-A_b}_\star
    \leq
    D_{\widehat{j}}+D_{j_*}+D_{j_*}
    \leq3D_{j_*}.
  \]
  On \(\Omega_{j_*}^{\complement}\), the truth and output both lie in
  the radius-\(B_0\) ball, so their distance is at most \(2B_0\).
  Splitting the expectation over the two events proves the risk bound.
\end{proof}

We next allocate the block budget across the radius grid and choose a failure probability small enough that the contribution from the complementary event remains negligible even at the smallest radius.
For the radius grid, suppose that \(a_j=2^{-j}a_0\), \(a_0=O(1)\), and
\[
  J+\log\left(\frac{e d n q}{\rho_{\mathrm{blk}}}\right)
  \lesssim
  \log n.
\]
Set \(\rho_j=\rho_{\mathrm{blk}}/(J+1)\) and take
\[
  \beta_{\mathrm{ad}}
  =
  \left(\frac{a_J}{a_0}\right)^2 n^{-3}.
\]
The resulting parameters satisfy
\[
  \sum_{j=0}^J\rho_j=\rho_{\mathrm{blk}},
  \qquad
  \log\left(\frac{e d n q}{\beta_{\mathrm{ad}}\rho_j}\right)\lesssim\log n,
  \qquad
  (J+1)\beta_{\mathrm{ad}}a_0^2\lesssim a_J^2\wedge\frac{1}{n}.
\]
The first identity is immediate. Since \(J\lesssim\log n\), the other two
bounds follow from
\((J+1)\beta_{\mathrm{ad}}a_0^2=(J+1)a_J^2n^{-3}\).

\begin{proposition}[Adaptive profile and greedy block oracles]
  \label{prop:DP_Cov_GH_AdaptiveBlock}
  Let \(a_0>\cdots>a_J>0\), where \(a_{j+1}=a_j/2\) and
  \(a_0=O(1)\), and let \(A_1,\ldots,A_q\) be \(k\times k\)
  covariance cross-blocks associated with the \(q\) blocks.
  Give each block total budget \(0<\rho_{\mathrm{blk}}\leq\rho\), set
  \(\rho_j=\rho_{\mathrm{blk}}/(J+1)\), and assume that
  \[
    J+\log\left(\frac{e d n q}{\rho_{\mathrm{blk}}}\right)
    \lesssim
    \log n
  \]
  and that \cref{eq:Supp_DPCov_GH_OperatorScaleCondition} holds.
  If \(A_b\in\mca{K}_k(a)\) for every \(b\leq q\), where the unknown radius
  satisfies \(a\in[a_J,a_0]\), then
  \begin{equation}
    \begin{aligned}
      &\E\max_{b\leq q}
      \norm{\widehat{A}_b^{\mca{G},\mathrm{ad}}-A_b}^2\\
      &\quad\lesssim_{\log}
      a^2
      \wedge
      \Biggl[
        \frac{k+\log(q+1)}{n}
        +
        \left\{
          \frac{a k}{n\sqrt{\rho_{\mathrm{blk}}}}
          \wedge
          \frac{k^3}{\rho_{\mathrm{blk}} n^2}
        \right\}
      \Biggr].
    \end{aligned}
    \label{eq:Adaptive_G_BlockOracle}
  \end{equation}
  If \(A_b\in\mca{E}_k(a)\) for every \(b\leq q\), then
  \begin{equation}
    \begin{aligned}
      &\E\max_{b\leq q}
      \norm{\widehat{A}_b^{\mca{H},\mathrm{ad}}-A_b}^2\\
      &\quad\lesssim_{\log}
      a^2
      \wedge
      \Biggl[
        \frac{k+\log(q+1)}{n}
        +
        \left\{
          \min_{1\leq r\leq k}
          \left\{
            \frac{a^2}{r}
            +
            \frac{k^2 r}{\rho_{\mathrm{blk}} n^2}
          \right\}
          \wedge
          \frac{k^3}{\rho_{\mathrm{blk}} n^2}
        \right\}
      \Biggr].
    \end{aligned}
    \label{eq:Adaptive_H_BlockOracle}
  \end{equation}
\end{proposition}

\begin{proposition}[Adaptive row-tail Frobenius block oracle]
  \label{prop:DP_Cov_G_FrobeniusAdaptiveBlock}
  Let \(a_0>\cdots>a_J>0\), where \(a_{j+1}=a_j/2\) and
  \(a_0=O(1)\), and let \(A_1,\ldots,A_q\) be \(k\times k\)
  covariance cross-blocks associated with the \(q\) blocks.
  Suppose \(A_b\in\mca{K}_k(a)\) for every \(b\leq q\), where the unknown
  radius satisfies \(a\in[a_J,a_0]\).
  Give each block total budget \(0<\rho_{\mathrm{blk}}\leq\rho\), and set
  \(\rho_j=\rho_{\mathrm{blk}}/(J+1)\).
  If
  \[
    J+\log\left(\frac{e d n q}{\rho_{\mathrm{blk}}}\right)
    \lesssim
    \log n,
  \]
  then the adaptive scale construction satisfies
  \begin{equation}
    \begin{aligned}
      &\E\max_{b\leq q}
      \norm{
        \widehat{A}_b^{\mca{G},\mathrm{F},\mathrm{ad}}-A_b
      }_F^2\\
      &\quad\lesssim_{\log}
      k a^2
      \wedge
      k
        \min_{1\leq r\leq k}
        \left\{
          \frac{a^2}{r}
          +
          \frac{r}{n}
          +
          \frac{k r^2}{\rho_{\mathrm{blk}} n^2}
        \right\}
      \wedge
      \left(
        \frac{k^2}{n}
        +
        \frac{k^4}{\rho_{\mathrm{blk}} n^2}
      \right).
    \end{aligned}
    \label{eq:Adaptive_G_Frobenius_BlockOracle}
  \end{equation}
\end{proposition}

\begin{proof}[Proofs of \cref{prop:DP_Cov_GH_AdaptiveBlock} and
  \cref{prop:DP_Cov_G_FrobeniusAdaptiveBlock}]
  For each mode, let \(W_j\) be the minimum of the zero, structured, and
  dense scores in the table following \cref{alg:DP_Cov_GH_Adaptive},
  evaluated at \(a_j\) and \(\rho_j\), and set
  \(D_j^2=C(\log n)^C W_j\).
  Each score, and hence \(D_j\), is nonincreasing in \(j\).

  Combining the deterministic zero-branch bound with the deviation bounds in
  \cref{prop:Supp_DPCov_G_ProfileOracle},
  \cref{prop:Supp_DPCov_H_GreedyOracle},
  \cref{prop:Supp_DPCov_G_FrobeniusDeviation}, and
  \cref{prop:Supp_DPCov_GH_DenseDeviation} shows that, after allocating the
  failure probability across the three branches, the selected candidate at
  radius \(a_j\) has error at most \(D_j\) simultaneously over the \(q\)
  blocks, except with probability \(\beta_{\mathrm{ad}}\).
  Here \(\beta_{\mathrm{ad}}\leq n^{-2}\), as required by the profile
  oracle, and the factor \((\log n)^C\) absorbs the sampling term
  \(\log(q+1)/n\).
  Radial clipping changes these bounds by at most a constant factor:
  the truth lies in the radius-\(a_0\) operator ball under either
  \(\mca{K}_k(a)\), by Schur's inequality, or \(\mca{E}_k(a)\), by the
  Frobenius bound, and in the radius-\(\sqrt{k}a_0\) Frobenius ball under
  \(\mca{K}_k(a)\).

  Let \(j_*=\max\{j:a_j\geq a\}\), so that
  \(a\leq a_{j_*}<2a\).
  A union bound gives simultaneous candidate control for
  \(0\leq j\leq j_*\) except with probability at most
  \((J+1)\beta_{\mathrm{ad}}\).
  Applying \cref{lem:DP_Cov_GH_PointLepski}, with
  \(B_0=a_0\) in the operator modes and
  \(B_0=\sqrt{k}a_0\) in the Frobenius mode, yields
  \[
    \E\max_{b\leq q}\norm{\widehat A_b-A_b}_\star^2
    \lesssim
    D_{j_*}^2+(J+1)\beta_{\mathrm{ad}}B_0^2.
  \]
  Because \(\rho_j=\rho_{\mathrm{blk}}/(J+1)\) and
  \(J+1\lesssim\log n\), replacing \(\rho_j\) by the total block budget
  \(\rho_{\mathrm{blk}}\) in the three scores costs only a logarithmic
  factor.

  In the operator modes, the grid bounds above give
  \[
    (J+1)\beta_{\mathrm{ad}}a_0^2
    \lesssim
    a_J^2\wedge\frac{1}{n}
    \leq
    a^2\wedge\frac{k}{n}.
  \]
  In the Frobenius mode, the corresponding remainder is at most
  \(k(a_J^2\wedge n^{-1})\); it is bounded by the zero score and by the
  structured and dense scores, which are at least \(k/n\).
  Evaluating the three scores at \(j_*\) now gives
  \cref{eq:Adaptive_G_BlockOracle},
  \cref{eq:Adaptive_H_BlockOracle}, and
  \cref{eq:Adaptive_G_Frobenius_BlockOracle}.
\end{proof}

\subsection{Aggregation Across Scales and Adaptive Rates}

The local oracle bounds now allow us to use the same aggregation argument as for known smoothness, with an additional logarithmic cost for the radius grid.
We first check that the budget assigned to each block satisfies the local assumptions at every active scale.
We then control the scales omitted by the operator procedure and combine the center and outer-block errors.

Put \(x=d/(\rho n^2)\), \(J_{\mathrm{sq}}=\log_2 d-1\), and, at a
dyadic increment scale \(k\), \(q=3(N_{2k}-1)\).
By \cref{alg:DP_Cov_GH_Adaptive}, the total budget for one block at scale
\(k\) is
\[
  \rho_{\mathrm{blk}}
  :=
  (J_k+1)\rho_{k,j}
  =
  \frac{\rho}{4J_{\mathrm{sq}}N_{2k}}
  \asymp
  \frac{\rho k}{dJ_{\mathrm{sq}}}.
\]
The radius grid satisfies \(J_k\lesssim\log(2k)\), while the standing
regime gives
\[
  J_k+
  \log\left(\frac{e d n q}{\rho_{\mathrm{blk}}}\right)
  \lesssim
  \log n.
\]
Thus the radius-grid bounds apply at every scale.
For the operator procedures, the radius grid bound above gives
\(\log(q+1)\lesssim\log n\).
Choosing the constant in the global operator rule sufficiently large therefore
ensures \cref{eq:Supp_DPCov_GH_OperatorScaleCondition} for all sufficiently
large \(n\).

The following lemma controls the sampling branch and the zero estimator at
inactive operator scales.

\begin{lemma}[Sampling and inactive operator scale bounds]
  \label{lem:Proof_GH_ScaleBounds}
  Fix \(\alpha>0\).
  Uniformly over the dyadic scales \(1\leq k\leq d/4\),
  \begin{equation}
    k^{-2\alpha}
    \wedge
    \frac{k+\log(3N_{2k})}{n}
    \lesssim_{\log}
    n^{-\frac{2\alpha}{2\alpha+1}}
    \wedge\frac{d}{n}.
    \label{eq:Proof_GH_SamplingBound}
  \end{equation}
  Every scale excluded by the global operator rule in
  \cref{alg:DP_Cov_GH_Adaptive} satisfies
  \begin{equation}
    k^{-2\alpha}
    \lesssim_{\log}
    n^{-\frac{2\alpha}{2\alpha+1}}
    \wedge\frac{d}{n}.
    \label{eq:Proof_GH_InactiveScale}
  \end{equation}
  The implicit constants and logarithmic powers can be chosen uniformly
  when \(\alpha\) ranges over a fixed compact subset of \((0,\infty)\).
\end{lemma}

\begin{proof}
  \emph{Sampling scales.}
  The functions \(k^{-2\alpha}\) and \(k/n\) cross at
  \(k=n^{1/(2\alpha+1)}\).
  Optimizing their minimum over \(1\leq k\leq d/4\) gives
  \[
    n^{-\frac{2\alpha}{2\alpha+1}}
    \wedge\frac{d}{n}
  \]
  up to constants depending only on \(\alpha\).
  Moreover,
  \[
    k^{-2\alpha}
    \wedge
    \frac{k+\log(3N_{2k})}{n}
    \leq
    \left(k^{-2\alpha}\wedge\frac{2k}{n}\right)
    +
    \left(k^{-2\alpha}\wedge\frac{C\log d}{n}\right).
  \]
  The second term is bounded both by \(C\log d/n\leq C d/n\) and, up to a
  logarithmic factor, by \(n^{-2\alpha/(2\alpha+1)}\), proving
  \cref{eq:Proof_GH_SamplingBound}.

  \emph{Inactive operator scales.}
  At an inactive scale, the global operator rule gives
  \(k\gtrsim n/(\log n)^C\).
  If \(d\geq n^{1/(2\alpha+1)}\), use
  \(k^{-2\alpha}\lesssim
  (\log n)^C n^{-2\alpha/(2\alpha+1)}\).
  If \(d<n^{1/(2\alpha+1)}\), the existence of such a \(k\leq d\)
  implies \((\log n)^C d/n\gtrsim1\), so
  \(k^{-2\alpha}\lesssim_{\log}d/n\).
  This proves \cref{eq:Proof_GH_InactiveScale}.
\end{proof}

The local oracles and these scale bounds yield the global estimates below.
Since \(k_0=1\) and \(\log d=O(\log n)\), the center condition in
\cref{lem:Proof_DyadicSquareError} holds for both constructions.
All constants and logarithmic factors below are uniform over
\(\alpha\in[\underline{\alpha},\overline{\alpha}]\).

\begin{proof}[Proof of \cref{thm:DP_Cov_GH_Adaptive}]
  Fix \(\mca{J}\in\{\mca{G},\mca{H}\}\) and
  \(\alpha\in[\underline{\alpha},\overline{\alpha}]\).
  Put \(a_k=C_{\mca{J}}k^{-\alpha}\).
  Every scale-\(k\) square belongs to \(\mca{K}_k(a_k)\) for
  \(\mca{J}=\mca{G}\), or to \(\mca{E}_k(a_k)\) for
  \(\mca{J}=\mca{H}\), and the radius grid brackets \(a_k\) within a
  factor two.

  \emph{Local scale risk.}
  At an active operator scale, apply the local oracle with
  \(\rho_{\mathrm{blk}}\asymp\rho k/(dJ_{\mathrm{sq}})\) in
  \cref{eq:Adaptive_G_BlockOracle} and
  \cref{eq:Adaptive_H_BlockOracle}.
  Using
  \[
    u\wedge(v+w)
    \leq
    (u\wedge 2v)+(u\wedge 2w),
  \]
  and \cref{lem:Proof_GH_ThreeRegimeReduction}, the expected maximum
  squared block error \(R_{\mca{J},k}\) satisfies
  \begin{equation}
    R_{\mca{J},k}
    \lesssim_{\log}
    \left\{
      k^{-2\alpha}
      \wedge
      \frac{k+\log(3N_{2k})}{n}
    \right\}
    +Q_k(k^{-2\alpha},xk).
    \label{eq:Proof_GH_AdaptiveLocalRisk}
  \end{equation}

  \emph{Uniform scale risk.}
  By \cref{eq:Rate_Q}, the supremum of the second term in
  \cref{eq:Proof_GH_AdaptiveLocalRisk} is at most
  \(\mca{Q}_{\alpha,d}(x)\).
  At an inactive operator scale, the zero estimator has squared error at most
  \(a_k^2\lesssim k^{-2\alpha}\).
  The two conclusions of \cref{lem:Proof_GH_ScaleBounds} therefore
  control the first term at active operator scales and the zero error at
  inactive operator scales.
  Uniformly over the fixed smoothness interval,
  \begin{equation}
    \sup_k R_{\mca{J},k}
    \lesssim_{\log}
    n^{-\frac{2\alpha}{2\alpha+1}}
    \wedge\frac{d}{n}
    +
    \mca{Q}_{\alpha,d}(x).
    \label{eq:Proof_GH_AdaptiveUniformRisk}
  \end{equation}

  \emph{Aggregation.}
  Apply \cref{lem:Proof_DyadicSquareError} with \(k_+=d/2\).
  The statistical term \(n^{-2\alpha/(2\alpha+1)}\wedge d/n\) is at least
  \(n^{-1}\).
  Moreover, the standing regime gives \(d\geq4\) and \(x<1\) for all
  sufficiently large \(n\); taking \(k=r=1\) in \cref{eq:Rate_Q}
  gives \(\mca{Q}_{\alpha,d}(x)\geq x\).
  Thus the center term with \(k_0=1\) is absorbed by
  \cref{eq:Proof_GH_AdaptiveUniformRisk}.
  The privacy guarantee follows from
  \cref{prop:Supp_DPCov_GH_AdaptivePrivacy}.
\end{proof}

\begin{proof}[Proof of \cref{thm:DP_Cov_G_Frobenius_Adaptive}]
  \emph{Uniform scale risk.}
  The Frobenius block oracles apply at every scale, so no operator scale
  truncation is needed.
  At every dyadic increment scale, put
  \(a_k=C_{\mca{G}}k^{-\alpha}\).
  The row-tail condition places every square block in
  \(\mca{K}_k(a_k)\), and the radius grid brackets \(a_k\) within a
  factor two.
  Applying \cref{prop:DP_Cov_G_FrobeniusAdaptiveBlock} with
  \(q=3(N_{2k}-1)\) and
  \(\rho_{\mathrm{blk}}\asymp\rho k/(dJ_{\mathrm{sq}})\), and then using
  \cref{lem:Proof_G_FrobeniusThreeBranchReduction} with
  \(y=CJ_{\mathrm{sq}}x\), gives
  \begin{equation}
    \max_{B\in\mca{B}_k}
    \E\norm{\widehat{\Sigma}_B-\Sigma_B}_F^2
    \lesssim_{\log}
    k\left\{
      n^{-\frac{2\alpha+1}{2\alpha+2}}
      \wedge\frac{d}{n}
      +
      x^{\frac{2\alpha+1}{2\alpha+3}}
      \wedge x d^2
    \right\}.
    \label{eq:Proof_G_AdaptiveFrobeniusUniformRisk}
  \end{equation}
  The auxiliary bounds \cref{eq:Proof_G_FrobeniusSamplingBound} and
  \cref{eq:Proof_G_FrobeniusPrivacyBound} remain valid with
  \(y=CJ_{\mathrm{sq}}x\): since \(J_{\mathrm{sq}}\lesssim\log n\),
  this replacement changes both rate terms by only a uniform logarithmic
  factor over the fixed smoothness interval.

  \emph{Aggregation.}
  Apply \cref{lem:Proof_DyadicSquareError} with \(k_+=d/2\).
  Since \(N_{2k} k/d=1/2\), its scale sum only contributes a
  logarithmic factor.
  The statistical term \(n^{-(2\alpha+1)/(2\alpha+2)}\wedge d/n\) is again at
  least \(n^{-1}\).
  Since \(x<1\) for all sufficiently large \(n\),
  \[
    x
    \leq
    x^{\frac{2\alpha+1}{2\alpha+3}}
    \wedge x d^2
  \]
  because \(d\geq1\).
  Thus the center term with \(k_0=1\) is absorbed by
  \cref{eq:Proof_G_AdaptiveFrobeniusUniformRisk}.
  The privacy guarantee follows from
  \cref{prop:Supp_DPCov_GH_AdaptivePrivacy}.
\end{proof}

  \clearpage

\section{Proof of the DP van Trees Inequality}
\label{sec:VanTrees}

In this section, we prove the DP van Trees inequality in \cref{thm:VanTrees__zCDP}.
The key ingredient is the information contraction bound in \cref{lem:VanTrees__zCDP_FisherInfoBound}, which limits the information provided by any zCDP mechanism.
Combining this bound with the multivariate van Trees inequality yields the desired result.

We first introduce the following regularity assumption.

\begin{assumption}[DQM and prior regularity]
  \label{assu:VanTreesRegularity}
  Let $\Theta\subseteq\R^p$ be compact with piecewise smooth boundary, and let $(P_\theta)_{\theta\in\Theta}$ be a dominated model on a standard Borel sample space.
  The model is differentiable in quadratic mean at every interior point of $\Theta$, with a jointly measurable score $s(x;\theta)$ satisfying
  \[
    \E_\theta s(X;\theta)=0,
    \qquad
    I_x(\theta)=\E_\theta\zk{s(X;\theta)s(X;\theta)^\top},
    \qquad
    \int_\Theta\norm{I_x(\theta)}\dd\pi(\theta)<\infty.
  \]
  The density map $\theta\mapsto p_\theta$ is continuously differentiable as an $L^1$-valued map on a neighborhood of $\Theta$, with derivative $p_\theta s(\cdot;\theta)$.
  The prior $\pi$ has a continuously differentiable density, vanishes on the boundary of $\Theta$, and has finite Fisher information matrix
  \[
    J_\pi
    =
    \E_\pi\zk{
      \nabla_\theta\log\pi(\theta)
      \nabla_\theta\log\pi(\theta)^\top
    }.
  \]
\end{assumption}

Differentiability in quadratic mean gives the score and Fisher information used below; see \citet[Chapter~7]{vaart1998_AsymptoticStatistics}.
It is preserved by any parameter-independent Markov kernel $T$ from the sample space to a standard Borel output space, and the induced output experiment has score
\begin{equation}
  \label{eq:VanTrees_OutputScore}
  s_T(T;\theta)
  =
  \E_\theta\zk{s(X;\theta)\mid T}.
\end{equation}
These facts follow from the information-loss theorem of \citet[Section~7, Proposition~4]{lecam1988_PreservationLocal}.
Applied to an i.i.d. sample, \cref{eq:VanTrees_OutputScore} holds with the product score $s(X;\theta)=\sum_{i=1}^n s(x_i;\theta)$.
Thus, the score identity applies to discrete, continuous, and mixed outputs without requiring the mechanism to have an output density.

\begin{lemma}[Multivariate van Trees inequality]
  \label{lem:VanTreesInequality}
  Let $\theta \in \Theta \subseteq \R^p$ be a random vector with prior distribution $\pi$.
  Let $\hat{\theta}(T)$ be an estimator based on $T=T(X)$; in particular, $T$ may be the estimator itself.
  Under \cref{assu:VanTreesRegularity}, we have
  \begin{equation}
    \E_{\pi} \E_{X|\theta} \norm{(\hat{\theta}(T(X)) - \theta)}^2_2
    \geq \frac{p^2}{\int \Tr I_T(\theta) \dd \pi(\theta) + \Tr J_\pi}.
  \end{equation}
\end{lemma}
This is the trace form of the multivariate van Trees inequality of \citet[Theorem~1]{gill1995_ApplicationsVan}.
It remains valid if the output Fisher information is singular, with the risk interpreted as an extended nonnegative expectation when necessary.

The following key lemma shows that the information provided by a $\rho$-zCDP mechanism is restricted by the privacy.

\begin{lemma}[Fisher Information Contraction under zCDP]
  \label{lem:VanTrees__zCDP_FisherInfoBound}
  Under \cref{assu:VanTreesRegularity}, let $T$ be an arbitrary randomized $\rho$-zCDP mechanism based on $n$ i.i.d.\ samples from $P_\theta$, with a standard Borel output space.
  Then,
  \begin{equation}
    \Tr I_T(\theta) \leq (e^{2\rho}-1) n^2 \norm{I_{x}(\theta)}.
  \end{equation}
\end{lemma}
\begin{proof}
  Let $s(X)=\sum_{i=1}^n s(x_i;\theta)$ be the product score.
  By \cref{eq:VanTrees_OutputScore}, the output score satisfies $s_T(T)=\E_\theta\{s(X)\mid T\}$,
  so
  \begin{align*}
    \Tr I_T(\theta) &= \E_{\theta} \ang{ s_T(T) ,  \E \xk{s(X) \mid T } }
    = \E_{\theta} \ang{ s_T(T) , s(X) },
  \end{align*}
  where we use the fact that
  \begin{equation*}
    \E \zk{V W} = \E \zk{\E \xk{V W | T}} = \E \zk{V \E(W|T)},\quad V = V(T).
  \end{equation*}
  Now, we can further decompose
  \begin{align*}
    \Tr I_T(\theta) &= \E_{\theta} \ang{ s_{T}(T), s(X) }
    = \sum_{i \in [n]} \E_{\theta} \ang{s_{T}(T), s(x_i)}
    \eqqcolon \sum_{i \in [n]} \E_{\theta}  A_i,
  \end{align*}
  where \( s(x_i) = \nabla_\theta \log p_{\theta}(x_i) \) is the score function of each data point.
  Using \cref{lem:Lower__AiControl} and noting that $s_{T}(T)$ is also $\rho$-zCDP by post-processing,
  we have
  \begin{equation*}
    \abs{\E_{\theta}  A_i}
    \leq \xk{(e^{2\rho}-1) \norm{I_{x}(\theta)} \cdot \E_{\theta} \norm{ s_{T}(T)}_2^2}^{\hf}
    = \xk{(e^{2\rho}-1) \norm{I_{x}(\theta)} \Tr I_T(\theta)}^{\hf}.
  \end{equation*}
  Plugging this back, we get
  \begin{equation*}
    \Tr I_T(\theta)
    \leq n \xk{(e^{2\rho}-1) \norm{I_{x}(\theta)} \Tr I_T(\theta)}^{\hf}.
  \end{equation*}
  Rearranging gives the desired result.
\end{proof}

We now restate \cref{thm:VanTrees__zCDP} under \cref{assu:VanTreesRegularity} and give its proof.

\begin{theorem}[DP van Trees inequality; restatement of \cref{thm:VanTrees__zCDP}]
  \label{thm:VanTrees__zCDP_Restated}
  Suppose $(P_\theta)_{\theta\in\Theta}$ and the prior $\pi$ satisfy \cref{assu:VanTreesRegularity}.
  Let $\hat{\theta}$ be an arbitrary randomized $\rho$-zCDP estimator computed from $n$ i.i.d.\ samples from $P_\theta$.
  \begin{equation}
    \label{eq:VanTrees__zCDP_Restated}
    \begin{aligned}
      \E_\pi\E_\theta\norm{\hat{\theta}-\theta}_2^2
      &\geq
      \frac{p^2}{I+\Tr J_\pi},\\
      I
      &=
      \left\{
        C_\rho\rho n^2
        \int_\Theta\norm{I_x(\theta)}\dd\pi(\theta)
      \right\}
      \wedge
      \left\{
        n\int_\Theta\Tr I_x(\theta)\dd\pi(\theta)
      \right\},
    \end{aligned}
  \end{equation}
  where $C_\rho=(e^{2\rho}-1)/\rho$.
\end{theorem}
\begin{proof}
  By \cref{eq:VanTrees_OutputScore}, the output experiment of $\hat{\theta}$ is DQM, with Fisher information $I_{\hat{\theta}}(\theta)$.
  The conditional-score identity and the $L^2$ contraction property of conditional expectation give
  \[
    \Tr I_{\hat{\theta}}(\theta)
    \leq
    n\Tr I_x(\theta).
  \]
  The privacy contraction bound in \cref{lem:VanTrees__zCDP_FisherInfoBound} gives
  \[
    \Tr I_{\hat{\theta}}(\theta)
    \leq
    (e^{2\rho}-1)n^2\norm{I_x(\theta)}.
  \]
  Integrating the smaller of these bounds with respect to $\pi$ and applying \cref{lem:VanTreesInequality} proves \cref{eq:VanTrees__zCDP_Restated}.
\end{proof}

\begin{lemma}
  \label{lem:Lower__AiControl}
  Suppose that $A_i = \ang{M(X), s(x_i)}$ for some $\rho$-zCDP mechanism $M(\cdot)$ and a mean-zero vector function $\E s(x_i) = 0$.
  Then,
  \begin{equation*}
    \abs{\E A_i}
    \leq
    \xk{(e^{2\rho}-1) \norm{I_{x}(\theta)} \cdot \E \norm{M(X)}_2^2}^{\hf}.
  \end{equation*}
\end{lemma}
\begin{proof}
  Now, let us take independent copies \( (x_1',\dots,x_n') \)
  and denote by
  \[
    X_i' = (x_1,\dots,x_{i-1}, x_i', x_{i+1},\dots,x_n)
  \]
  the dataset where we replace the $i$-th data point by its independent copy.
  We also denote
  \begin{equation*}
    A_i' = \ang{M(X_i'), s(x_i)}.
  \end{equation*}
  By the independence between \( X' \) and \( x_i \), we have
  \begin{equation*}
    \E A_i' = \E \ang{M(X'_i), s(x_i)} = 0.
  \end{equation*}

  Now, we bound the difference between \( \E A_i \) and \( \E A_i' \).
  Using Jensen's inequality, we have
  \begin{equation*}
    \xk{\E A_i - \E A_i'}^2 \leq \E \zk{\E \xk{A_i - A_i' | X,X'}}^2.
  \end{equation*}
  Now, since $M(\cdot)$ is $\rho$-zCDP, the data processing inequality implies that
  \begin{equation*}
    D_{\alpha|X,X'}(A_i \| A_i') \leq \alpha \rho,\quad \forall \alpha > 1,
  \end{equation*}
  so from the definition of chi-squared divergence, we have
  \begin{equation*}
    \chi^2_{|X,X'}(A_i \| A_i') \leq \exp(2\rho) - 1.
  \end{equation*}
  Using \cref{lem:MeanBoundViaChi2}, we have
  \begin{equation*}
    \zk{\E \xk{A_i - A_i' | X,X'}}^2
    \leq \chi^2_{|X,X'}(A_i \| A_i') \E \xk{\abs{A_i'}^2 | X,X'}
    \leq (e^{2\rho}-1) \E \xk{\abs{A_i'}^2 | X,X'}.
  \end{equation*}
  Therefore, we obtain
  \begin{equation*}
    \xk{\E A_i - \E A_i'}^2
    \leq (e^{2\rho}-1) \E \abs{A_i'}^2.
  \end{equation*}
  Finally, noticing that \( (X_i',x_i) \) and \( (X,x_i') \) have the same distribution,
  we find
  \begin{align*}
    \E \abs{A_i'}^2 &= \E \ang{M(X_i'), s(x_i)}^2
    = \E \Tr \zk{ s(x_i) s(x_i)^{\top} M(X_i') M(X_i')^{\top} } \\
    &= \Tr \zk{\E s(x_i) s(x_i)^{\top}} \zk{\E M(X_i') M(X_i')^{\top}} \\
    &= \Tr \zk{I_{x}(\theta) \E M(X_i') M(X_i')^{\top}} \\
    &\leq \norm{I_{x}(\theta)} \Tr \zk{\E M(X_i') M(X_i')^{\top}} \\
    &= \norm{I_{x}(\theta)} \cdot \E \norm{M(X)}_2^2.
  \end{align*}
  Therefore,
  \begin{equation*}
    \abs{\E A_i}
    =
    \abs{\E A_i - \E A_i'}
    \leq
    \xk{(e^{2\rho}-1) \norm{I_{x}(\theta)} \cdot \E \norm{M(X)}_2^2}^{\hf}.
  \end{equation*}
\end{proof}

  \clearpage

\section{Proof of the Minimax Lower Bounds}
\label{sec:ProofLower}

This section applies the DP van Trees inequality to obtain the minimax lower bounds for each covariance class.
The central step is to choose matrix priors that respect the distinct geometries of the pointwise-decay, row-tail, and separated-block covariance classes while keeping both the likelihood and prior Fisher information under control.
For the Frobenius lower bounds, we use an entrywise product prior supported on a band of width \(k\) around the diagonal.
For the separated-block class under Schatten loss, we place independent smooth priors supported on operator norm balls on disjoint \(k\times k\) off-diagonal blocks, whereas for the pointwise-decay and row-tail classes under operator loss, we use priors of adjustable rank constructed from disjoint pairs of adjacent coordinate blocks and \(r\) groups of shared column directions.
The final subsection constructs the smooth prior supported on an operator norm ball that is required by the separated-block argument.

\subsection{Frobenius Loss for \texorpdfstring{\(\mca{H}_\alpha\) and \(\mca{G}_\alpha\)}{H-alpha and G-alpha}}
\label{subsec:ProofLower__Frobenius}

It suffices to prove the lower bound due to the privacy constraint, since the non-private minimax lower bound is already known~\citep{cai2010_OptimalRates}.
We set $x \sim N(0,\Sigma)$ and use the DP van Trees inequality in \cref{thm:VanTrees__zCDP_Restated}.
Let us take the prior distribution $\Pi$ over covariance matrices, where $k < d$ is a parameter to be chosen later:
\begin{itemize}
  \item $\Sigma_{ii} = 1$.
  \item For \( 1 \leq \abs{i-j} \leq k \), set
  \[
    \Sigma_{ij} = \Sigma_{ji} = c k^{-(\alpha+1)} u_{ij},
  \]
  where \( u_{ij} \) are i.i.d.\ from density \( \cos^2(\pi t / 2) \) on \( [-1,1] \).
  \item The other entries are zero.
\end{itemize}
We denote by $\Pi_{jl}$ the marginal distribution of $\Sigma_{jl}$ and its density.
It is easy to see that \( \Sigma \in \mca{H}_\alpha \) provided that $c$ is small enough.
Moreover, $\Sigma$ is strictly diagonally dominant and $\norm{\Sigma^{-1}}$ is bounded by some constant.
Let us denote $P = \dk{(i, j) : 1 \leq j-i \leq k}$ be the upper-triangular band indices containing the free parameters,
which has the cardinality $\abs{P} \asymp d k$.

Without loss of generality, we can assume that the output of the estimator $\hat{\Sigma}$ is symmetric and $\hat{\Sigma}_{ii} = \Sigma_{ii}$, $\hat{\Sigma}_{ij} = 0$ for $\abs{i-j} > k$.
Otherwise, we can project $\hat{\Sigma}$ to such a space without increasing the error.
Then,
\begin{equation*}
  \norm{\hat{\Sigma} - \Sigma}_F^2 \geq 2 \sum_{(j,l) \in P} (\hat{\Sigma}_{jl} - \Sigma_{jl})^2,
\end{equation*}
so it suffices to lower bound the last term.

Under the standing condition \(\rho\lesssim1\), we have \(C_\rho=(e^{2\rho}-1)/\rho\lesssim1\).
Applying \cref{thm:VanTrees__zCDP_Restated} over the parameters in $P$ therefore gives
\begin{equation*}
  \sum_{(j,l) \in P} (\hat{\Sigma}_{jl} - \Sigma_{jl})^2
  \geq \frac{\abs{P}^2}{C \rho n^2 \int \norm{I_{x}(\Sigma_P)} \dd \Pi(\Sigma_P) + \Tr J_{\Pi}}.
\end{equation*}
For $a=(i,j)\in P$, let $D_a=e_ie_j^\T+e_je_i^\T$ be the derivative of $\Sigma$ with respect to the free coordinate $\Sigma_{ij}$.
The Fisher information for the free parameter vector is the pullback of the bilinear form in \cref{subsec:FisherNormal}:
\begin{equation*}
  I_x(\Sigma_P)_{ab}
  =
  \frac{1}{2}\Tr\zk{\Sigma^{-1}D_a\Sigma^{-1}D_b}.
\end{equation*}
The matrices $(D_a)_{a\in P}$ are pairwise orthogonal in the Frobenius inner product, and each satisfies $\norm{D_a}_F^2=2$.
Consequently, for $D(v)=\sum_{a\in P}v_aD_a$,
\begin{equation*}
  v^\T I_x(\Sigma_P)v
  =
  \frac{1}{2}\norm{\Sigma^{-1/2}D(v)\Sigma^{-1/2}}_F^2
  \leq
  \norm{\Sigma^{-1}}^2\norm{v}_2^2,
\end{equation*}
and hence $\norm{I_x(\Sigma_P)}\lesssim1$.
On the other hand, using $\Pi_{jl} = c k^{-(\alpha+1)} u_{ij}$, we have
\begin{equation*}
  \E_{\Pi} \zk{\frac{\Pi_{jl}'(\Sigma_{jl})}{\Pi_{jl}(\Sigma_{jl})}}^2 = c^{-2} \pi^2 k^{2(\alpha+1)} = C k^{2(\alpha+1)},
\end{equation*}
so
\begin{equation*}
  \Tr J_{\Pi} = \sum_{(j,l) \in P} \E_{\Pi} \zk{\frac{\Pi_{jl}'(\Sigma_{jl})}{\Pi_{jl}(\Sigma_{jl})}}^2 \lesssim dk\cdot k^{2(\alpha+1)} = d k^{2\alpha + 3}.
\end{equation*}
Combining the above estimates yields
\begin{equation*}
  \sum_{(j,l) \in P} (\hat{\Sigma}_{jl} - \Sigma_{jl})^2 \gtrsim \frac{d^2 k^2}{\rho n^2 + d k^{2\alpha + 3}}.
\end{equation*}
Taking \( k \asymp (\rho n^2 / d)^{1/(2\alpha + 3)} \wedge d \), we obtain the desired lower bound:
\begin{equation*}
  \frac{1}{d} \E_{\Pi} \norm{\hat{\Sigma} - \Sigma}_F^2 \gtrsim \xk{\frac{d}{\rho n^2}}^{\frac{2\alpha + 1}{2\alpha + 3}} \wedge \frac{d^3}{\rho n^2}.
\end{equation*}
This proves the bound for \(\mca{H}_\alpha\).
After a fixed reduction of the pointwise-decay radius, \cref{prop:ClassInclusions} places the same Gaussian testing family and band prior in \(\mca{P}_\alpha^{\mca{G}}\), without changing either lower-bound calculation up to constants.
This completes the proof for both classes.

\subsection{Schatten Loss for \texorpdfstring{\(\mca{F}_\alpha\)}{F-alpha}}
\label{subsec:ProofLower__Operator}

We will apply \cref{thm:VanTrees__zCDP_Restated} for both the statistical rate and the DP rate.
We set $x \sim N(0,\Sigma)$.
However, we have to construct a more delicate prior distribution over covariance matrices to fit the hypothesis class $\mca{F}_\alpha$, as the bounded operator norm condition cannot be optimally guaranteed by independent priors over each entry of $\Sigma$.
Therefore, \cref{lem:Lower__PriorDistribution} is crucial to construct such a prior distribution of the covariance matrix with Fisher information of desired order.

To construct the prior distribution, recall the blocks \(B_{k;l,1}\) defined in \cref{subsec:DyadicBlockStructure}.
Define the following prior distribution $\Pi$ over covariance matrices, where $k < d$ is a parameter to be chosen later and $c$ is a small constant:
\begin{itemize}
  \item $\Sigma_{ii} = 1$.
  \item For \(B_{k;2l,1}\), set \(\Sigma[B_{k;2l,1}]=ck^{-\alpha} W_l\) for i.i.d.\ \(W_l \sim\Upsilon_k\) in \cref{lem:Lower__PriorDistribution} and symmetrize over \(B_{k;2l,1}^{\T}\).
  \item The other entries are zero.
\end{itemize}
Denote by \(P=\bigsqcup_l B_{k;2l,1}\) the indices of the free parameters, so \(\abs{P}\asymp dk\).
We refer to \cref{fig:Lower__BlockPrior} for an illustration of the prior distribution.

\begin{figure}
  \centering
  \includegraphics[width=0.4\textwidth]{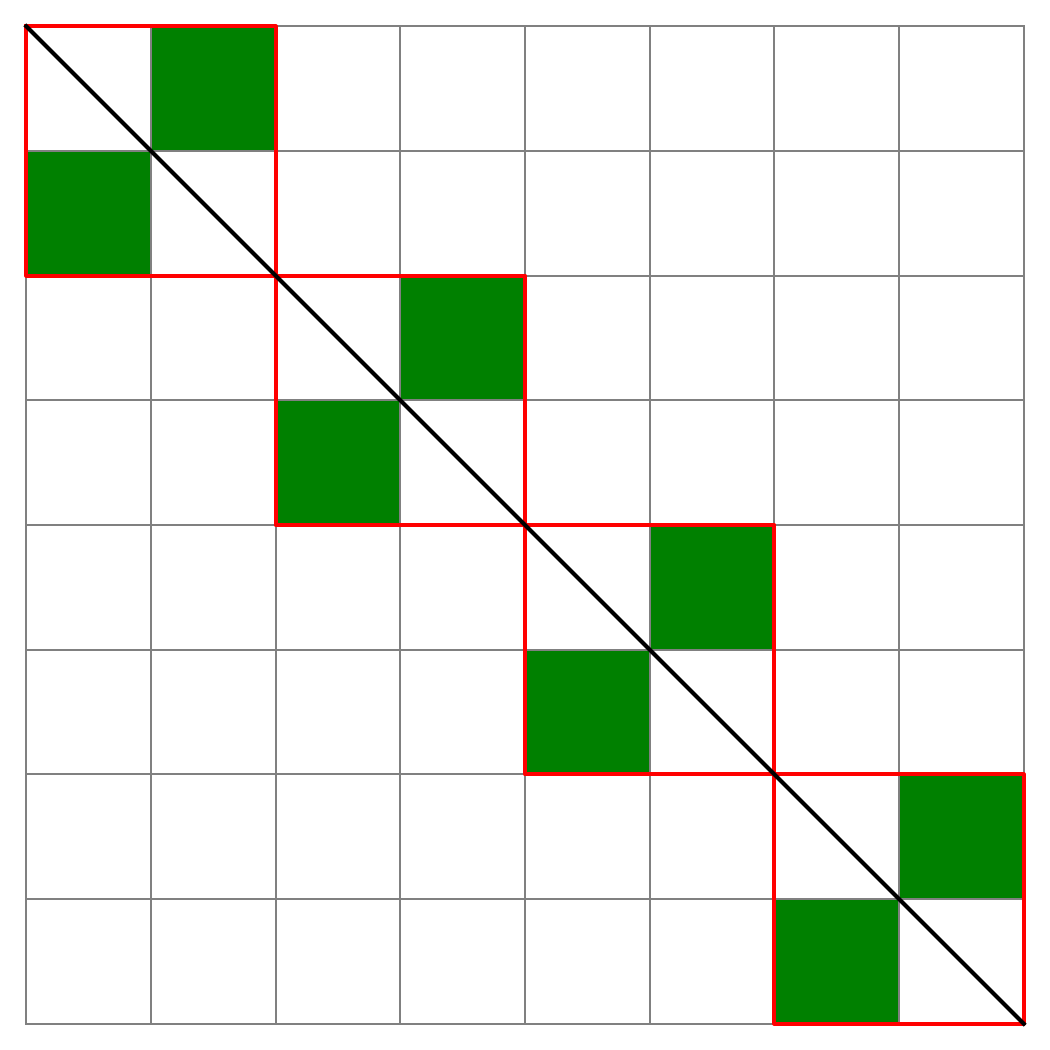}
  \caption{An illustration of the blockwise prior distribution.}
  \label{fig:Lower__BlockPrior}
\end{figure}

To show that \(\Sigma\in\mca{F}_\alpha\), note that any \(k\)-off-diagonal block \(D_k\) intersects at most one \(B_{k;2l,1}\), since these blocks have disjoint supports.
Moreover, \(D_k \cap B_{k;2l,1}\) is a sub-block of \(B_{k;2l,1}\).
Therefore,
\begin{equation*}
  \norm{\Sigma[D_k]} \leq c k^{-\alpha} \max_l \norm{W_l} \leq c k^{-\alpha}.
\end{equation*}
It is easy to see that $\Sigma \in \mca{F}_\alpha$ provided that $c$ is small enough.
Moreover, $\norm{\Sigma^{-1}}$ is also bounded by some constant.

Now, let us compute the lower bound in \cref{thm:VanTrees__zCDP_Restated}.
Since \(\Pi\) is a product of i.i.d.\ distributions over \(B_{k;2l,1}\), \cref{lem:Lower__PriorDistribution} gives
\begin{equation}
  \label{eq:Lower__Pi_PriorFisherInfo}
  \Tr J_{\Pi} \lesssim \sum_{l} k^{2\alpha} \cdot \Tr J_{\Upsilon_k} \lesssim \frac{d}{k} \cdot k^{2\alpha} \cdot k^3 = d k^{2\alpha+2}.
\end{equation}
For each free entry $a=(i,j)\in P$, let $D_a=e_ie_j^\T+e_je_i^\T$.
The same Jacobian calculation gives
\begin{equation*}
  I_x(\Sigma_P)_{ab}
  =
  \frac{1}{2}\Tr\zk{\Sigma^{-1}D_a\Sigma^{-1}D_b}.
\end{equation*}
The matrices $(D_a)_{a\in P}$ are again pairwise orthogonal in the Frobenius inner product, and each satisfies $\norm{D_a}_F^2=2$.
Together with $\norm{\Sigma^{-1}}\lesssim1$, this gives
\begin{equation*}
  \norm{I_x(\Sigma_P)}\lesssim1,
  \qquad
  \Tr I_x(\Sigma_P)
  \leq
  \abs{P}\norm{I_x(\Sigma_P)}
  \lesssim dk.
\end{equation*}
Plugging the above estimates into \cref{thm:VanTrees__zCDP_Restated} and noticing $\abs{P} \asymp dk$, we have
\begin{align*}
  \frac{1}{d}\sup_{\Sigma \in \mca{F}_\alpha}  \E_{\Sigma} \norm{\hat{\Sigma} - \Sigma}^2_F
  & \geq \frac{1}{d} \E_{\Pi} \E_{\Sigma} \norm{\hat{\Sigma}_P - \Sigma_P}^2_F
  \geq \frac{d k^2}{\rho n^2 \wedge n d k + d k^{2\alpha + 2}} \\
  & = \frac{d k^2}{\rho n^2 + d k^{2\alpha + 2}} \vee \frac{k}{n + k^{2\alpha + 1}}.
\end{align*}
Choosing
\begin{equation*}
  k_1 \asymp \xk{\frac{\rho n^2}{d}}^{\frac{1}{2\alpha + 2}} \wedge d,\quad k_2 \asymp n^{\frac{1}{2\alpha + 1}} \wedge d,
\end{equation*}
we have
\begin{equation*}
  \frac{1}{d}\sup_{\Sigma \in \mca{F}_\alpha}  \E_{\Sigma} \norm{\hat{\Sigma} - \Sigma}^2_F
  \gtrsim \xk{\frac{d}{\rho n^2}}^{\frac{\alpha}{\alpha + 1}} \wedge \frac{d^3}{\rho n^2} + n^{-\frac{2\alpha}{2\alpha + 1}} \wedge \frac{d}{n}.
\end{equation*}
Finally, using the fact that
\begin{math}
  \norm{A} \geq d^{-\frac{1}{q}} \normsch{A} \geq d^{-\frac{1}{2}} \norm{A}_F
\end{math}
yields
\begin{equation*}
  d^{-\frac{2}{q}}\sup_{\Sigma \in \mca{F}_\alpha}  \E_{\Sigma} \normsch{\hat{\Sigma} - \Sigma}^2
  \geq \frac{1}{d}\sup_{\Sigma \in \mca{F}_\alpha}  \E_{\Sigma} \norm{\hat{\Sigma} - \Sigma}^2_F
  \gtrsim \xk{\frac{d}{\rho n^2}}^{\frac{\alpha}{\alpha + 1}} \wedge \frac{d^3}{\rho n^2} + n^{-\frac{2\alpha}{2\alpha + 1}} \wedge \frac{d}{n}.
\end{equation*}

\subsection{Operator Loss for \texorpdfstring{\(\mca{H}_\alpha\) and \(\mca{G}_\alpha\)}{H-alpha and G-alpha}}

By \cref{prop:ClassInclusions}, we may choose a fixed \(M_{\mca{H}}>0\) such that \(C_\alpha M_{\mca{H}}\leq M\) and
\begin{equation*}
  \mca{P}_\alpha^{\mca{H}}(K,M_0,M_{\mca{H}})
  \subseteq
  \mca{P}_\alpha^{\mca{G}}(K,M_0,M).
\end{equation*}
Consequently, any lower bound established on such a pointwise-decay submodel transfers to the row-tail class under the same loss.

We therefore seek a family of covariance matrices in the pointwise-decay class that is difficult to estimate under operator loss.
We construct this family using low-rank perturbations within off-diagonal blocks, with repeated column directions allowing the entries to contribute coherently to the operator norm.
The block width \(k\) and rank parameter \(r\) control the size and complexity of these perturbations.
Varying both parameters yields the lower bound below and, after optimization, the different regimes under operator loss.

\begin{proposition}[Low-Rank Gaussian Submodel]
  \label{prop:Lower__GH_LowRankGaussian}
  Fix \(\alpha,K,M_0,M>0\) and integers
  \(1\leq r\leq k\leq\lfloor d/4\rfloor\).
  \begin{equation}
    \inf_{\widehat{\Sigma}\in\mca{M}_\rho}
    \sup_{P\in\mca{P}_\alpha^{\mca{H}}}
    \E_P \norm{\widehat{\Sigma}-\Sigma(P)}^2
    \gtrsim
    \min\xk{
      \frac{k^{-2\alpha}}{r},
      \frac{dkr}{\rho n^2}
    }.
    \label{eq:Lower__GH_LowRankGaussian}
  \end{equation}
\end{proposition}

\begin{proof}
  Let \(N=\lfloor d/(2k)\rfloor\), and partition the first \(2Nk\) coordinates into disjoint adjacent pairs \(L_\ell,R_\ell\), each of size \(k\).
  Since \(k\leq d/4\), we have \(Nk\asymp d\).
  Partition the coordinates of every \(R_\ell\) into \(r\) groups of sizes between \(k/(2r)\) and \(2k/r\), and let
  \(H\in\{0,1\}^{k\times r}\) be their incidence matrix.
  Thus \(H^\T H\) is diagonal and
  \(\norm{H}\leq\sqrt{2k/r}\).

  Put
  \[
    a_k=c_0 k^{-\alpha},
    \qquad
    b_k=\frac{a_k}{k},
    \qquad
    B_\ell(\Theta_\ell)
    =
    b_k \Theta_\ell H^\T,
    \qquad
    \Theta_\ell \in[-1,1]^{k\times r}.
  \]
  Starting from \(\lambda_0 I_d\), place
  \[
    \begin{pmatrix}
      0&B_\ell(\Theta_\ell)\\
      B_\ell(\Theta_\ell)^\T&0
    \end{pmatrix}
  \]
  on \(L_\ell \cup R_\ell\), independently over \(\ell\), and denote the resulting covariance matrix by \(\Sigma_\Theta\).
  Because
  \[
    \norm{\Theta_\ell}\leq\sqrt{kr},
    \qquad
    \norm{B_\ell(\Theta_\ell)}
    \leq
    b_k \sqrt{kr}\sqrt{\frac{2k}{r}}
    =
    \sqrt2a_k,
  \]
  fixed choices of \(\lambda_0>0\) and sufficiently small \(c_0>0\), depending only on \(\alpha,K,M_0,M\), make the family uniformly positive definite and ensure the covariance and sub-Gaussian bounds.
  Each affected entry is bounded by
  \(b_k=c_0 k^{-(\alpha+1)}\) and lies within distance \(2k\), so a further fixed reduction of \(c_0\) gives pointwise-decay membership.
  Consequently, the centered Gaussian laws
  \(N(0,\Sigma_\Theta)\) belong to
  \(\mca{P}_\alpha^{\mca{H}}\).

  Give the \(m=Nkr\asymp dr\) coordinates of
  \(\Theta=(\Theta_\ell)_{\ell\leq N}\) independent density
  \[
    h(u)
    =
    \cos^2 \left(\frac{\pi u}{2}\right)
    \bm{1}\{\abs{u}\leq1\}.
  \]
  The scalar Fisher information of \(h\) is \(\pi^2\), and hence
  \(\Tr J_\pi=m\pi^2 \lesssim dr\).
  For a parameter coordinate \(u=(\ell,i,s)\), the covariance derivative is
  \[
    D_u
    =
    b_k
    \begin{pmatrix}
      0&e_i h_s^\T\\
      h_s e_i^\T&0
    \end{pmatrix},
  \]
  where \(h_s\) is the incidence vector of the \(s\)-th group.
  These derivatives are pairwise orthogonal in the Frobenius inner product and satisfy
  \[
    \norm{D_u}_F^2
    \asymp
    \frac{a_k^2}{kr}.
  \]
  For one centered Gaussian observation,
  \[
    I_x(\Theta)_{uv}
    =
    \frac{1}{2}
    \Tr\left(
      \Sigma_\Theta^{-1} D_u
      \Sigma_\Theta^{-1} D_v
    \right).
  \]
  For \(v\in\R^m\), write \(D(v)=\sum_u v_u D_u\).
  Uniform conditioning and Frobenius orthogonality give
  \[
    \begin{aligned}
      v^\T I_x(\Theta)v
      &=
      \frac{1}{2}
      \norm{\Sigma_\Theta^{-1/2} D(v)\Sigma_\Theta^{-1/2}}_F^2\\
      &\lesssim
      \norm{D(v)}_F^2
      \lesssim
      \frac{a_k^2}{kr}\norm{v}_2^2.
    \end{aligned}
  \]
  Hence
  \[
    \sup_\Theta \norm{I_x(\Theta)}
    \lesssim
    \frac{a_k^2}{kr}.
  \]
  The decoder defined below is deterministic post-processing of
  \(\widehat{\Sigma}\), so it remains \(\rho\)-zCDP.
  Under the standing condition \(\rho\lesssim1\), applying
  \cref{thm:VanTrees__zCDP_Restated} and bounding its information term by the privacy branch gives
  \begin{equation}
    \E_\pi \E_\Theta
    \norm{\widehat{\Theta}-\Theta}_F^2
    \gtrsim
    \frac{m^2}
    {\rho n^2 a_k^2/(kr)+m}.
    \label{eq:Proof_Lower__GH_Coordinate}
  \end{equation}
  The cosine prior and the uniformly conditioned Gaussian family satisfy the regularity conditions of that theorem.

  It remains to convert coordinate loss into operator loss for an arbitrary matrix estimator.
  Let \(P_H=H(H^\T H)^{-1}H^\T\) be the orthogonal projector onto the column space of \(H\).
  Starting from
  \(\operatorname{sym}(\widehat{\Sigma})-\lambda_0 I_d\), apply block-diagonal pinching to the \(N\) disjoint superblocks and off-diagonal pinching within each superblock.
  On every remaining symmetric block apply
  \[
    \begin{pmatrix}0&C\\C^\T&0\end{pmatrix}
    \longmapsto
    \begin{pmatrix}0&CP_H\\P_H C^\T&0\end{pmatrix}.
  \]
  Symmetrization and both pinching maps are operator contractions.
  The operator norm of either displayed symmetric block equals the norm of its upper-right block, so the last map is contractive because
  \(\norm{CP_H}\leq\norm{C}\).
  The true perturbation
  \(\Sigma_\Theta-\lambda_0 I_d\) is fixed by all these maps.
  Thus, if \(E^{\mathrm{pr}}\) is the projected covariance error,
  \(\norm{\widehat{\Sigma}-\Sigma_\Theta}\geq\norm{E^{\mathrm{pr}}}\).
  Let \(\widehat{\Delta}_{L_\ell,R_\ell}\) denote the resulting projected rectangular block, and decode it by
  \[
    \widehat{\Theta}_\ell
    =
    b_k^{-1}
    \widehat{\Delta}_{L_\ell,R_\ell}
    H(H^\T H)^{-1}.
  \]
  For this projected error,
  \[
    \norm{E^{\mathrm{pr}}}_F^2
    \asymp
    \frac{a_k^2}{kr}
    \norm{\widehat{\Theta}-\Theta}_F^2.
  \]
  Moreover,
  \(\operatorname{rank}(E^{\mathrm{pr}})\leq2Nr\), so \(Nk\asymp d\) yields
  \[
    \norm{E^{\mathrm{pr}}}^2
    \geq
    \frac{\norm{E^{\mathrm{pr}}}_F^2}{2Nr}
    \gtrsim
    \frac{a_k^2}{dr^2}
    \norm{\widehat{\Theta}-\Theta}_F^2.
  \]
  Combining this inequality with
  \cref{eq:Proof_Lower__GH_Coordinate},
  \(m\asymp dr\), and \(a_k^2 \asymp k^{-2\alpha}\) gives
  \[
    \E_\pi \E_\Theta
    \norm{\widehat{\Sigma}-\Sigma_\Theta}^2
    \gtrsim
    \frac{a_k^2 d}
    {\rho n^2 a_k^2/(kr)+dr}
    \asymp
    \min\xk{
      \frac{a_k^2}{r},
      \frac{dkr}{\rho n^2}
    }.
  \]
  This proves \cref{eq:Lower__GH_LowRankGaussian}.
\end{proof}

\begin{proof}[Proof of \cref{thm:Lower__GH_Operator}]
  We first prove the result for \(\mca{H}_\alpha\).
  The classical centered-Gaussian testing family can be chosen in the pointwise-decay class and rescaled to meet the model bounds.
  Since \(\mca{M}_\rho\) is a subset of all estimators, this gives
  \[
    \inf_{\widehat{\Sigma}\in\mca{M}_\rho}
    \sup_{P\in\mca{P}_\alpha^{\mca{H}}}
    \E_P \norm{\widehat{\Sigma}-\Sigma(P)}^2
    \gtrsim
    n^{-\frac{2\alpha}{2\alpha+1}}
    \wedge\frac{d}{n};
  \]
  see \citet[Theorem~1 and the discussion following Equation~(30)]{cai2010_OptimalRates}.
  On the other hand,
  \cref{prop:Lower__GH_LowRankGaussian} holds for every pair of integers
  \(1\leq r\leq k\leq\lfloor d/4\rfloor\), with a constant uniform in \(k,r\).
  Taking the supremum over the admissible \(k,r\) and using
  \(d/(\rho n^2)\) in \cref{eq:Rate_Q} gives exactly
  \[
    \mca{Q}_{\alpha,d} \left(\frac{d}{\rho n^2}\right).
  \]
  The maximum of the statistical and privacy lower bounds is at least half of their sum, proving the asserted rate for \(\mca{H}_\alpha\).

  Applying the class-inclusion transfer at the beginning of this subsection gives the same lower bound for \(\mca{G}_\alpha\), completing the proof.
\end{proof}

\begin{remark}[Compact-uniform lower-bound constants]
  \label{remark:Lower__GH_Uniformity}
  On a fixed interval
  \([\underline{\alpha},\overline{\alpha}]\), choose the amplitude prefactors in the classical centered-Gaussian testing family, the Frobenius band prior, and the rank-interpolating prior using the worst endpoint
  \(\overline{\alpha}\), for example
  \(c_0 \leq c_{\mathrm{amp}}2^{-(\overline{\alpha}+1)}\).
  The transfer from the pointwise-decay submodel to the row-tail class is uniform because
  \[
    2\sum_{\ell>k} \ell^{-(\alpha+1)}
    \leq
    2\underline{\alpha}^{-1} k^{-\alpha}.
  \]
  Thus the conditioning, Fisher-information, decoding, integer-rounding constants, and the threshold on \(n\) in
  \cref{thm:Lower__Frobenius} and
  \cref{thm:Lower__GH_Operator}
  can be chosen uniformly on the interval.
\end{remark}

\subsection{Construction of the Matrix Prior}

We prove \cref{lem:Lower__PriorDistribution} by constructing a matrix prior that satisfies two requirements: bounded operator norm on its support and prior Fisher information with trace of order \(d^3\).
We start with a Gaussian matrix distribution, whose Fisher information has the required order, and smoothly truncate its density to enforce the support constraint.
The main task is to ensure that this truncation does not increase the order of the Fisher information.

First, let $W$ have i.i.d.\ normal entries \( w_{ij} \sim N(0, c/d) \) for some small constant $c$.
Classical random matrix theory (\cref{lem:MatrixGaussianConcentration}) shows that as long as $c$ is small enough,
\begin{equation*}
  \bbP \dk{\norm{W} \geq \frac{1}{2}} \leq 0.01.
\end{equation*}
Let $p(W)$ be the density of $W$.
It is easy to see that the trace of Fisher information
\begin{equation*}
  \Tr I_W = \sum_{i,j} \E \zk{\frac{\partial}{\partial w_{ij}} \log p(W)}^2 = \sum_{i,j} \frac{d}{c} = \frac{d^3}{c}.
\end{equation*}
However, we need to enforce the bounded operator norm constraint.

We first introduce the following smooth approximation of the squared operator norm.
For a matrix $X$, we define
\begin{equation*}
  h_{\tau}(X) = \tau \log \Tr \exp(X^\T X/\tau),\quad \tau = \frac{1}{2 \log d}
\end{equation*}
It is easy to see that
\begin{equation}
  \label{eq:Lower__SmoothMax}
  \norm{X}^2 \leq h_{\tau}(X) \leq \norm{X}^2 + \tau \log d = \norm{X}^2 + \frac{1}{2}.
\end{equation}
Moreover, its gradient is given by
\begin{equation*}
  \nabla_X h_{\tau}(X) = 2X \frac{\exp(X^\T X/\tau)}{\Tr \exp(X^\T X/\tau)}.
\end{equation*}

Now, let us define another smoothed indicator function $\varphi(s)$ as
\begin{equation*}
  \varphi(s) =
  \begin{cases}
    1, & s \leq 1,\\
    \cos^2(\frac{\pi}{2}(s-1)), & 1 < s < 2,\\
    0, & s \geq 2.
  \end{cases}
\end{equation*}
Then, $\varphi(s)$ is continuously differentiable with
\begin{equation*}
  \varphi'(s) =
  \begin{cases}
    0, & s \leq 1,\\
    -\pi \sin(\frac{\pi}{2} (s-1))\cos(\frac{\pi}{2}(s-1)), & 1 < s < 2,\\
    0, & s \geq 2.
  \end{cases}
\end{equation*}
Moreover, define the continuous extension $\psi$ of the factor associated with Fisher information by
\begin{equation*}
  \psi(s)
  =
  \begin{cases}
    0, & s \leq 1,\\
    \pi^2 \sin^2(\frac{\pi}{2}(s-1)), & 1 < s < 2,\\
    \pi^2, & s \geq 2.
  \end{cases}
\end{equation*}
Thus $\psi(s)=\xk{\varphi'(s)}^2/\varphi(s)$ for $1<s<2$, and $\psi(2)=\pi^2$ is the limiting value at the boundary; the extension is bounded by $\pi^2$.

Let us now define the prior distribution $\Upsilon$ over matrices $X$ with density
\begin{equation*}
  q(X) = \frac{1}{Z} p(X) \varphi(h_{\tau}(X)),
\end{equation*}
where $Z$ is the normalizing constant.
Then, using \cref{eq:Lower__SmoothMax} and the property of $\varphi(s)$, we notice
\begin{align*}
  q(X) &= \frac{1}{Z} p(X),\quad \text{for } \norm{X}^2 \leq 1/2, \\
  q(X) &= 0, \quad \text{for } \norm{X}^2 \geq 2,
\end{align*}
so it is supported on the operator norm ball \( \norm{X} \leq \sqrt{2} \).
Moreover, regarding the normalizing constant $Z$, we have
\begin{align*}
  Z &= \int p(X) \varphi(h_{\tau}(X)) \dd X \geq \int_{\norm{X}^2 \leq 1/2} p(X) \varphi(h_{\tau}(X)) \dd X \\
  & \geq \int_{\norm{X}^2 \leq 1/2} p(X) \dd X \\
  & = 1 - \bbP \dk{\norm{W}^2 > 1/2} \geq 0.99.
\end{align*}

Let us now compute the trace of Fisher information of this prior distribution.
Expanding the gradient, we have
\begin{equation*}
 \nabla_X \log q(X) = \nabla_X \log p(X) + \nabla_X \log \varphi(h_{\tau}(X))
\end{equation*}
so
\begin{equation*}
  T = \E \norm{\nabla_X \log q(X)}_F^2
  \leq 2 \E \norm{\nabla_X \log p(X)}_F^2 + 2 \E \norm{\nabla_X \log \varphi(h_{\tau}(X))}_F^2.
\end{equation*}
The first term is readily bounded as
\begin{align*}
  \E \norm{\nabla_X \log p(X)}_F^2
  &= \frac{1}{Z} \int \norm{\nabla_X \log p(X)}_F^2 p(X) \varphi(h_{\tau}(X)) \dd X \\
  &\leq \frac{1}{Z} \int \norm{\nabla_X \log p(X)}_F^2 p(X) \dd X \\
  &= \frac{1}{Z} \Tr I_W \lesssim d^3.
\end{align*}

For the second term, on the support of $q$, where $\varphi(h_{\tau}(X))>0$, the gradient expands as
\begin{align*}
  \nabla_X \log \varphi(h_{\tau}(X)) = \frac{\varphi'(h_{\tau}(X)) }{\varphi(h_{\tau}(X))} \nabla_X h_{\tau}(X) = 2 \frac{\varphi'(h_{\tau}(X)) }{\varphi(h_{\tau}(X))} X \frac{\exp(X^\T X/\tau)}{\Tr \exp(X^\T X/\tau)},
\end{align*}
and
\begin{align*}
  \norm{X\exp(X^\T X/\tau)}_F^2 &= \Tr \xk{X \exp(X^\T X/\tau) \exp(X^\T X/\tau) X^\T} \\
  &= \Tr \xk{X^\T X \exp(2 X^\T X/\tau)} \\
  &\leq \norm{X}^2 \Tr \exp(2 X^\T X/\tau).
\end{align*}

Hence,
\begin{equation*}
  \norm{X \frac{\exp(X^\T X/\tau)}{\Tr \exp(X^\T X/\tau)}}_F^2
  \leq \norm{X}^2 \frac{\Tr \exp(2 X^\T X/\tau)}{(\Tr \exp(X^\T X/\tau))^2} \leq \norm{X}^2,
\end{equation*}
and thus
\begin{align*}
  \norm{\nabla_X \log \varphi(h_{\tau}(X))}_F^2 \leq 4 \xk{\frac{\varphi'(h_{\tau}(X)) }{\varphi(h_{\tau}(X))}}^2 \norm{X}^2
  \leq 8 \xk{\frac{\varphi'(h_{\tau}(X)) }{\varphi(h_{\tau}(X))}}^2.
\end{align*}
Plugging this into the expectation, we have
\begin{align*}
  \E \norm{\nabla_X \log \varphi(h_{\tau}(X))}_F^2
  &= \int \norm{\nabla_X \log \varphi(h_{\tau}(X))}_F^2 q(X) \dd X \\
  &\leq \frac{8}{Z} \int \xk{\frac{\varphi'(h_{\tau}(X)) }{\varphi(h_{\tau}(X))}}^2 p(X) \varphi(h_{\tau}(X)) \dd X \\
  &\leq \frac{8}{Z} \int \psi(h_{\tau}(X)) p(X) \dd X \\
  &\leq \frac{8 \pi^2}{Z} \int \bm{1}\dk{h_{\tau}(X) \geq 1} p(X) \dd X \\
  & \leq \frac{8 \pi^2}{Z}  \int \bm{1}\dk{\norm{X}^2 \geq \frac{1}{2}} p(X) \dd X \\
  & = \frac{8 \pi^2}{Z} \bbP \dk{\norm{W}^2 \geq \frac{1}{2}} \\
  & \lesssim 1.
\end{align*}
Combining the above two bounds, we obtain \( T \lesssim d^3 \).
After a constant rescaling, we obtain the desired distribution in the lemma.

  \clearpage

\section{Proofs for Precision Matrix Estimation}
\label{sec:ProofPrecision}

The guarantees for precision matrix estimation in \cref{thm:Precision__MinimaxRates} follow from two reductions to covariance estimation.
For the upper bounds, we spectrally truncate the covariance estimator before inversion, which transfers its guarantee under operator loss and preserves privacy by post-processing.
For the lower bounds, we invert a private precision matrix estimator to obtain a private covariance estimator and then invoke the corresponding covariance minimax lower bounds.
Together, these arguments show that precision matrix estimation inherits the covariance rates over the parameter spaces considered in the main paper.

\subsection{Upper Bounds}

Let $\hat{\Sigma}$ be the DP covariance matrix estimator.
Write
\begin{equation*}
  \bar{\Sigma}
  =
  \hat{U}\mr{Diag}\xk{\max\dk{\hat{\lambda}_i,L_2^{-1}}}_{i\in[d]}\hat{U}^{\T},
  \qquad
  \hat{\Omega}=\bar{\Sigma}^{-1},
\end{equation*}
where \(L_2^{-1}\leq c_0\).
Since every eigenvalue of \(\Sigma\) is at least \(c_0\), Weyl's inequality shows that the spectral truncation moves \(\hat{\Sigma}\) by at most \(\norm{\hat{\Sigma}-\Sigma}\).
Hence,
\begin{equation*}
  \norm{\bar{\Sigma}-\Sigma}
  \leq
  2\norm{\hat{\Sigma}-\Sigma}.
\end{equation*}
Using \(\hat{\Omega}-\Omega=\bar{\Sigma}^{-1}(\Sigma-\bar{\Sigma})\Sigma^{-1}\), we obtain
\begin{equation*}
  \norm{\hat{\Omega}-\Omega}
  \leq
  2L_2c_0^{-1}\norm{\hat{\Sigma}-\Sigma}.
\end{equation*}
The estimator \(\hat{\Omega}\) is deterministic post-processing of \(\hat{\Sigma}\), so it has the same zCDP guarantee.
Applying this bound to the upper-bound part of \cref{thm:Overview_F_Rates} and to \cref{thm:DPCov_GH_Operator} gives the asserted upper bounds for \(\tilde{\mca{P}}_\alpha^{\mca{F}}\), \(\tilde{\mca{P}}_\alpha^{\mca{G}}\), and \(\tilde{\mca{P}}_\alpha^{\mca{H}}\).

\subsection{Minimax Lower Bounds}

We transfer the covariance lower bounds through matrix inversion.
Fix \(\mca{J}\in\{\mca{F},\mca{G},\mca{H}\}\), let \(\hat{\Omega}\) be any \(\rho\)-zCDP precision matrix estimator, and set
\begin{equation*}
  A=\frac{\hat{\Omega}+\hat{\Omega}^{\T}}{2}.
\end{equation*}
If \(A=U\mr{Diag}(\lambda_i)_{i\in[d]}U^{\T}\), define
\begin{equation*}
  \bar{\Omega}
  =
  U\mr{Diag}\xk{
    \min\dk{c_0^{-1},\max\dk{\lambda_i,M_0^{-1}}}
  }_{i\in[d]}U^{\T},
  \qquad
  \bar{\Sigma}=\bar{\Omega}^{-1}.
\end{equation*}
This is deterministic post-processing of \(\hat{\Omega}\), so \(\bar{\Sigma}\) is also \(\rho\)-zCDP.

For every \(P\in\tilde{\mca{P}}_\alpha^{\mca{J}}\), abbreviate \(\Sigma=\Sigma(P)\) and \(\Omega=\Omega(P)\).
The eigenvalues of \(\Omega\) lie in \([M_0^{-1},c_0^{-1}]\).
Since symmetrization is contractive and spectral truncation moves \(A\) by at most \(\norm{A-\Omega}\), we have
\begin{equation*}
  \norm{\bar{\Omega}-\Omega}
  \leq
  2\norm{A-\Omega}
  \leq
  2\norm{\hat{\Omega}-\Omega}.
\end{equation*}
Consequently,
\begin{align*}
  \norm{\bar{\Sigma}-\Sigma}
  & =
  \norm{\bar{\Omega}^{-1}(\Omega-\bar{\Omega})\Sigma} \\
  & \leq
  \norm{\bar{\Omega}^{-1}}
  \norm{\Omega-\bar{\Omega}}
  \norm{\Sigma} \\
  & \leq
  2M_0^2\norm{\hat{\Omega}-\Omega}.
\end{align*}
It follows that
\begin{equation}
  \label{eq:Precision_Lower_Reduction}
  \inf_{\hat{\Omega}\in\mca{M}_\rho}
  \sup_{P\in\tilde{\mca{P}}_\alpha^{\mca{J}}}
  \E_P\norm{\hat{\Omega}-\Omega(P)}^2
  \geq
  \frac{1}{4M_0^4}
  \inf_{\hat{\Sigma}\in\mca{M}_\rho}
  \sup_{P\in\tilde{\mca{P}}_\alpha^{\mca{J}}}
  \E_P\norm{\hat{\Sigma}-\Sigma(P)}^2.
\end{equation}

The Gaussian submodels used in the covariance lower bounds can be centered at a fixed \(\lambda_0I_d\) with \(c_0<\lambda_0<M_0\), and their perturbation amplitudes can be reduced by a fixed constant.
They are therefore uniformly well-conditioned, while their rates change only by constants.
Applying \cref{thm:Lower__Operator} proves the lower bound in \cref{thm:Precision__MinimaxRates}, while applying \cref{thm:Lower__GH_Operator} proves the analogous conclusions for \(\tilde{\mca{P}}_\alpha^{\mca{G}}\) and \(\tilde{\mca{P}}_\alpha^{\mca{H}}\).

  \clearpage

\section{Auxiliary Results}

This section collects the technical tools shared by several arguments for the upper and lower bounds.

\subsection{Covariance Class Inclusions}

\begin{proof}[Proof of \cref{prop:ClassInclusions}]
  If \(\Sigma\in\mca{H}_\alpha(M_0,M)\), then for every \(j\in[d]\) and \(k\geq1\),
  \[
    \sum_{i:\,\abs{i-j}>k}\abs{\Sigma_{ij}}
    \leq
    2M\sum_{r>k} r^{-(\alpha+1)}
    \lesssim_\alpha
    Mk^{-\alpha}.
  \]
  This proves \(\mca{H}_\alpha(M_0,M)\subseteq\mca{G}_\alpha(M_0,C_\alpha M)\).

  Next, let \(\Sigma\in\mca{G}_\alpha(M_0,M)\) and let \(R=I\times J\) be a \(k\)-off-diagonal block.
  For \(k\geq2\), the row-tail condition and symmetry bound every row and column sum of \(\Sigma_R\) by \(M(k-1)^{-\alpha}\lesssim_\alpha Mk^{-\alpha}\).
  When \(k=1\), split off the entries at distance one, which satisfy \(\abs{\Sigma_{ij}}\leq\sqrt{\Sigma_{ii} \Sigma_{jj}}\leq M_0\), and apply the row-tail condition to the remaining entries.
  Thus,
  \[
    \norm{\Sigma_R}_{\ell^1} \vee\norm{\Sigma_R}_{\ell^\infty}
    \lesssim_\alpha
    (M+M_0)k^{-\alpha}.
  \]
  Hence
  \[
    \norm{\Sigma_R}^2
    \leq
    \norm{\Sigma_R}_{\ell^1} \norm{\Sigma_R}_{\ell^\infty}
    \lesssim_\alpha
    (M+M_0)^2k^{-2\alpha},
  \]
  proving \(\mca{G}_\alpha(M_0,M)\subseteq\mca{F}_\alpha(M_0,C_\alpha(M+M_0))\).

  For the first reverse inclusion, let \(\Sigma\in\mca{F}_{\alpha+1/2}(M_0,M)\), fix \(j\in[d]\) and \(k\geq1\), and partition each side of the row tail into the dyadic shells
  \[
    I_m^\pm
    =
    \left\{i:\,2^m k<\pm(i-j)\leq2^{m+1} k\right\},
    \qquad m\geq0.
  \]
  Since \(\{j\}\times I_m^+\) and \(I_m^-\times\{j\}\) are \(2^m k\)-off-diagonal blocks, the definition of \(\mca{F}_{\alpha+1/2}\) and Cauchy--Schwarz yield
  \[
    \sum_{i\in I_m^\pm} \abs{\Sigma_{ij}}
    \leq
    \abs{I_m^\pm}^{1/2}
    \left(\sum_{i\in I_m^\pm} \abs{\Sigma_{ij}}^2 \right)^{1/2}
    \lesssim
    M(2^m k)^{-\alpha}.
  \]
  Summing over \(m\) and both sides gives
  \[
    \max_{j\in[d]}
    \sum_{i:\,\abs{i-j}>k}\abs{\Sigma_{ij}}
    \lesssim_\alpha
    Mk^{-\alpha},
  \]
  and therefore \(\mca{F}_{\alpha+1/2}(M_0,M)\subseteq\mca{G}_{\alpha}(M_0,C_\alpha M)\).

  Finally, let \(\Sigma\in\mca{F}_{\alpha+1}(M_0,M)\). For \(i\neq j\), take the singleton block \(R=\{i\}\times\{j\}\) and \(k=\abs{i-j}\).
  Then
  \[
    \abs{\Sigma_{ij}}
    =
    \norm{\Sigma_R}
    \leq
    M\abs{i-j}^{-(\alpha+1)},
  \]
  which proves \(\mca{F}_{\alpha+1}(M_0,M)\subseteq\mca{H}_{\alpha}(M_0,M)\).
\end{proof}

\subsection{Divergence}

The $\chi^2$ divergence between two probability measures $P$ and $Q$ is defined as
\begin{align*}
  \chi^2(P\|Q) = \E_{X \sim Q} \zk{\frac{\dd P}{\dd Q}(X)- 1}^2.
\end{align*}
The order-2 R\'enyi divergence is related to the $\chi^2$ divergence by~\citep{erven2014_RenyiDivergence}:
\begin{align}
  \label{eq:Divergence__RenyiChi2}
  D_{2}(P\|Q) = \ln(1 + \chi^2(P\|Q)).
\end{align}
We have the following fact, which bounds the difference of means via the $\chi^2$ divergence.
\begin{lemma}
  \label{lem:MeanBoundViaChi2}
  Let $X,Y$ be two random vectors taking values in $\R^d$.
  Then,
  \begin{align*}
    \norm{\E X - \E Y}^2 \leq \chi^2(X\|Y) \E \norm{Y}^2.
  \end{align*}
\end{lemma}

\subsection{Probability and Geometric Tools}

\begin{lemma}[Sub-Gaussian products]
  \label{lem:SubGaussianProduct}
  If \(U,V\) are sub-Gaussian, not necessarily independent, then
  \[
    \norm{UV}_{\psi_1}
    \leq
    C\norm{U}_{\psi_2}\norm{V}_{\psi_2}.
  \]
  If their sub-Gaussian norms are bounded by fixed constants, then, for every
  \(L\geq1\),
  \[
    \abs{\E\{UV-\chi_L(UV)\}}
    \leq
    C\exp(-cL).
  \]
\end{lemma}
The first bound is Lemma 2.7.7 of
\citet{vershynin2018_HighdimensionalProbability}; the second follows from
the sub-exponential tail characterization in Proposition 2.7.1 of the same
reference.

\begin{lemma}[Euclidean sphere nets]
  \label{lem:SphereNet}
  The Euclidean unit sphere in \(\R^s\) has a \(1/4\)-net
  \(\mca{N}_s\) with \(\abs{\mca{N}_s}\leq9^s\), and every
  \(A\in\R^{s\times s}\) satisfies
  \[
    \norm{A}
    \leq
    2\max_{u,v\in\mca{N}_s}\abs{u^\T Av}.
  \]
\end{lemma}
See Corollary 4.2.13 and Exercise 4.4.3 of
\citet{vershynin2018_HighdimensionalProbability}.

\subsection{Privacy Tools}

\begin{lemma}[Finite exponential mechanism]
  \label{lem:FiniteExponentialMechanism}
  Let \(\mca{A}\) be finite and suppose every score \(q(D,a)\),
  \(a\in\mca{A}\), has replacement sensitivity at most \(\Delta\).
  The exponential mechanism with privacy parameter \(\epsilon\) is
  \(\epsilon\)-DP and hence \(\epsilon^2/2\)-zCDP.
  Conditionally on the database, with probability at least \(1-\eta\), its
  output \(\widehat{a}\) satisfies
  \[
    \max_{a\in\mca{A}}q(D,a)-q(D,\widehat{a})
    \leq
    \frac{2\Delta}{\epsilon}
    \left\{
      \log\abs{\mca{A}}+\log(1/\eta)
    \right\}.
  \]
\end{lemma}
The privacy and utility statements are Theorem 3.10 and Corollary 3.12 of
\citet{dwork2014_AlgorithmicFoundations}; the zCDP conversion is
Proposition 3.3 of \citet{bun2016_ConcentratedDifferential}.

\subsection{Concentration Inequalities}

\begin{lemma}
  \label{lem:StandardCovarianceEst}
  Let $X$ be a sub-Gaussian random vector in $\R^p$ with \( \norm{X}_{\psi_2} \leq K \) and \( X_1,\dots,X_n \) be i.i.d.\ copies of \( X \).
  Denote by $\hat{V}_n = \frac{1}{n} \sum_{i=1}^n X_i X_i^\T$ and $V = \E X X^\T$.
  Then, for any $t \geq 0$, with probability at least $1- 2 e^{-t}$,
  \begin{align}
    \norm{\hat{V}_n - V} \leq C K^2 \xk{\sqrt {\frac{p+t}{n}} + \frac{p+t}{n}},
  \end{align}
  where $C$ is an absolute constant.
  Also, we have the expectation bound
  \begin{equation}
    \zk{\E \norm{\hat{V}_n - V}^q}^{1/q} \lesssim_q K^2 \xk{\sqrt {\frac{p}{n}} + \frac{p}{n}}, \quad \forall q \geq 1.
  \end{equation}
\end{lemma}

Let us denote by $\mr{GUE}(d)$ the distribution over $d\times d$ symmetric matrices whose upper triangular entries are i.i.d.\ $\mathcal{N}(0,1)$.
We set $M_{ji}=M_{ij}$; in particular, for disjoint index sets $J,J'$, the block $M_{J,J'}$ has independent $\mathcal{N}(0,1)$ entries and $\E\norm{M_{J,J'}}_F^2=\abs{J}\abs{J'}$.

\begin{lemma}
  \label{lem:MatrixGaussianConcentration}
  Let $J,J' \subset [d]$ be two index sets with $\abs{J} \leq k$ and $\abs{J'} \leq k$.
  There exists an absolute constant $C$ such that
  \[
    \bbP_{M \sim \mathrm{GUE}(d)} \dk{\norm{M_{J,J'}} > C (\sqrt {k} + t)} \leq 4 \exp(-t^2).
  \]
  Consequently, we have
  \begin{equation*}
    \E \norm{M_{J,J'}}^q \lesssim k^{q/2},\quad \forall q \geq 1.
  \end{equation*}
\end{lemma}
\begin{proof}
  The first part is Corollary 4.4.8 in \citet{vershynin2018_HighdimensionalProbability}.
  The second statement follows from integrating the tail bound.
\end{proof}

\subsection{Truncation}

\begin{proposition}[Sub-Gaussian norm bound]
  \label{prop:SubGaussianNormBound}
  Let $X$ be a sub-Gaussian random vector in $\R^p$ with $\norm{X}_{\psi_2} \leq K$.
  Then, there are absolute constants $C,c > 0$ such that
  \begin{align*}
    \bbP \dk{\norm{X}_2 \geq C K\sqrt{p} + t} \leq \exp(-ct^2/K^2),\quad \forall t \geq 0.
  \end{align*}
  Consequently, there are absolute constants $C_1, c_1 > 0$ such that
  \begin{equation*}
    \bbP \dk{\norm{X}_2 \geq u} \leq \exp(-c_1 u^2 / K^2),\quad \forall u \geq C_1 K \sqrt{p}.
  \end{equation*}
\end{proposition}

\begin{proposition}[Sub-Gaussian Vector Truncation]
  \label{cor:SubGVectorTruncation}
  Let $X$ be a sub-Gaussian random vector in $\R^p$ with $\norm{X}_{\psi_2} \leq K$.
  For any $C_0 > 0$, there exist absolute constants $B,C > 0$ such that
  \begin{equation}
    \E \norm{X}_2 \ind{\norm{X}_2 \geq B K \sqrt {p}} \leq C K \sqrt {p}\exp(-C_0 p)
  \end{equation}
\end{proposition}
\begin{proof}
  Let us take \( B \geq C_1 \) large enough, where $C_1$ is the constant in \cref{prop:SubGaussianNormBound}.
  We have
  \begin{align*}
    \E \norm{X}_2 \bm{1}\dk{\norm{X}_2 \geq B K \sqrt {p}}
    &= \int_{0}^\infty \bbP \dk{\norm{X}_2 \geq \max(B K \sqrt {p}, u)} \dd u \\
    &= B K \sqrt {p} \cdot \bbP \dk{\norm{X}_2 \geq B K \sqrt {p}} + \int_{B K \sqrt {p}}^\infty \bbP \dk{\norm{X}_2 \geq u} \dd u \\
    & \leq C K \sqrt {p}\exp(-C_0 p) + \int_{B K \sqrt {p}}^\infty \exp(-c u^2/K^2) \dd u \\
    & \leq C K \sqrt {p}\exp(-C_0 p).
  \end{align*}
\end{proof}

\begin{proposition}
  \label{prop:SubGVectorCovTruncation}
  Let $X,Y$ be two sub-Gaussian random vectors in $\R^p$ with $\norm{X}_{\psi_2} \leq K$ and $\norm{Y}_{\psi_2} \leq K$.
  For any $C_0 > 0$, there exist constants $B, C > 0$ such that, with \(R=BK\sqrt p\),
  \begin{equation}
    \norm{\E\chi_R(X)\chi_R(Y)^\T-\E XY^\T}
    \leq
    CK^2 p \exp(-C_0 p).
  \end{equation}
\end{proposition}
\begin{proof}
  Write
  \(\chi_R(X)=a(X)X\) and \(\chi_R(Y)=a(Y)Y\), where
  \(a(z)=1\wedge R/\norm{z}_2\) for \(z\neq0\) and \(a(0)=1\).
  Since
  \[
    \abs{1-a(X)a(Y)}
    \leq
    \ind{\norm{X}\vee\norm{Y}\geq R},
  \]
  it suffices to bound
  \(\E\norm{XY^\T}\ind{\norm{X}\vee\norm{Y}\geq R}\).
  Write this expectation as
  \begin{align*}
    \E \norm{XY^\T} \ind{\norm{X} \vee \norm{Y} \geq R}
    = \int_0^{\infty} F(u) \dd u,
  \end{align*}
  where
  \begin{equation*}
    F(u) \coloneqq\bbP \dk{\norm{XY^\T} \ind{\norm{X} \vee \norm{Y} \geq R} > u}.
  \end{equation*}
  With  \cref{prop:SubGaussianNormBound}, for $t \geq C_1 K \sqrt{p}$, we have
  \begin{equation*}
    \bbP\dk{\norm{X} \vee \norm{Y} \geq t} \leq \bbP \dk{\norm{X} \geq t} + \bbP \dk{\norm{Y} \geq t}
    \leq 2 \exp(-c t^2 / K^2).
  \end{equation*}
  On one hand, we have
  \begin{equation*}
    F(u) \leq \bbP \dk{\norm{X} \vee \norm{Y} \geq B K \sqrt{p}}
    \leq 2 \exp(-cB^2 p) \leq 2 \exp(-C_0 p),
  \end{equation*}
  as long as $B \geq C_1$ is large enough.
  On the other hand, when $u \geq C_1^2 K^2 p$, we have
  \begin{align*}
    F(u)&\leq \bbP \dk{\norm{XY^\T} > u} = \bbP \dk{\norm{X} \norm{Y} > u} \\
    & \leq \bbP \dk{\norm{X} \vee \norm{Y} > \sqrt{u}} \\
    & \leq 2 \exp(-C u / K^2).
  \end{align*}
  Therefore,
  \begin{align*}
    \E \norm{XY^\T} \ind{\norm{X} \vee \norm{Y} \geq R}
    & = \int_0^{B^2 K^2 p} F(u) \dd u + \int_{B^2 K^2 p}^\infty F(u) \dd u \\
    & \leq 2 B^2 K^2 p \exp(-C_0 p) + 2 \int_{B^2 K^2 p}^\infty \exp(-c u / K^2) \dd u \\
    & \leq C K^2 p \exp(-C_0 p).
  \end{align*}
\end{proof}

\subsection{Fisher Information of Normal Distributions}
\label{subsec:FisherNormal}

Let $x \sim N(0,\Sigma)$ be a multivariate Gaussian vector in $\R^p$ with non-singular covariance matrix $\Sigma$, whose log-density is given by
\begin{equation*}
  \log f(x; \Sigma) = -\frac{p}{2} \log(2\pi) - \frac{1}{2} \log \det \Sigma - \frac{1}{2} x^\T \Sigma^{-1} x,
\end{equation*}
and the score function with respect to $\Sigma$ is given by
\begin{equation*}
  s(x;\Sigma) = \frac{1}{2} \xk{\Sigma^{-1} x x^\T \Sigma^{-1} - \Sigma^{-1}}.
\end{equation*}
We view $\Sigma$ as a parameter in the Euclidean space $\bbS^p$ of symmetric matrices, equipped with the Frobenius inner product.
For symmetric directions $H_1,H_2\in\bbS^p$, the Fisher information is the bilinear form
\begin{equation*}
  \mca{I}_\Sigma(H_1,H_2)
  =
  \frac{1}{2}\Tr\zk{\Sigma^{-1}H_1\Sigma^{-1}H_2}.
\end{equation*}
\begin{proposition}
  \label{prop:Lower__Gaussian_FisherInfoOpNorm}
  Let $x \sim N(0,\Sigma)$, and let $I_x(\Sigma)$ be the matrix of $\mca{I}_\Sigma$ in any orthonormal basis of $\bbS^p$ with respect to the Frobenius inner product.
  Then,
  \begin{equation*}
    \Tr I_x(\Sigma) =  \frac{1}{4} \zk{\Tr(\Sigma^{-2}) + \xk{\Tr \Sigma^{-1}}^2}, \quad
    \norm{I_{x}(\Sigma)} = \frac{1}{2} \norm{\Sigma^{-1}}^2.
  \end{equation*}
\end{proposition}
\begin{proof}

  Parseval's identity on $\bbS^p$ gives
  \begin{equation}
    \label{eq:Proof_Lower__Gaussian_FisherInfoOpNorm_1}
    \Tr I_x(\Sigma) = \E \norm{s(x;\Sigma)}_F^2,\quad
    \norm{I_x(\Sigma)} = \sup_{B\in\bbS^p,\,\norm{B}_F \leq 1} \E \ang{s(x;\Sigma), B}^2.
  \end{equation}

  For the trace, plugging in the expression of the score function and setting $A = \Sigma^{-1}$ yield
  \begin{align*}
    \Tr I_x(\Sigma) &= \frac{1}{4} \E \norm{A x x^\T A - A}_F^2 = \frac{1}{4} \E \Tr \xk{Ax x^\T A^2 x x^\T A - 2 A^2 x x^\T A + A^2 } \\
    &=\frac{1}{4} \zk{\E \Tr \xk{x^\T A^2 x x^\T A^2 x} - 2 \Tr(A^2 \E x x^\T A) + \Tr(A^2)} \\
    &= \frac{1}{4} \zk{\E (x^\T A^2 x)^2 - \Tr(A^2)} \\
    &= \frac{1}{4}\xk{2 \Tr(A^2) + \xk{\Tr(A)}^2 - \Tr(A^2)} = \frac{1}{4} \zk{\Tr(A^2) + \xk{\Tr(A)}^2}.
  \end{align*}

  Now fix $B\in\bbS^p$.
  Since $\E\ang{s(x;\Sigma),B}=0$, it suffices to compute its variance.
  Using the expression of the score function, we have
  \begin{equation*}
    \ang{s(x;\Sigma), B} = \frac{1}{2}  x^\T \Sigma^{-1} B \Sigma^{-1} x - \frac{1}{2} \Tr(\Sigma^{-1} B) \eqqcolon \frac{1}{2}(T-C),
  \end{equation*}
  where, writing $x = \Sigma^{1/2} z$ with $z \sim N(0,I_p)$ and
  \begin{equation*}
    N = \Sigma^{-1/2} B \Sigma^{-1/2},
  \end{equation*}
  we have \( T = z^\T N z \). Since \( N \) is symmetric,
  \begin{equation*}
    \E T = \Tr(N) = \Tr(\Sigma^{-1} B),\quad
    \mr{Var}(T) = 2 \Tr\zk{N^2} = 2 \Tr\zk{(\Sigma^{-1} B)^2}.
  \end{equation*}
  Hence,
  \begin{equation}
    \label{eq:Proof_Lower__Gaussian_FisherInfoOpNorm_2}
    \E \ang{s(x;\Sigma), B}^2 = \frac{1}{4} \mr{Var}(T) = \frac{1}{2} \Tr\zk{(\Sigma^{-1} B)^2}.
  \end{equation}

  Plugging \cref{eq:Proof_Lower__Gaussian_FisherInfoOpNorm_2} into \cref{eq:Proof_Lower__Gaussian_FisherInfoOpNorm_1} yields the final result:
  \begin{equation*}
    \norm{I_x(\Sigma)} =  \frac{1}{2} \sup_{B\in\bbS^p,\,\norm{B}_F \leq 1} \Tr\zk{(\Sigma^{-1} B)^2} = \frac{1}{2} \norm{\Sigma^{-1}}^2,
  \end{equation*}
  where in the last step we used the fact that for any symmetric matrix $A$,
  \begin{equation*}
    \norm{A}^2 = \sup_{M \text{ symmetric, }\norm{M}_F \leq 1} \Tr(A M A M).
  \end{equation*}
  To see this, we notice that
  \begin{equation*}
    \Tr(A M A M) = \norm{A^{1/2} M A^{1/2}}_F^2 \leq \norm{A}^2 \norm{M}_F^2 \leq \norm{A}^2,
  \end{equation*}
  and the supremum is attained at $M = u u^\T$ with $u$ being the top eigenvector of $A$, since
  \begin{equation*}
    \Tr(A u u^\T A u u^\T) = (u^\T A u)^2 = \norm{A}^2.
  \end{equation*}

\end{proof}

  \clearpage

\section{Numerical Experiments}
\label{sec:Numeric}

Because the multiscale block procedures are primarily theoretical and computationally intensive, the numerical study focuses on the DP blockwise tridiagonal estimator and its adaptive version.
The experiments first illustrate the block structures produced by the two procedures and then examine how their estimation errors vary with the privacy budget, dimension, and sample size.
We also compare the adaptive procedure with blockwise tridiagonal estimators tuned to fixed decay parameters to assess the empirical cost of adaptation.
Figure~\ref{fig:CompareEstimators} first compares the true covariance matrix with the two estimates and highlights their different block structures.

\begin{figure}[h]
  \centering
  \includegraphics[width=0.9\textwidth]{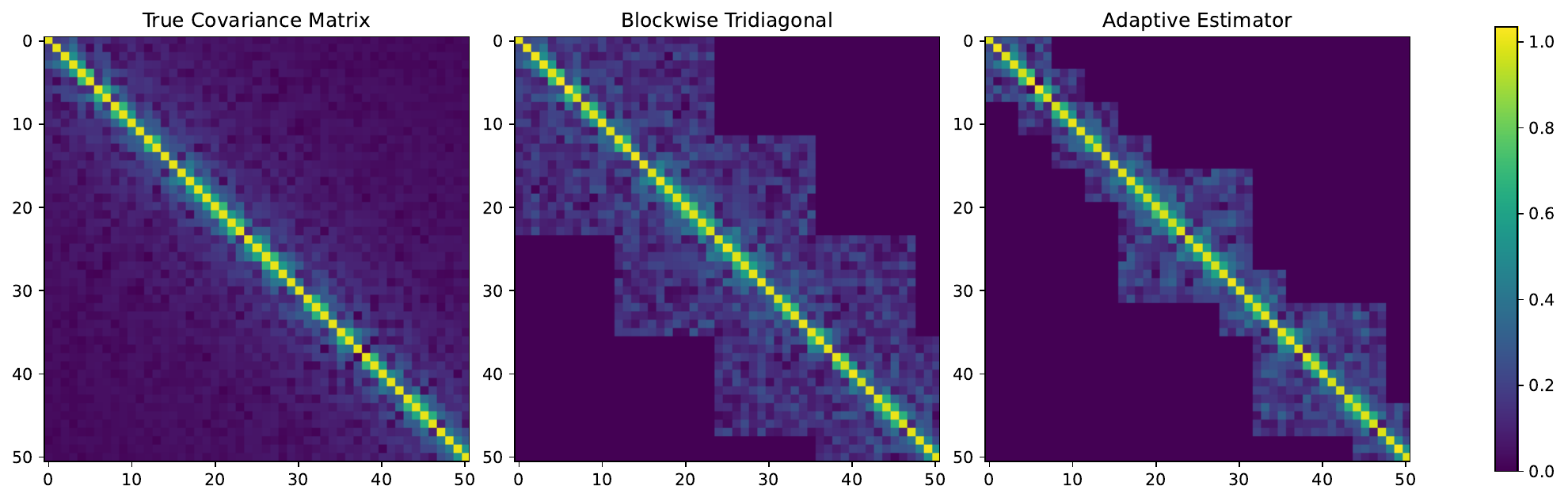}
  \caption{Comparison of the true covariance matrix and the estimators.}
  \label{fig:CompareEstimators}
\end{figure}

In the experiments, the data are generated from a multivariate normal distribution with mean zero.
Unless otherwise stated, each Monte Carlo summary uses 100 repetitions and random seed 1001.
The reported loss is the squared operator norm $\norm{\hat{\Sigma}-\Sigma}^2$; the convergence plots display its logarithm, with error bars given by the empirical standard deviation on the log scale.
As shown in the figure, the estimator with a fixed decay parameter clearly exhibits the blockwise tridiagonal structure, while the adaptive estimator displays a more flexible pattern with varying block sizes.
Moreover, the adaptive estimator tends to be slightly more conservative, reflecting the additional privacy cost incurred by the adaptation procedure.

Next, we investigate how the privacy budget $\rho$ affects the estimation error of the DP blockwise tridiagonal estimator.
We set the true covariance matrix $\Sigma$ to have entries $\Sigma_{ii}=1$ and
$\Sigma_{ij}=0.5|i-j|^{-(\alpha+1)}$ for $i\neq j$, which belongs to the class $\mca{F}_\alpha$,
and choose the decay parameter $\alpha=1$.
The sample size is fixed at $n=500$, while the privacy budget $\rho$ varies from $0.1$ to $10$ and $\infty$.
The values $\rho=10$ and $\rho=\infty$ are included only as reference curves for large privacy budgets and lie outside the primary regime $\rho\lesssim 1$.
To examine the impact of dimensionality, we consider $d=50$ and $d=500$.

The block size $k$ is selected according to the theoretical guideline
\[
k = \Big\lfloor n^{\frac{1}{2\alpha+1}} \wedge 
0.5\left(\frac{\rho n^2}{d}\right)^{\frac{1}{2\alpha+2}} \Big\rfloor,
\]
where the factor $0.5$ is introduced to better balance the error from privacy.
Figure~\ref{fig:Err_Rho} shows that the estimation error decreases as the privacy budget $\rho$ increases, which is consistent with our theoretical results.
In particular, when $\rho$ is sufficiently large, the estimation error approaches that of the estimator without privacy and becomes nearly insensitive to the dimension $d$.
In contrast, when $\rho$ is small and the dimension $d$ is large, the estimation error becomes substantially larger, illustrating the difficulty of private covariance estimation in such regimes.

\begin{figure}[htp]
  \centering

  \begin{minipage}{0.45\textwidth}
    \centering
    \includegraphics[width=1\textwidth]{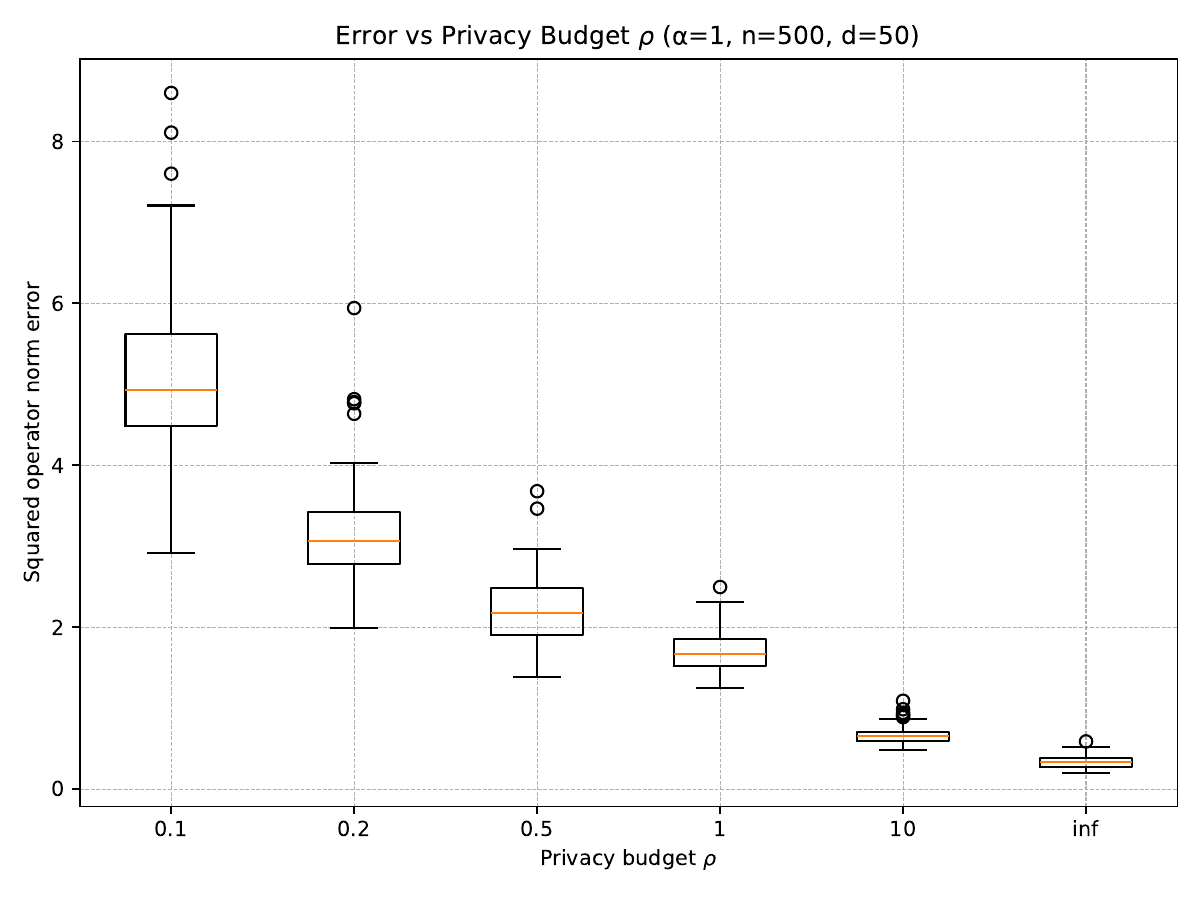}
  \end{minipage}%
  \begin{minipage}{0.45\textwidth}
    \centering
    \includegraphics[width=1\textwidth]{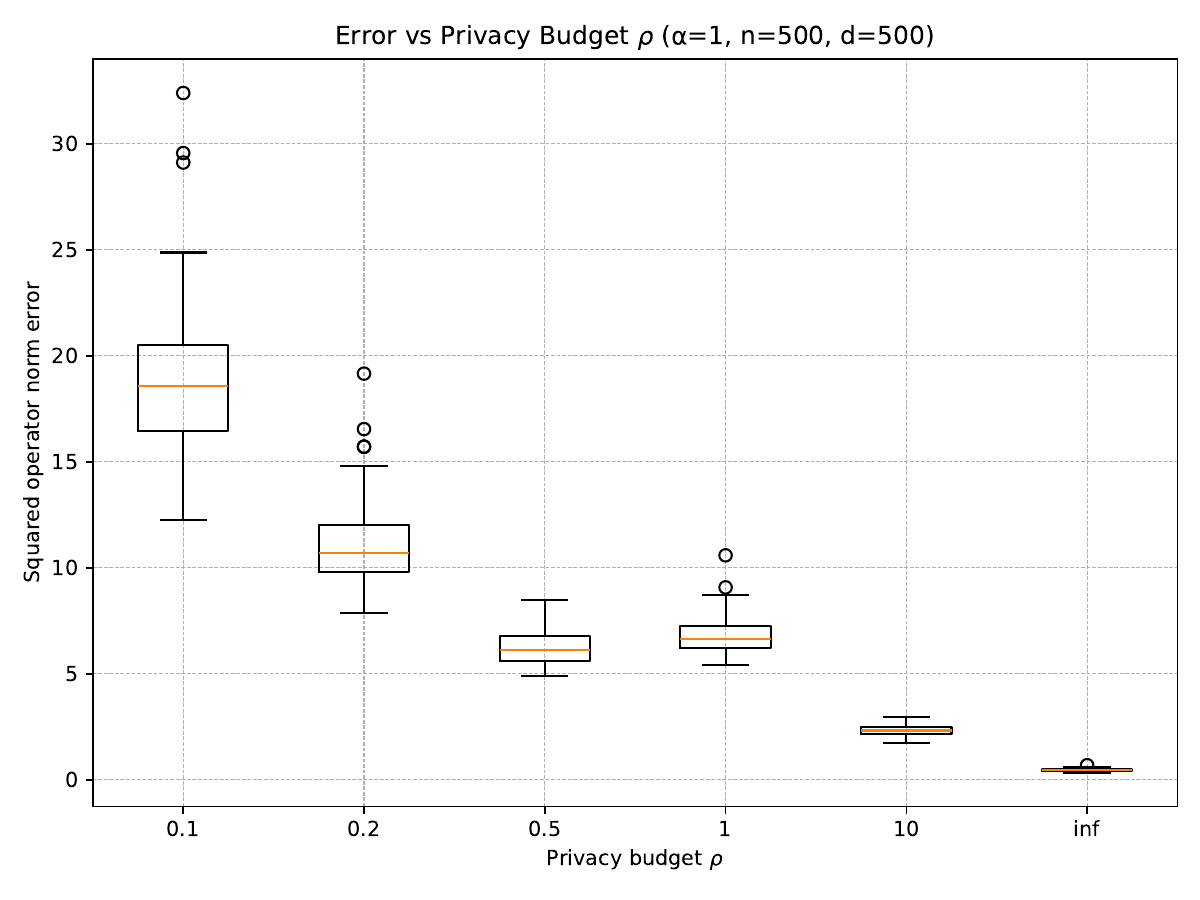}
  \end{minipage}%

  \caption{
    Squared operator norm errors of the DP blockwise tridiagonal estimator under different privacy budgets.
  }
  \label{fig:Err_Rho}
\end{figure}

To further evaluate the performance of the proposed estimators, we study their convergence behavior as the sample size $n$ increases.
Figure~\ref{fig:ConvergenceRates_NonAdaptive} presents plots on logarithmic scales of the estimation error versus the sample size for both the DP blockwise tridiagonal estimator and the adaptive estimator.
To illustrate the transition between statistical error and privacy error, we consider two asymptotic regimes: (left) $d \asymp n^{0.6}$ with constant $\rho$, where the statistical error dominates; and (right) $d \asymp n^{0.7}$ with $\rho \asymp n^{-0.3}$, where the privacy error becomes dominant.

The figures show that both estimators exhibit convergence rates that closely match the theoretical predictions in the two regimes, confirming the effectiveness of the proposed methods.
For instance, in the left panel of Figure~\ref{fig:ConvergenceRates_NonAdaptive}, the theoretical rate  \( n^{-\frac{2\alpha}{2\alpha+1}} = n^{-0.67} \) is consistent with the empirical slope of approximately $-0.67$.
In the right panel, the theoretical rate
\( (\rho n^2/d)^{-\frac{\alpha}{\alpha+1}} \asymp n^{-0.5} \) also agrees well with the observed slope of about $-0.49$.

Finally, comparing the DP blockwise tridiagonal estimator with the adaptive estimator, we observe that the adaptive estimator performs slightly worse and exhibits greater variability.
This is expected, as the adaptive procedure incurs additional cost from adaptation, whereas the estimator with a fixed decay parameter uses the true value to select the optimal block size.

\begin{figure}[htp]
  \centering

  \begin{minipage}{0.45\textwidth}
    \centering
    \includegraphics[width=1\textwidth]{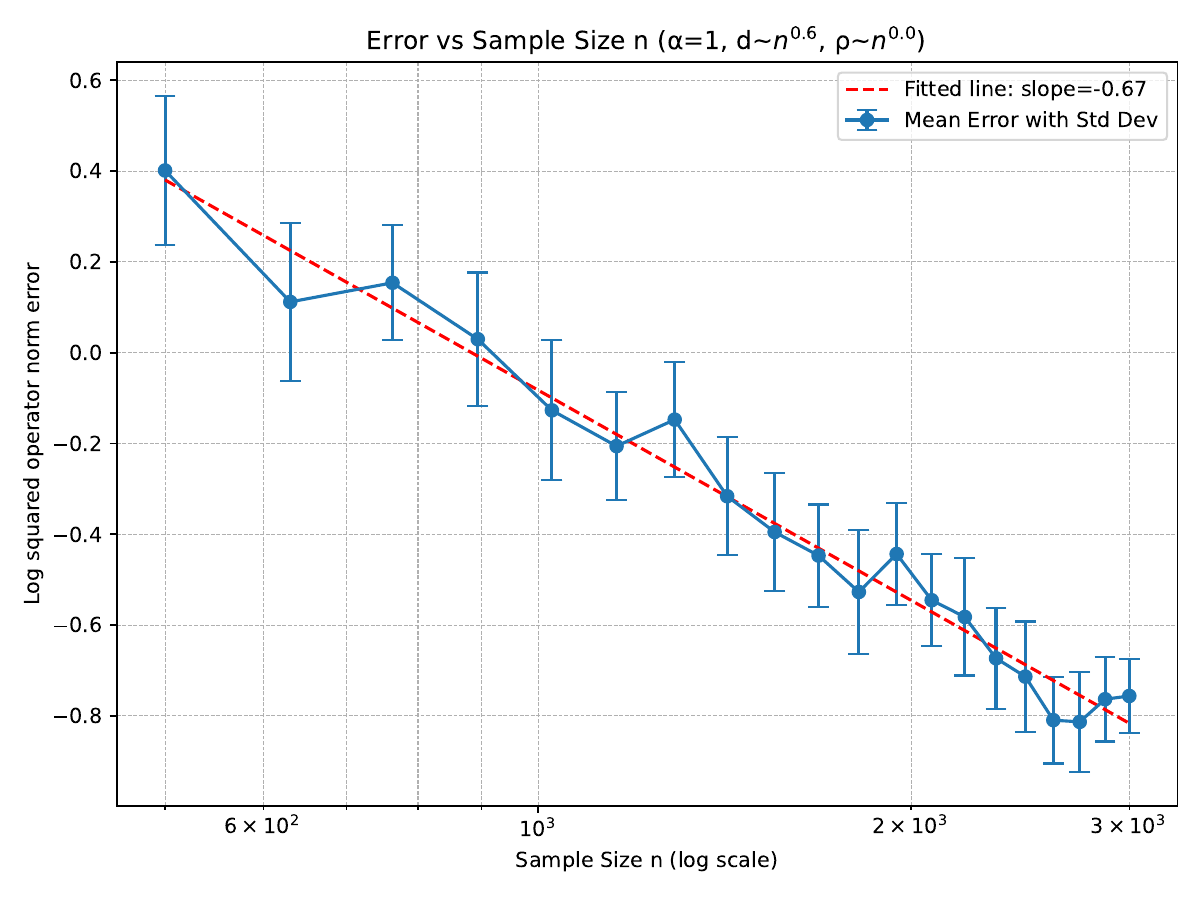}
  \end{minipage}%
  \begin{minipage}{0.45\textwidth}
    \centering
    \includegraphics[width=1\textwidth]{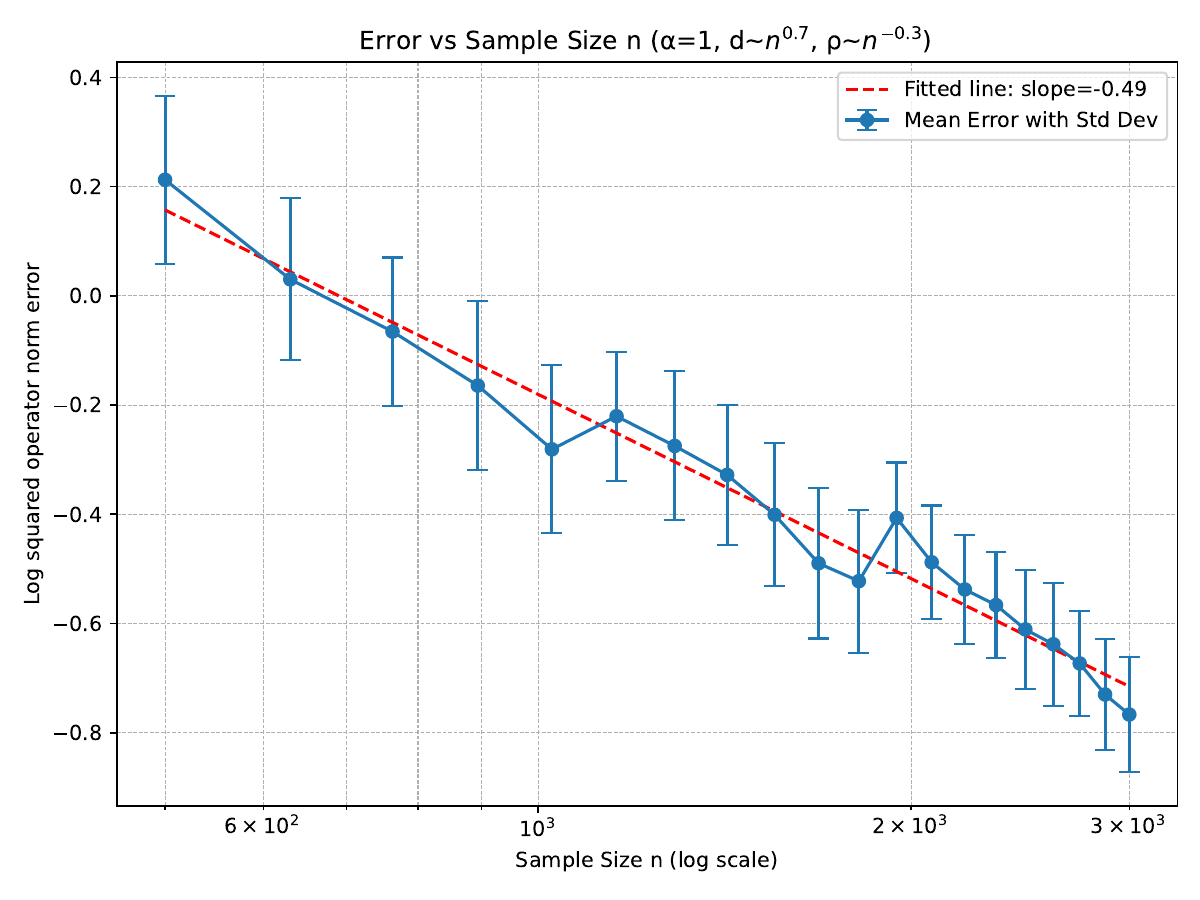}
  \end{minipage}%

  \begin{minipage}{0.45\textwidth}
    \centering
    \includegraphics[width=1\textwidth]{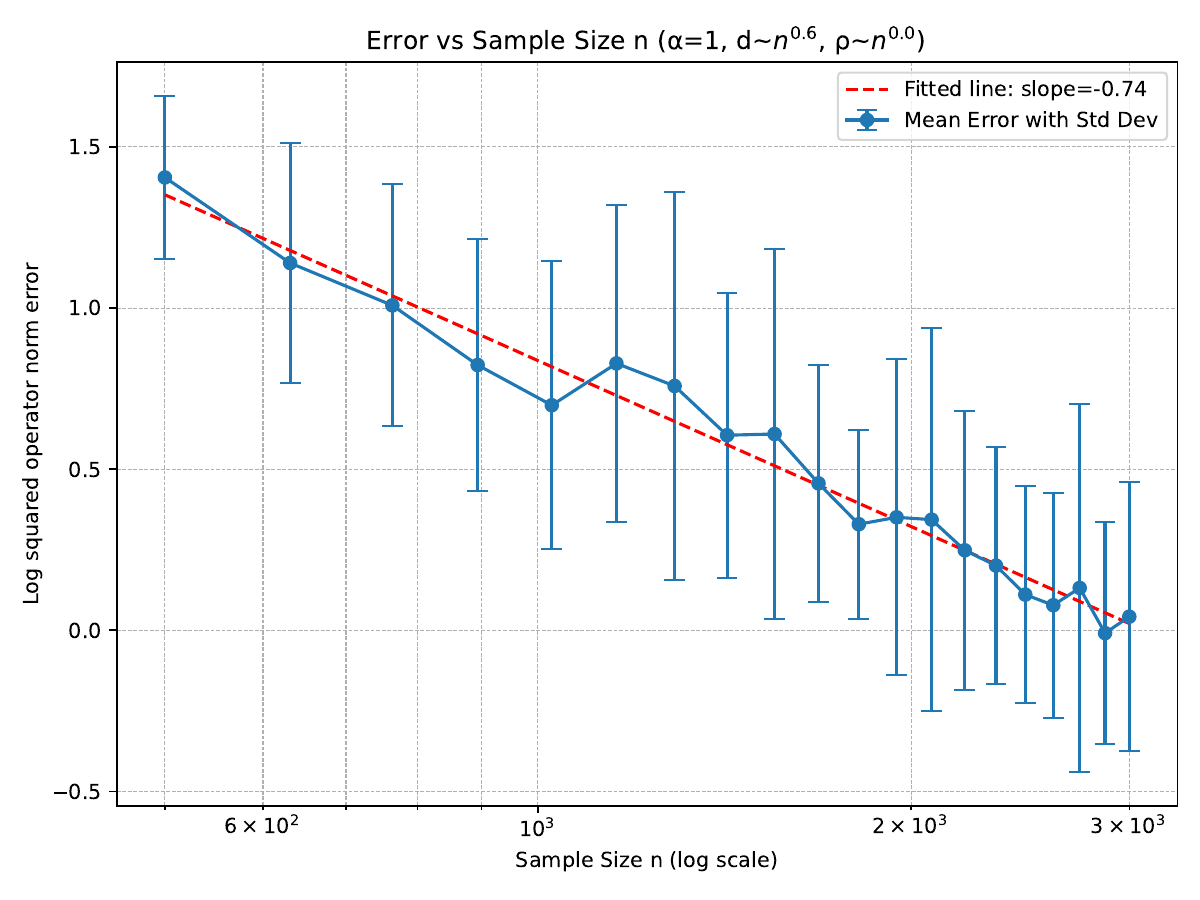}
  \end{minipage}%
  \begin{minipage}{0.45\textwidth}
    \centering
    \includegraphics[width=1\textwidth]{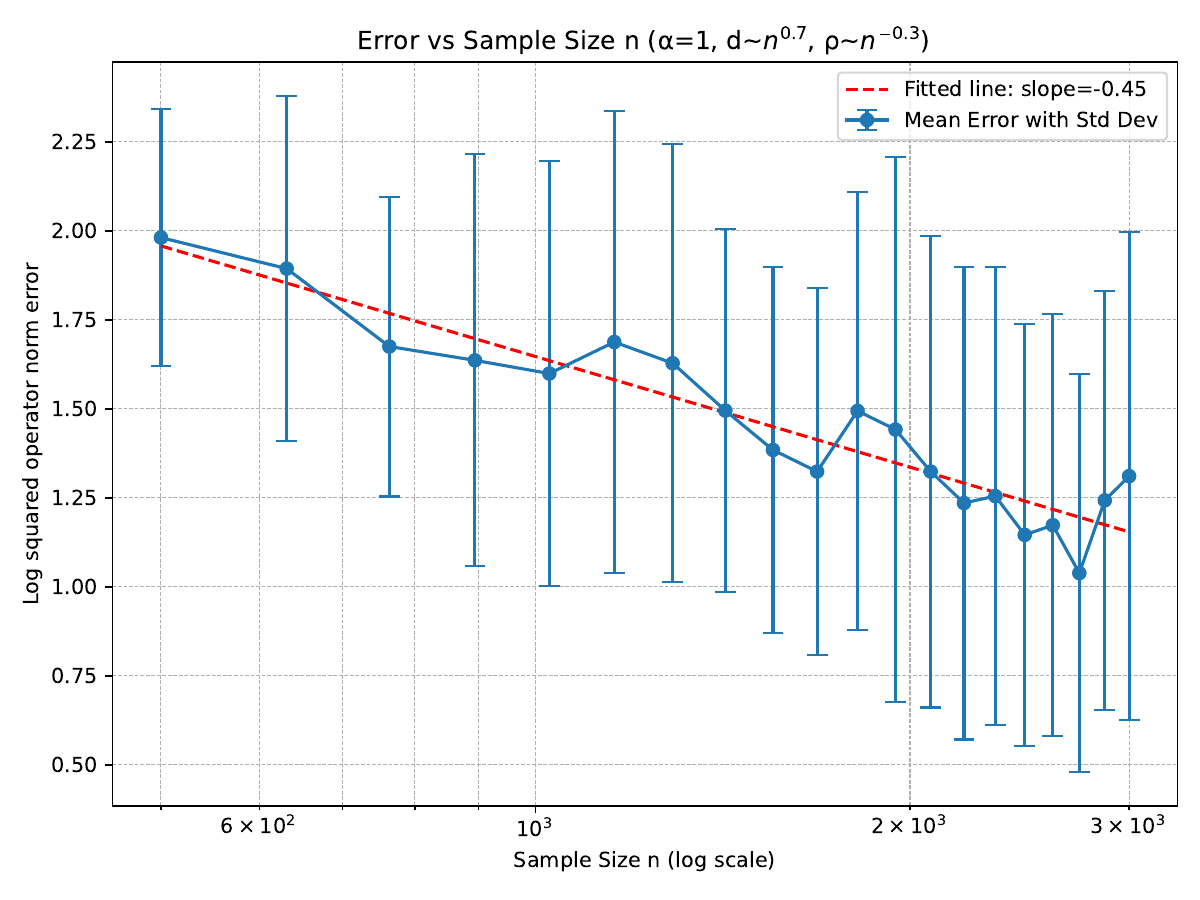}
  \end{minipage}%

  \caption{
    Log squared operator norm errors under two asymptotic regimes; the upper row uses the estimator with the true decay parameter, the lower row the adaptive estimator, and the left and right columns correspond to the regimes dominated by statistical error and privacy error, respectively. Error bars show the empirical standard deviation on the log scale.
  }
  \label{fig:ConvergenceRates_NonAdaptive}
\end{figure}

%

To illustrate the advantage of the adaptive estimator, we compare it with several estimators using fixed decay parameters in Figure~\ref{fig:Compare_Adaptive}.
The covariance matrix is generated as $\Sigma_{ii}=1$ and $\Sigma_{ij}=0.5|i-j|^{-(\alpha+1)}u_{ij}$ for $i>j$, where $u_{ij}$ are i.i.d.\ samples from the uniform distribution on $[0,1]$, with the true decay parameter set to $\alpha=1$.
For these estimators, we use different values of $\alpha \in \{0.25, 0.5, 1, 1.5, 2\}$ to determine the block size.

The results show that the performance of the estimator with a fixed decay parameter depends heavily on the choice of $\alpha$.
When $\alpha$ is correctly specified (i.e., $\alpha=1$), this estimator achieves the best performance.
However, when the block size is poorly tuned due to misspecification of $\alpha$, the estimation error can increase substantially.
In contrast, the adaptive estimator performs reasonably well across all settings without requiring prior knowledge of $\alpha$.

Nevertheless, the adaptive estimator typically exhibits slightly larger errors because it incurs an additional privacy cost associated with the adaptation procedure.
\begin{figure}[htp]
  \centering

  \begin{minipage}{0.45\textwidth}
    \centering
    \includegraphics[width=1\textwidth]{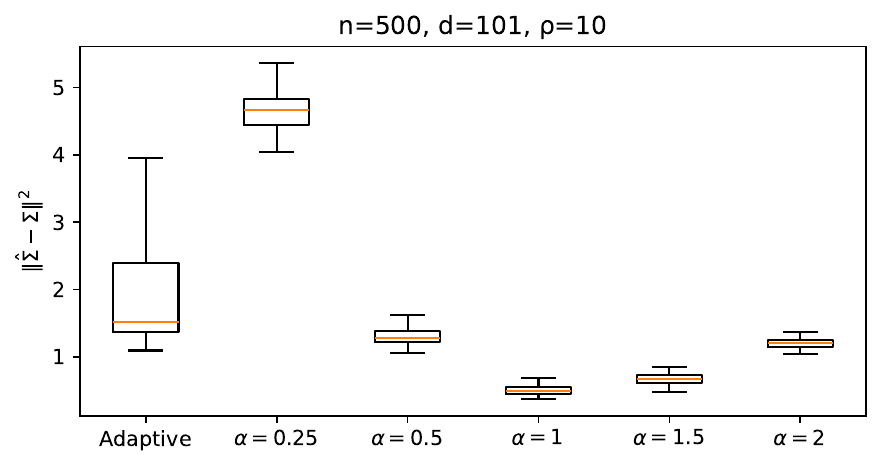}
  \end{minipage}%
  \begin{minipage}{0.45\textwidth}
    \centering
    \includegraphics[width=1\textwidth]{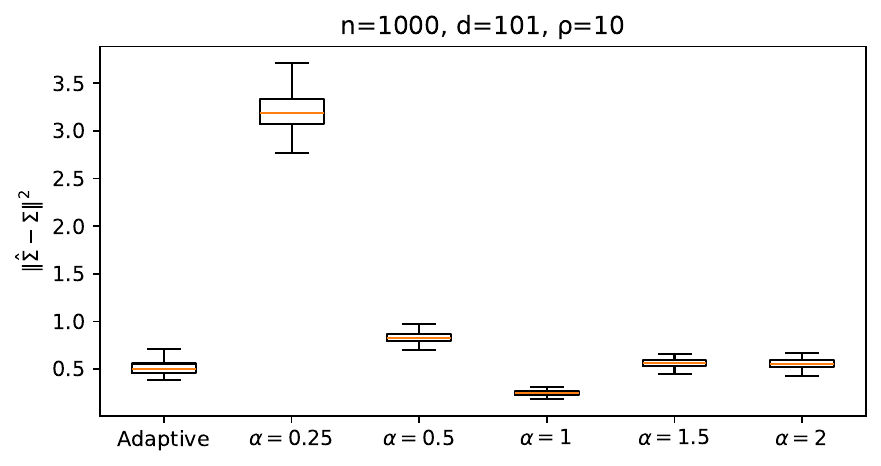}
  \end{minipage}%

  \begin{minipage}{0.45\textwidth}
    \centering
    \includegraphics[width=1\textwidth]{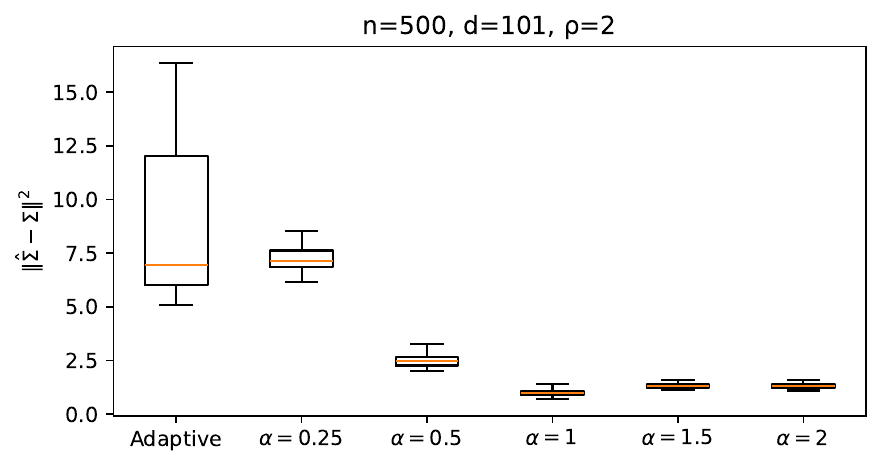}
  \end{minipage}%
  \begin{minipage}{0.45\textwidth}
    \centering
    \includegraphics[width=1\textwidth]{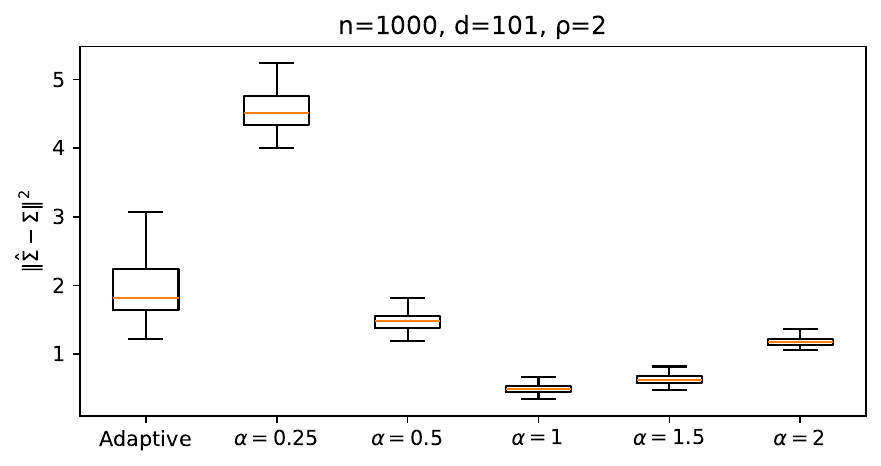}
  \end{minipage}%

  \caption{
    Squared operator norm errors for the adaptive estimator and estimators using fixed decay parameters.
  }
  \label{fig:Compare_Adaptive}
\end{figure}

\section*{Funding}
The research of Tony Cai was supported in part by NSF grant DMS-2413106 and NIH grants R01-GM123056 and R01-GM129781.

\bibliographystyle{plainnat}
\bibliography{main}

\end{document}